\documentclass[a4paper,11pt,leqno]{article}   
\usepackage{amssymb, amsmath, amsthm, graphicx, color, epsfig, 
amsfonts, tikz, url}

\usepackage{thmtools} 
\newtheorem{thm}{Theorem}[section]
\newtheorem{lem}[thm]{Lemma}
\newtheorem{cor}[thm]{Corollary}
\newtheorem{prop}[thm]{Proposition}
\newtheorem{rem}{Remark}[section]

\numberwithin{equation}{section}

\renewcommand{\a}{\alpha}
\renewcommand{\b}{\beta}
\newcommand{\e}{\varepsilon}
\newcommand{\de}{\delta}
\newcommand{\fa}{\varphi}
\newcommand{\ga}{\gamma}
\renewcommand{\k}{\kappa}
\newcommand{\la}{\lambda}
\renewcommand{\th}{\theta}
\newcommand{\si}{\sigma}

\renewcommand{\t}{\tau}
\newcommand{\om}{\omega}
\newcommand{\De}{\Delta}
\newcommand{\Ga}{\Gamma}
\newcommand{\La}{\Lambda}
\newcommand{\Om}{\Omega}
\newcommand{\lan}{\langle}
\newcommand{\ran}{\rangle}

\def\R{{\mathbb{R}}}
\def\N{{\mathbb{N}}}
\def\Z{{\mathbb{Z}}}
\def\T{{\mathbb{T}}}

\def\Tr{\operatorname{Tr}}

\def\lbr{\left(}
\def\rbr{\right)}

\allowdisplaybreaks

\newcommand{\vertiii}[1]{{\left\vert\kern-0.25ex\left\vert\kern-0.25ex\left\vert #1 
    \right\vert\kern-0.25ex\right\vert\kern-0.25ex\right\vert}}

\title{Interface motion from Glauber-Kawasaki dynamics \\ 
of non-gradient type}
\author{Tadahisa Funaki}
\date{\today}

\begin{document}
\maketitle

\begin{abstract}
We consider the Glauber-Kawasaki dynamics on a
$d$-dimensional periodic lattice of size $N$, that is,
a stochastic time evolution of particles performing random walks with interaction
subject to the exclusion rule (Kawasaki part), in general, 
of non-gradient type, together with the effect of the
creation and annihilation of particles
(Glauber part) whose rates are set to favor two levels of particle density, called
sparse and dense.  We then study the limit of our dynamics under 
the hydrodynamic space-time scaling, that is, 
$1/N$ in space and a diffusive scaling $N^2$ for the Kawasaki part and 
another scaling $K=K(N)$, which diverges slower, for the Glauber part in time.  
In the limit as $N\to\infty$, we show that
the particles autonomously make phase separation into sparse or dense
phases at the microscopic level, and an interface separating two regions 
is formed at the macroscopic level and evolves under an anisotropic 
curvature flow. 

In the present article, we show that the particle density at the macroscopic level
is well approximated by a solution of a reaction-diffusion equation with a nonlinear 
diffusion term of divergence form and a large reaction term.
Furthermore, by applying the results of Funaki, Gu and Wang [arXiv:2404.12234]
for the convergence rate of the diffusion matrix approximated by local functions,
we obtain a quantitative hydrodynamic limit as well as the upper bound for the 
allowed diverging speed of $K=K(N)$.

The above result for the derivation of the interface motion
is proved by combining our result with that in a companion paper
by Funaki and Park [arXiv:2403.01732], in which we analyzed the asymptotic 
behavior of the solution of the reaction-diffusion equation obtained in the 
present article and derived
an anisotropic curvature flow in the situation where the macroscopic
reaction term determined from the Glauber part is bistable and balanced.

\footnote{
\hskip -6mm 
Beijing Institute of Mathematical Sciences and Applications, 
No.\ 544 Hefangkou, Huairou District, Beijing 101408, China.
e-mail: funaki@ms.u-tokyo.ac.jp }
\footnote{
\hskip -6mm
Abbreviated title: Interface motion from Glauber-Kawasaki dynamics}
\footnote{
\hskip -6mm
MSC2020: 60K35, 82C22, 74A50.}
\footnote{
\hskip -6mm
Keywords: Kawasaki dynamics, exclusion process, Glauber dynamics, 
hydrodynamic limit, non-gradient model, anisotropic curvature flow, 
interface problem, phase separation.}
\end{abstract}

\section{Introduction --- model and results}

The present article studies the hydrodynamic behavior of the Glauber-Kawasaki
dynamics with a diverging scaling parameter $K$ in the Glauber part.  
The Kawasaki part governs the time evolution of particles moving as 
interacting random
walks subject to the exclusion rule.  It is generally of non-gradient type
and we assume its reversibility under Bernoulli measures.  
The Glauber part prescribes the law of creation and annihilation of particles.
Here we quote the original papers \cite{Gl} by Glauber and \cite{Ka} by 
Kawasaki initially
designed for stochastic dynamics corresponding to the Ising model.
We show that, in particular in the situation
that the particles have two favorable stable phases called `sparse' and `dense' with
different mean densities, they autonomously undergo phase separation into one of
these two phases at the 
microscopic level, and an interface separating the two regions is
formed at the macroscopic level and it evolves under an anisotropic curvature flow;
see Theorem \ref{Theorem 1.3}.

To show this, in the present article, we establish that the particle density
at the macroscopic level is well approximated by the solution of a reaction-diffusion
equation with nonlinear diffusion and a diverging reaction term;
see Theorem \ref{Theorem 1.1}.  The asymptotic behavior of the solution 
of this equation is studied in a separate article \cite{FP} based on a method
in partial differential equations (PDEs); 
see Theorems \ref{thm:gen} and \ref{thm:prop}.
Combining these results, we can complete the derivation of the
interface motion from our particle systems.  Note that the PDE is used only
secondarily.

The present article extends the results obtained in \cite{FvMST}, 
\cite{EFHPS}  for the Glauber-Kawasaki dynamics 
or the Glauber-Zero range process of gradient type to a model
of non-gradient type.  Moreover, applying the results in \cite{FGW},
we obtain a quantitative hydrodynamic limit with its convergence rate in
Theorem \ref{Theorem 1.1}
and also the upper bound for the allowed diverging speed of $K=K(N)$ in
the Glauber part.

\subsection{Model at microscopic level}  \label{Section 1.1}

Let us formulate our model.
We consider the Glauber-Kawasaki dynamics on a $d$-dimensional 
periodic square lattice  $\T_N^d= (\Z/N\Z)^d\equiv \{1,2,\ldots,N\}^d$ of large 
size $N\in\N\equiv \{1,2,\ldots\}$.  The generator  $\mathcal{L}_N$  
of our dynamics is given by the sum of those of Kawasaki and Glauber dynamics with
time change factors $N^2$ and $K=K(N)\ge 1$, respectively:
\begin{equation} \label{eq:LN}
\mathcal{L}_N = N^2 L_E + K L_G.
\end{equation}
The Kawasaki part (also called the exclusion process part)
is the same as in \cite{FUY}, \cite{F96}.  It is of non-gradient 
type and reversible under Bernoulli measures; see \eqref{1.1} and
the conditions (1)--(3) below.  The hydrodynamic scaling limit for
the Kawasaki dynamics reversible under Gibbs measures was studied in \cite{VY}.
In \cite{FvMST}, \cite{FvMST-2}, the so-called gradient 
condition was assumed for the Kawasaki part 
but here we discuss without assuming it.   The large deviation principle
corresponding to \cite{VY} was shown in \cite{BFG}.

To define the operators $L_E$ and $L_G$ precisely, we introduce
several notations.  The configuration space of the dynamics
is $\mathcal{X}_N = \{0,1\}^{\T_N^d}$  whose element is 
denoted by $\eta = \{\eta_x; x \in \T_N^d\}$  where
$\eta_x = 0$ or $1$  indicates that the site  $x$  is vacant or
occupied, respectively.  
We denote  $\mathcal{F}_N$  the set of all functions
on  $\mathcal{X}_N$.  Let  $\tau_x$, $x \in \T_N^d$, be the shift
operators acting on  $\mathcal{X}_N$  by  $(\tau_x\eta)_y = \eta_{y+x}, y \in
\T_N^d,$  where addition is modulo $N$.  They also act on  $\mathcal{F}_N$ by
$(\tau_x f)(\eta) = f(\tau_x\eta), f \in \mathcal{F}_N$.

For  $x, y \in \T_N^d$
and  $\eta\in \mathcal{X}_N$, $\eta^{x,y}$  denotes the element of $\mathcal{X}_N$,
obtained from  $\eta$ by exchanging the values of  $\eta_x$  and  $\eta_y$, 
that is,
\begin{align*}
(\eta^{x,y})_z = \left\{
\begin{aligned}
\eta_y,& \quad \text{ if } z=x, \\
\eta_x,& \quad \text{ if } z=y, \\
\eta_z,& \quad \text{ if } z\not= x, y.
\end{aligned}
\right.
\end{align*}
For  $x \in \T_N^d$
and  $\eta\in \mathcal{X}_N$, $\eta^{x}$  denotes the element of $\mathcal{X}_N$,
obtained from  $\eta$ by flipping the value of  $\eta_x$, that is,
\begin{align*}
(\eta^x)_z = \left\{
\begin{aligned}
1-\eta_x,& \quad \text{ if } z=x, \\
\eta_z,\;\;\;& \quad \text{ if } z\not= x.
\end{aligned}
\right.
\end{align*}

The notations  $\tau_x$, $\eta^{x,y}$  and $\eta^x$, $x,y\in \Z^d$ also indicate
the corresponding ones for  $\mathcal{X} = \{0,1\}^{\Z^d}$, the configuration space 
on the whole lattice $\Z^d$.  For  $\La \subset \T_N^d$ or  $\subset  \Z^d$,
$(\La)^*$ denotes the set of all (undirected) bonds  $b = \{x,y\}$  inside  $\La$,
i.e., $x,y \in \La$  and  $|x-y| = 1$.  Throughout the paper, we use
the norm $|z|$ for $z=(z_i)_{i=1}^d \in \Z^d$ in $\ell^\infty$-sense:
$|z|= \max_{1\le i \le d} |z_i|$. 
We sometimes write  $\eta^b$
instead of  $\eta^{x,y}$  for  bonds  $b = \{x,y\}$.

The generator $L_E$ of the Kawasaki part (exclusion process part)
on $\T_N^d$ is defined as
\begin{equation}
L_E  = \sum_{b \in (\T_N^d)^*} c_b(\eta) \pi_b 
  = \frac12 \sum_{x,y \in\T_N^d: |x-y| = 1} c_{x,y}(\eta) \pi_{x,y},
\label{1.1}
\end{equation}
where   $\pi_b \equiv \pi_{x,y}, b =\{x,y\},$  is the exchange operator 
on  $\mathcal{F}_N$  defined by
$$
\pi_b f(\eta) = f(\eta^b) - f(\eta), \quad f \in \mathcal{F}_N.
$$
The functions  $\{ c_b(\eta) \equiv c_{x,y}(\eta) ; b=\{x,y\} \in
(\Z^d)^* \}$  are defined on  $\mathcal{X}$  and determine the
jump (or exchange) rates of particles between two neighboring sites  
$x$  and  $y$.  We assume that they satisfy the following three conditions  (1)--(3):
\parindent=12mm
\begin{itemize}
\item[(1)] Non-degeneracy and locality: $c_{x,y}(\eta) > 0$  and  it is local,
that is, it depends only on  $\{\eta_z ; |z-x| \le r_c\}$  for some  $r_c>0$.
\item[(2)] Spatial homogeneity:  $c_{x,y} = \tau_x c_{0,y-x}$ for every 
$\{x,y\} \in  (\Z^d)^*$.
\item[(3)] Detailed balance under Bernoulli measures:  $c_{x,y}(\eta)$  
does not depend on $\{\eta_x, \eta_y\}$.
\end{itemize}
\parindent=8mm
In view of (1), the jump rate  $c_b(\eta)$  is naturally regarded as a
function on  $\mathcal{X}_N$  for  $b \in (\T_N^d)^*$, at least if  $N$  is large 
enough such that $N > 2r_c$.
The third condition (3) is equivalent to the symmetricity of  $L_E$
with respect to the Bernoulli measures  $\nu_{\rho}^N, \rho \in [0,1]$,  
that is, the product probability measures on $\mathcal{X}_N$ such that
$\nu_{\rho}^N(\eta_x =1) = \rho$  for every  $x\in \T_N^d$.  We will denote
the Bernoulli measures on $\mathcal{X}$ by $\nu_\rho$, under which the
operator $L_E$ considered on $\Z^d$ is symmetric; see \cite{F18}, \cite{FvMST}.
Note that we do not assume the gradient condition; cf.\ \cite{F18}, \cite{FvMST}.

On the other hand, the generator $L_G$ of the Glauber part is
given by
\begin{equation*}  
L_G  =\sum_{x\in\T_N^d} c_x(\eta) \pi_x,
\end{equation*}
where $\pi_x$, $x\in \T_N^d$, is the flip operator on $\mathcal{F}_N$
defined by
$$
\pi_x f(\eta) = f(  \eta^x )  -f(  \eta), \quad f \in \mathcal{F}_N.
$$
For the flip rates $c_x(\eta)$ defined on $\mathcal{X}$, we assume the 
non-negativity, the locality (with the same range $r_c>0$ as $c_{x,y}$)
and the spatial homogeneity, that is, 
$c_x(\eta) = \t_x c(\eta)$ for some local function $c(\eta) = c_0(\eta) \ge 0$ 
on $\mathcal{X}$ (regarded as that on $\mathcal{X}_N$ for $N$ large enough). 
Since $\eta_0$ takes values only in $\{0,1\}$, $c(\eta)$ can be decomposed as
\begin{equation}  \label{eq:c+-}
c(\eta) = c^+(\eta)(1-\eta_0)+c^-(\eta) \eta_0
\end{equation}
for some local functions $c^\pm(\eta)$ which do not depend on $\eta_0$.
We interpret $c^+(\eta)$ and $c^-(\eta)$ as the rates of creation and annihilation
of a particle at $x=0$, respectively. 

Let  $\eta^N(t) = \{\eta_x^N(t); x \in \T_N^d\}$  be the Markov process
on  $\mathcal{X}_N$  governed by the infinitesimal generator  $\mathcal{L}_N$ in
\eqref{eq:LN} with properly taken $K$ such that 
$1\le K=K(N) \nearrow\infty$ as $N\to\infty$.
We are interested in the asymptotic behavior as  $N \to \infty$ of
its macroscopic empirical mass distribution, that is, 
the measure-valued process defined by
\begin{equation}
\rho^N(t,dv) = N^{-d} \sum_{x\in\T_N^d} \eta_x^N(t) \de_{x/N}(dv),
       \quad   v\in \T^d,   \label{1.2}
\end{equation}
where  $\T^d = \R^d/\Z^d$  is a $d$-dimensional continuous torus identified
with  $[0,1)^d$  and  $\de_v$  is the $\de$-measure at  $v$.

Our main result can be stated as follows.  We consider the situation, under a proper
choice of rates $\{c_{x,y}, c_x\}$, that the particles favor two levels of mean
densities $\rho_+$ and $\rho_-\in (0,1)$ with the same degree of stability.
We then prove that $\rho^N(t,dv)$ converges to $\rho_+dv$ or $\rho_-dv$
on two regions separated by a hypersurface $\Ga_t$ called the interface and $\Ga_t$
evolves under the anisotropic curvature flow; see Theorem \ref{Theorem 1.3}.
We also give the rate of convergence in Theorem \ref{Theorem 1.1}.

The proof is divided into two parts, that is, a probabilistic part 
(Theorem \ref{Theorem 1.1}) which is developed in this article and a PDE part
summarized in Theorems \ref{thm:gen} and \ref{thm:prop}, shown in \cite{FP}.

\subsection{Diffusion matrix and reaction term at macroscopic level}

To state our result for the hydrodynamic limit, which is the main contribution of
this article, first corresponding to the
Kawasaki part, we introduce a quadratic form
$\left(\th,\widehat c(\rho)\th\right),\, \th \in \R^d,$
for each   $\rho \in [0,1]$ called the conductivity via the variational formula
\begin{equation}
\left(\th,\widehat c(\rho)\th\right) = \inf_{F\in \mathcal{F}_0^d}
   \left(\th,\widehat c(\rho;F)\th\right),    \label{1.3}
\end{equation}
where  $(\cdot,\cdot)$  is the inner product of  $\R^d$,
$\mathcal{F}_0$  denotes the class of all 
local functions  on  $\mathcal{X}$, $\mathcal{F}_0^d = (\mathcal{F}_0)^d$,
\begin{equation}  \label{eq:1.6-Q}
\left(\th,\widehat c(\rho;F)\th\right) = \frac12 
  \sum_{|x|=1} \left\lan c_{0,x} \bigg(\th,x(\eta_x-\eta_0) - 
    \pi_{0,x}\Big(\sum_{y \in\Z^d} \tau_y F\Big)\bigg)^2 \right\ran_\rho,  
\end{equation}
and  $\lan\, \cdot\, \ran_\rho\equiv E^{\nu_\rho}[\,\cdot\, ]$ stands for
the expectation with respect to the Bernoulli measure  $\nu_{\rho}$ on 
$\mathcal{X}$.  Note that $\pi_{0,x}(\sum_{y \in\Z^d} \, \tau_y F)$ is well-defined
as a finite sum $\sum_{y \in\Z^d} \,\pi_{0,x}( \tau_y F)$.  
We choose a $d\times d$ symmetric  matrix
$\widehat c(\rho) = \{\widehat c_{ij}(\rho)\}_{1\le i,j \le d}$, 
which is written in the same notation, corresponding to the quadratic
form introduced above, especially to apply results in partial differential equations;
cf.\ {\it Remark} after Theorem 1.1 of \cite{FUY}.
Proposition \ref{prop:5.4}, taken from \cite{FGW}, provides an information
on  a minimizing sequence $\{\Phi_n\}_n$ for $F$ in the variational formula
\eqref{1.3}.

We also introduce the compressibility:
\begin{equation}
\chi(\rho) = \rho - \rho^2,   \label{1.4}
\end{equation}
and the diffusion matrix $D(\rho) = \{D_{ij}(\rho)\}_{1\le i,j \le d}$ by the
Einstein relation:
\begin{equation}
D(\rho) = \frac{\widehat c(\rho)}{2\chi(\rho)}, \quad
  \rho \in [0,1],   \label{1.5}
\end{equation}
see Proposition 2.2 of \cite{10}, p.\ \!180 for the relation to the Green-Kubo formula.
It is known that $D(\rho)$ is a $C^\infty$-function of 
$\rho\in [0,1]$:
\begin{equation} \label{1.D}
D \in C^\infty([0,1]),
\end{equation}
see \cite{Ber}.
Furthermore, the diffusion matrix $D(\rho)$ is uniformly positive and bounded:
\begin{align}  \label{eq:1.9}
c_* |\th|^2 \le (\th,D(\rho)\th) \le c^* |\th|^2, \quad \th \in \R^d, \; \rho\in [0,1],
\end{align}
where $c_*, c^*>0$ are constants defined by
$$
0< c_* := \min_{\eta\in\mathcal{X}, x\in \Z^d: |x|=1} c_{0,x}(\eta) 
\le c^* := \max_{\eta\in\mathcal{X}, x\in \Z^d: |x|=1} c_{0,x}(\eta) < \infty,
$$
which follows from the condition (1); see Lemma \ref{lem:1.1} below
for the proof of \eqref{eq:1.9}.

Next, corresponding to the Glauber part, we introduce a function $f=f(\rho)$
as an ensemble average under $\nu_\rho$ as follows:
\begin{align} \label{eq:f}
  f(\rho) 
&= \lan(1-2\eta_0) c(\eta)\ran_\rho    \\ \notag
& = (1-\rho) \lan c^+\ran_\rho - \rho \lan c^-\ran_\rho.
\end{align}
Note that $f$ is a polynomial of $\rho$ and, in particular, a $C^\infty$-function 
of $\rho\in [0,1]$. Note also that $f(0)>0$ and $f(1)<0$, which is important to
show the comparison theorem; see Section \ref{sec:8}.

\subsection{Quantitative and non-gradient hydrodynamic limit}

We show that, asymptotically as  $N \to \infty$, 
the macroscopic empirical mass distribution  $\rho^N(t,dv)$
is close to the solution $\rho(t,v)=\rho_K(t,v)$  of the following 
reaction-diffusion equation with a nonlinear diffusion term and 
a diverging factor $K=K(N)$ in a reaction term
\begin{equation}
\partial_t \rho(t,v) = \sum_{i,j=1}^d
  \partial_{v_i} \big\{ D_{ij}(\rho(t,v))\partial_{v_j}\rho(t,v) \big\} + K f(\rho(t,v)),
  \quad t\ge 0, \; v\in \T^d,
               \label{1.6}
\end{equation}
where $ v=(v_i)_{i=1}^d$.  
We assume that its initial value  $\rho_0(v)$ satisfies
the condition:
\begin{equation} \label{eq:1.12}
\rho_0\in C^5(\T^d) \quad \text{ and } \quad 0<\rho_0(v)<1.
\end{equation}
Then, the equation \eqref{1.6} has a unique classical solution $\rho(t,v) \in
C^{1,3}([0,\infty)\times \T^d)$ by applying the results in \cite{LSU};
see the beginning of Section \ref{sec:Schauder}.

To state our result, let  $\nu^N\equiv \nu_{1/2}^N$  be the Bernoulli measure 
on $\mathcal{X}_N$ with $\rho=1/2$, that is, the uniform probability measure on  
$\mathcal{X}_N$: $\nu^N(\eta) = 2^{-N}, \eta\in\mathcal{X}_N$.
For two probability densities  $f$  and   $g$  with respect to $\nu^N$, 
define the relative entropy  $H(f|g) \equiv H_N(f|g)$  by
\begin{equation}
H(f|g) = \int_{\mathcal{X}_N} f \log(f/g)\,d\nu^N.   \label{2.1}
\end{equation}
We set
\begin{equation}
\bar{\rho}(\la) = e^\la / (e^\la +1), \quad \la \in \R,  \label{2.3}
\end{equation}
and denote its inverse function by  
$\bar{\la}(\rho)$:  
\begin{equation}  \label{eq:2.4}
\bar{\la}(\rho) = \log \{\rho/(1-\rho)\}, \quad \rho \in (0,1).  
\end{equation}

Let $f_0=f_0(\eta), \eta \in \mathcal{X}_N,$ be the (initial) density of the 
distribution of $\eta^N(0)$ on $\mathcal{X}_N$ with respect to $\nu^N$
and, for $\rho_0=\rho_0(v)$ satisfying \eqref{eq:1.12}, let 
\begin{equation}  \label{eq:1.tildepsi}
\psi_{0,0}(\eta) =Z_N^{-1} \exp\bigg\{\sum_{x\in \T_N^d} 
\bar{\la}(\rho_0(x/N))\eta_x\bigg\}
\end{equation}
with $\bar{\la}$  defined by \eqref{eq:2.4} and a normalization constant $Z_N$ 
with respect to $\nu^N$.  In other words, $\psi_{0,0}d\nu^N$, denoted
also by $P^{\psi_{0,0}}$, is a product measure on $\mathcal{X}_N$ with
a marginal distribution $P^{\psi_{0,0}}(\eta_x=1) = \rho_0(x/N)$ for
every $x\in \T_N^d$.  We use the nation $\psi_{0,0}$ to be consistent with
$\psi_{0,F}$ which denotes the local equilibrium state of the second order
approximation $F$; cf.\ \eqref{eq:2psit} and above Lemma \ref{lem:Add-2}.
In $\psi_{0,0}$, the first `$0$' means `time $0$' and the second `$0$'
is for `$F=0$'.

Then, the difference between $\rho^N(t,dv)$ and $\rho_K(t,v)$ with
$K=K(N)$ is estimated in $L^2(\Om,H^{-\a}(\T^d))$
as in the following theorem.  We call it a quantitative
hydrodynamic limit, since it gives the convergence rate. Note that
$\rho_K$ is also moving in $N$.  To set the upper threshold for the allowed
$K=K(N)$, define
\begin{align}  \label{eq:overK}
\overline{K}(N) \equiv \overline{K}_\de(N):= \de \log N,
\end{align}
for $\de>0$.

\begin{thm} \label{Theorem 1.1}
For each $K\ge 1$, let $\rho_K(t,v)$ be the solution of the equation \eqref{1.6} 
with an initial value $\rho_0(v)$ that satisfies the condition \eqref{eq:1.12}.
We assume that $f_0$ and $\psi_{0,0}$ defined above satisfy 
$H(f_0| \psi_{0,0}) \le C_1N^{d-\k_1}$ for some $C_1>0$ and $\k_1>0$.  
Then, there exists $\frak{c}>0$ small enough such that
for any $T>0$, if the sequence $K=K(N)\to\infty$ (as $N\to\infty$)
satisfies $1\le K(N)\le \overline{K}_\de(N)$ with $\de=\de_T:=\frak{c}/T$,
we have
\begin{align}  \label{eq:1.18-P}
\sup_{t\in [0,T]} E \big[ \|\rho^N(t) - \rho_K(t)\|_{H^{-\a}(\T^d)}^2 \big] \le C N^{-\k}
\end{align}
for every $\a>d/2$, and some $C=C_{T,\a}>0$ and $\k=\k_{\frak{c},\a}>0$, 
where $H^{-\a}(\T^d)$ is the usual Sobolev space on $\T^d$
with the exponent $-\a$; see Section \ref{Section 2.2} for the precise definition.
\end{thm}

Dizdar et al.\ \cite{DMOW} has a similar quantitative result for
one-dimensional Ginzburg-Landau lattice model of gradient type with
$d=1, \a=1, \k=2/3$.  In the above theorem, the range of $\a$ is $\a>d/2$,
which is taken as widely as possible for the measure $\rho^N(t)$ on $\T^d$.
Our $\k>0$
is determined from several factors and small.  The central limit theorem
scaling may suggest $\k \simeq d$.

The constants $\frak{c}>0$ and $\de>0$ in $\overline{K}_\de(N)$ 
are chosen in Theorem \ref{Corollary 2.1}, or in Section \ref{sec:6.2},
for showing the estimate on the relative entropy.  In fact, to make $\k>0$, 
the product $\de\cdot T$ must be small enough and this determines $\frak{c}>0$.  
Once $\frak{c}$ is determined, we choose $\de=\de_T=\frak{c}/T$.  In this sense,
$\k$ depends on $\frak{c}$, but it is independent of $\de$ and $T$.
Then $\k>0$ in \eqref{eq:1.18-P} is chosen
again depending on $\a$ in Section \ref{Section 2.2}.

The assumption for $H(f_0|\psi_{0,0})$ in this theorem is satisfied
if we take $f_0 = \psi_{0,0}$.

In \cite{EFHPS}, \cite{FvMST},  a similar theorem
was shown for models of gradient type; see also \cite{FvMST-2}.
In particular, it was shown that the theorem holds taking
$\overline{K}(N) = \de(\log N)^{\si/2}$ for some small $\de>0$ and some
$\si\in (0,1)$.  The bound for $\overline{K}(N)$ in Theorem \ref{Theorem 1.1}
is better than this.  The reason is that our estimates are more accurate
for several error terms.  Moreover, for our non-gradient model,
new results obtained in \cite{FGW} based on
the method of the quantitative homogenization theory are essential to derive
the above decay rate $CN^{-\k}$ and also to
determine the upper bound $\overline{K}(N)$.

Theorem \ref{Theorem 1.1} reduces the study of the limit of $\rho^N(t,dv)$ as
$N\to\infty$ to that of $\rho_K(t,v)$ as $K\to\infty$.
The latter is a pure PDE problem and is discussed separately in a companion
paper \cite{FP}.  The results are summarized in the next subsection.

\subsection{Interface motion from nonlinear Allen-Cahn equation}
\label{sec:1.4}

We now consider the case that $d\ge 2$ and assume that the reaction term
$f(\rho)$ in \eqref{1.6} is bistable, 
i.e.\ $f$ has three zeros $0<\rho_- < \rho_*<\rho_+<1$ such that 
$f'(\rho_\pm)<0$ and $f'(\rho_*)>0$.
In addition, we assume the balance condition:
$$
\int_{\rho_-}^{\rho_+} f(\rho) D(\rho) d\rho =0,
$$
as matrices.
This means that  two stable phases with densities $\rho_\pm$,
sparse and dense, have the same degree of stability.  
See Section \ref{Section9} for some examples of the jump rate $c_{x,y}(\eta)$
of the Kawasaki  part and the flip rate $c_{x}(\eta)$ of the Glauber part
for which the corresponding $f$ and $D$ satisfy all these conditions.
Recalling the condition \eqref{eq:1.12} for the initial value  $\rho_0(v)$,
we define $\Gamma_0$ by
\begin{align*}
	\Gamma_0 := \{ v \in \T^d: \rho_0(v) = \rho_*\}
\end{align*}
and assume that $\Gamma_0$  is a $C^5$-hypersurface in $\T^d$ without boundary 
such that
\begin{align}  \label{cond:gamma0_normal}
	\nabla \rho_0(v) \cdot n(v) \neq 0 ~\text{ for }~ v \in \Gamma_0, 
\end{align}
where $n(v)$ is the normal vector to $\Ga_0$.
Two regions $\Om_0^\pm$ surrounded by $\Ga_0$ are defined by
$\Om_0^+=\{\rho_0>\rho_*\}$ and $\Om_0^-=\{\rho_0<\rho_*\}$, respectively.

Then, one can show that 
$$
\rho_K(t,v) \to \Xi_{\Ga_t}(v)
\quad \text{as }\;  K\to\infty.
$$
Here, $\Ga_t$ is a smooth closed hypersurface in $\T^d$ and $\Xi_{\Ga_t}$
is a step function taking two values $\rho_\pm$ defined by
\begin{equation*}
\Xi_{\Ga_t}(v) = \left\{
\begin{aligned}
\rho_+,& \quad v \in \Om_t^+, \\
\rho_-,& \quad v \in \Om_t^-,
\end{aligned}\right.
\end{equation*}
where $\Om_t^\pm$ are two regions surrounded by $\Ga_t$.
The sides of these regions are determined initially at $t=0$ and 
then continuously for $t>0$.

The evolution of $\Ga_t$ starting from $\Ga_0$ is governed by the equation
\begin{equation}  \label{eq:d}
V=-{\rm Tr} (\mu(n) \partial_v n)
\equiv -\sum_{i,j=1}^d \mu_{ij}(n) \partial_{v_i}n_j
\quad \text{on }~ \Ga_t,
\end{equation}
where $V$ is the normal velocity of $\Ga_t$ from the side of $\Om_t^-$ to 
$\Om_t^+$
and $ n=(n_i)$ denotes the unit normal vector to $\Ga_t$ of the same direction.
This describes an anisotropic (direction dependent) curvature flow. 
It is a mean curvature flow in the special case that $\mu_{ij}(e) = \mu \de_{ij}$
with a constant $\mu$.

The matrix $\mu(e) = \{\mu_{ij}(e)\}_{1\le i,j \le d}$ is defined 
for $e\in \R^d$: $|e|=1$ as
\begin{align*}
    \mu_{ij}(e)  &= \frac{1}{\la(e)} \int_{\rho_-}^{\rho_+} 
    \left[     D_{ij}(\rho) \sqrt{W_e(\rho)} -\frac12 \partial_{e_i} (W_e(\rho))
    \,\partial_{e_j} \left( \frac{ a_e(\rho)}{\sqrt{W_e(\rho)}} \right)  \right] d\rho,  \\
    \la(e)   &=  \int_{\rho_-}^{\rho_+}  \sqrt{W_e(\rho)} d\rho,  \quad
    W_e(\rho)  = -2 \int_{\rho_-}^\rho a_e(\rho) f(\rho) d\rho, \quad a_e(\rho) = e \cdot D(\rho) e.
\end{align*}
Note that $W_e(\rho)\ge 0$ for $\rho\in [\rho_-, \rho_+]$ by the conditions
for $f(\rho)$.

The local-in-time well-posedness on a certain time interval $[0,T]$ with
$T>0$ of the equation \eqref{eq:d} follows from the non-degeneracy of $\mu(e)$
in the tangential direction to the interface $\Ga$: For some $c>0$,
$$
(\th, \mu(e) \th) \ge c |\th|^2
\quad \text{ for } \; \th \in \R^d \; \text{ such that } \; (\th, e)=0.
$$
Note that \eqref{eq:d} can be rewritten in an equivalent PDE  for the signed 
distance function $d(t,v)$ from $\Ga_t$;  see \cite{FP}.

The first result is for the generation of the interface, that is,
$\rho_K(t,v)$ reaches the neighborhood of $\rho_-$ or $\rho_+$ in 
a very short time of order $K^{-1}\log K$.

\begin{thm}\label{thm:gen}  {\rm (Theorem 1.1 of \cite{FP} on $\Om=\T^d$)}
Let $\rho_K(t,v)$ be the solution of the equation \eqref{1.6} and let
$\epsilon$ be such that $0 < \epsilon < \bar\rho$, where $\bar\rho 
:= \min \{ \rho_+ - \rho_*, \rho_* - \rho_- \}$.  We assume the conditions
given at the beginning of this subsection.  
 Then, there exist $K_0>0$ and $M_0>0$ such that, for all
$K\ge K_0$, the following holds at $t= t_K := K^{-1}\log K /(2f'(\rho_*))$.
\begin{enumerate}
\item[$(1)$]  $\rho_- - \epsilon
	\leq \rho_K(t_K,v) \leq \rho_+ + \epsilon$ for all $v \in \T^d$,
\item[$(2)$]  $\rho_K(t_K,v) \geq \rho_+ - \epsilon$
for $v\in \T^d$ such that  $\rho_0(v) \geq \rho_* +  M_0 /K^{1/2}$,
\item[$(3)$]  $\rho_K(t_K,v) \leq \rho_- + \epsilon$
for $v\in \T^d$ such that  $\rho_0(v) \leq \rho_* - M_0/K^{1/2}$.
\end{enumerate}
\end{thm}

The second result is for the propagation of the interface, that is,
the derivation of the interface motion $\Ga_t$ as long as it remains smooth.
The size of the transition layer is $O(K^{-1/2})$.

\begin{thm}\label{thm:prop}  {\rm (Theorem 1.2 of \cite{FP})}
Under the conditions of Theorem \ref{thm:gen}, for any $0 < \epsilon
  < \bar\rho$, there exist $K_0 > 0, C > 0$ and $T > 0$ (within the 
  local-in-time well-posedness of \eqref{eq:d})
  such that for every $K\ge K_0$
and $t \in [t_K, T]$ we have
\begin{align*}
\rho_K(t,v) \in~
	\begin{cases}
    [\rho_{-} - \epsilon, \rho_{+} + \epsilon] 
	&\text{ for all }
	v \in \T^d,
    \\
	[\rho_{+} - \epsilon, \rho_{+} + \epsilon] 
	& \text{ if } v \in \Om_t^+ \setminus \mathcal{N}_{ C/ K^{1/2}}(\Ga_t),
	\\
	[\rho_{-} - \epsilon, \rho_{-} + \epsilon] 
	& \text{ if } v \in \Om_t^- \setminus \mathcal{N}_{ C/K^{1/2} }(\Ga_t),
	\end{cases}
\end{align*}
where $\mathcal{N}_r(\Ga_t) := \{ v \in \T^d, {\rm dist}(v, \Ga_t) \le r \}$ is the 
$r$-neighborhood of $\Ga_t$.
\end{thm}

The proofs of these two  theorems are given based on the comparison theorem
for the PDE \eqref{1.6}.  We construct its super and sub solutions.
If the PDE \eqref{1.6} is a gradient flow
of a certain functional, one can apply the method of $\Ga$-convergence
to derive the motion of $\Ga_t$.  In particular, if the energy in the limit 
is a total surface tension, one would obtain an evolution of Wulff shape.
But our equation is different from this class.

\subsection{Main result}

The following theorem is obtained for the particle system $\eta^N(t)$
(with $K=K(N)$)
by combining Theorems \ref{Theorem 1.1} and \ref{thm:prop}.
The convergence rate for \eqref{eq:1.20} is much slower than $N^{-\k}$
given in Theorem \ref{Theorem 1.1}, since it contains \eqref{eq:1.21}
and the rate for this term is given in terms of $K$. 

\begin{thm} \label{Theorem 1.3}
Assume the condition \eqref{eq:1.12} for the initial value $\rho_0(v)$ and 
$H(f_0| \psi_{0,0}) \le C_1N^{d-\k_1}$ for some $C_1>0$ and $\k_1>0$
in Theorem \ref{Theorem 1.1} and
those in Theorem \ref{thm:prop}, that is, $d\ge 2$, $f(\rho)$ determined 
in \eqref{eq:f} from the flip rate of the Glauber part and the initial value 
$\rho_0(v)$ satisfy the conditions given at the beginning of Section \ref{sec:1.4}.
We also assume the sequence $K(N)\to\infty$ satisfies $1\le K(N)\le
\overline{K}_{\de_T}(N)$ as in Theorem \ref{Theorem 1.1}.

Then, for every $\a>d/2$ and $t_0\in (0,T]$, we have
\begin{equation}  \label{eq:1.20}
\lim_{N\to\infty} \sup_{t\in [t_0,T]} E[\|\rho^N(t) - \Xi_{\Ga_t}\|_{H^{-\a}(\T^d)}^2]
=0,
\end{equation}
where $T>0$ is determined as in Theorem \ref{thm:prop}.
\end{thm}

\begin{proof}
Noting that $\Xi_{\Ga_t}= \rho_+ 1_{\Om_t^+}+\rho_-1_{\Om_t^-}$ and $|\T^d|=1$,
for any $0 < \epsilon< \bar\rho$, we have by Theorem \ref{thm:prop}
$$
|\lan \rho_K(t),\phi\ran -\lan \Xi_{\Ga_t},\phi\ran|
\le \big(\epsilon+ (C_\epsilon/K^{1/2})^{d-1} \cdot |\Ga_t| \big) \|\phi\|_\infty
$$
for $K\ge K_0=K_{0,\epsilon}$ and $t\in [t_K, T]$.  This implies
\begin{equation}  \label{eq:1.21}
\|\rho_K(t) - \Xi_{\Ga_t}\|_{H^{-\a}(\T^d)}
\le C \big(\epsilon+ (C_\epsilon/K^{1/2})^{d-1} \cdot |\Ga_t| \big)
\end{equation}
for $\a>d/2$ and some $C>0$ as in the proof of Theorem \ref{Theorem 1.1} 
given in Section \ref{Section 2.2}.  Choose $\epsilon>0$ 
small enough and $K=K(N) \ge K_0$ taking $N$ large.  Then, since 
$|\Ga_t|$ is bounded, one can make the right-hand side smaller than 
$2C\epsilon$ for large enough $N$.  Note also that $t_K\le t_0$ holds 
for $K$ large.  Thus, the conclusion follows from Theorem \ref{Theorem 1.1}.
\end{proof}

This theorem establishes the autonomous phase separation directly for
the particle system.  The PDE \eqref{1.6} was used only secondarily.
We set the flip rate $c_0$ in the Glauber part as the corresponding $f$
determined in \eqref{eq:f} to satisfy the bistability and balance conditions.
This means that the particles prefer two phases with mean densities
$\rho_-$ or $\rho_+$.

From gradient models, we derived the mean curvature flow, i.e., 
$\mu_{ij}(e) = \mu \de_{ij}$; see \cite{EFHPS}, \cite{EFHPS-2},
\cite{FvMST} and also \cite{KS94}, \cite{Bona}.  
In the unbalanced case for gradient models 
assuming that $D(\rho)\equiv D$ (constant), 
on a shorter time scale such as $O(1/K^{1/2})$, 
Huygens' principle was derived.  In other words, the stronger phase region expands 
with a constant speed; see \cite{FvMST-2}.  The emergence of phase-separating
interface continuum flows in large scale interacting particle systems has been of
long-standing interest; see \cite{DP}, \cite{10}.
See \cite{F16} for other approaches.

\subsection{Outline of the article} 
\label{sec:1.6}

The article is organized as follows. In Section \ref{Section 2}, following
\cite{FUY}, we introduce a local equilibrium state $\psi_t(\eta)$ of second
order approximation with a leading order term determined by the hydrodynamic
equation \eqref{1.6}.  Then we prove that Theorem \ref{Theorem 1.1} is shown
once one can prove the bound $N^{-d}H(f_t|\psi_t) \le CN^{-\k}$, $t\ge 0$ for some
$C, \k>0$ for the relative entropy per volume of the density $f_t(\eta)=f_t^N(\eta)$ of 
the distribution of the process $\eta^N(t)$ with respect to $\psi_t(\eta)$
with a properly chosen second order term $F(\eta)$.
This bound is formulated in Theorem \ref{Corollary 2.1}.

Thus, our goal is to prove Theorem \ref{Corollary 2.1}, i.e.\ the entropy bound
$H(f_t|\psi_t) \le CN^{d-\k}$.  We start the proof in Section \ref{Section 3}.
It is concluded in Section \ref{sec:3.7}.
Unlike in \cite{FUY}, our model has the Glauber part and it
includes a diverging factor $K=K(N)$.  Moreover, in order to establish 
a quantitative hydrodynamic limit, we need to derive fine
error estimates at every step of the proof.
The proof of Theorem \ref{Corollary 2.1} is outlined in Section \ref{sec:2.3-B}
to clarify the main steps of the arguments. 
It is also provided here below.

In Section \ref{Section 3},
we first calculate the time derivative of $H(f_t|\psi_t)$ for a general $\psi_t$ 
along with careful error estimates; see Lemma \ref{Lemma 3.1}.  
Then, we formulate the refined one-block estimate,
that is, the replacement of microscopic functions (even those with a diverging
factor $K$) by their ensemble averages, providing fine error estimates; 
see Proposition \ref{One-block}.  

Sections \ref{sec:3.3} and \ref{sec:5} are the core of the non-gradient and
quantitative method.
In Section \ref{sec:3.3}, we show the Boltzmann-Gibbs principle for
the gradient replacement, by which one can replace a microscopic function 
that looks of diverging
order $O(N)$ but with vanishing ensemble average by a linear function of $\eta_x$
of gradient form with a properly chosen coefficient; see Theorem \ref{Theorem 3.2}.
This theorem, which gives good error estimates for the replacement, is proven
by using a key lemma stated as Lemma \ref{Lemma 3.3} 
(sometimes called Kipnis-Varadhan estimate or It\^o-Tanaka trick) and
the decay estimates for the so-called CLT (central limit theorem)
variances given in Section
\ref{sec:5}.  We then need Lemma \ref{Lemma 3.4} to replace the linear function
of $\eta_x$ of gradient form with a function of order $O(1)$.

Section \ref{sec:5} concerns the CLT variances related to the gradient
replacement and deals with quantities determined only by the localized Kawasaki
generator and under the Bernoulli measures $\{\nu_\rho\}_\rho$.  
In particular, the Glauber part plays no role in this section.   
For the decay rate for the CLT variance of the term $A_\ell$ (a linear function
of $\eta_x$ of gradient form), we have Theorem \ref{Theorem 5.1} 
as a refinement of Theorem 5.1 of \cite{FUY}.
This theorem and Proposition \ref{prop:5.4}, which provides a decay estimate 
for a minimizing sequence for the variational formula \eqref{1.3},
are the key in the quantitative gradient replacement and are obtained by 
applying new results from \cite{FGW}.
We note that Theorem 5.1 of \cite{FUY}
gave only the convergence without rate and it was shown based on the so-called
characterization of the closed forms originally due to Varadhan \cite{11};
cf.\ \cite{5}.  However, this is not sufficient for our purpose.  Instead, we use
the results of \cite{FGW}, which were shown inspired by the recent progress
in quantitative homogenization theory.  Proposition \ref{Proposition 5.1},
which calculates the CLT variances for other terms, is 
a refinement of Proposition 5.1 of \cite{FUY}.  
The results of this section are used in the proof of Theorem \ref{Theorem 3.2}.

In Section \ref{section:6}, we first summarize
the estimates obtained so far in Lemma \ref{Lemma 3.2}.
To provide a further bound for an expectation under $f_t$ (which is an unknown
distribution) by that under $\psi_t$ (which is a well-understood distribution),
we apply the entropy inequality.  Next, we are required to show the large
deviation type upper bound under $\psi_t$ with fine error estimates (see 
Lemma \ref{Theorem 3.3}) and, as its consequence, we obtain
Proposition \ref{Theorem 2.1}.
One can observe that the main term in the estimate in Proposition
\ref{Theorem 2.1} is negligible when the leading term of $\psi_t$ is determined
according to the hydrodynamic equation \eqref{1.6}; see Lemma \ref{Lem:2.3}.
The equation \eqref{1.6} is used only for this lemma.  

In Section \ref{sec:6}, the choices of  the second order term $F_N=
\Phi_{n(N)}$ of $\psi_t$ and the mesoscopic scaling parameters $n=n(N)$ and
$\ell=\ell(N)$ are given.  We apply multi-scale analysis
with different scaling parameters chosen as $1\ll n \ll \ell \ll N$.
We use Proposition \ref{prop:5.4} to determine $\Phi_n$ for $F_N$.
We have the convergence rate of the diffusion matrix
approximated from the  finite volume, which is shown to be uniform in the density,
and an $L^\infty$-estimate on the local function $\Phi_n$ with support size $n$,
which is a minimizing sequence for the variational formula \eqref{1.3} in $F$.  
The possible range of $K=K(N)$: $K \le \overline{K}_\de(N)=\de \log N$,
is determined in this section.  The Schauder estimates for the solution of
the hydrodynamic equation \eqref{1.6} are used essentially only in this section.

In Section \ref{sec:3.7}, we summarize all these results and conclude the 
proof of Theorem \ref{Corollary 2.1}.

Section \ref{sec:4-C} gathers the proofs of lemmas and propositions,
postponed from Sections \ref{Section 3}--\ref{section:6}:  
Lemma \ref{Lemma 3.1} (calculation of $\partial_t h_N(t)$), 
Proposition \ref{One-block} (refined one-block estimate),
Lemma \ref{Lemma 3.3} (a key lemma), 
Lemma \ref{Lemma 3.4} (rewriting $O(N)$-looking term to $O(1)$),
Proposition \ref{Proposition 5.1} (CLT variance)
and Lemma \ref{Lem:2.3} (showing $g(t)\le 0$ for small $\de>0$).
These are shown based on known results with some modifications
and refinements, providing careful error estimates.  For example,
for Proposition \ref{One-block}, we combine
the argument in \cite{FvMST} and the equivalence of ensembles
with a precise convergence rate.  For Proposition \ref{Proposition 5.1},
we provide a decay estimate for the error terms in the CLT variances,
which was not given in \cite{FUY}.  We also use the equivalence
of ensembles with a precise convergence rate.  
Careful error estimates are found also in \cite{JM} for related models.

Sections \ref{sec:Schauder}, \ref{sec:8}, \ref{sec:9} and \ref{Section9} are
complementary and are devoted to the proofs, respectively, of 
the Schauder estimates, the comparison
theorem, the non-degeneracy of the diffusion matrix $D(\rho)$ and to 
giving examples of the rates $\{c_{x,y}\}$ and $\{c_x\}$ that satisfy
our assumptions for $D(\rho)$ and $f(\rho)$.

In Appendices, though they are well-known or shown by elementary calculations,
for the sake of completeness, we show  the $r$-Markov property used 
in the proof of Lemma \ref{Lemma 3.1} in Appendix A and 
the large deviation error estimate stated in Lemma \ref{Theorem 3.3}
in Appendix B
for the local equilibrium state $\psi_t$ of second order approximation
defined by \eqref{eq:2psit}.

\subsection{Discussion}

The restriction $K(N)\le \de \log N$ for the diverging speed of $K(N)$
is technical and caused when we apply Gronwall's inequality at the end
of the proof in Section \ref{sec:3.7}; see also \eqref{eq:overlineK}.
This condition for $K(N)$ is certainly not the best.  When the Kawasaki part
is the simple exclusion, \cite{KS94} showed a similar result for
$K$ such that $K(N) = N^{2\zeta^*/(1+\zeta^*)}\times o(1)$ (as $N\to\infty$)
for a certain constant $\zeta^*\in (0,1)$ and they derived the mean curvature flow
for $\Ga_t$ globally in time interpreted in the viscosity sense
beyond the time that a geometric singularity appears in $\Ga_t$.  
(Their scaling parameters $\e$ and $\la=\la(\e)$ are related to our
$N$ and $K=K(N)$ as $\e=K^{1/2}/N$ and $\la=K^{-1/2}$.)
Their method for the particle system seems difficult to apply to our model.

We consider the Kawasaki part which has the Bernoulli product measures
as its reversible measures.  This corresponds to an infinite temperature model.
The hydrodynamic limit for the Kawasaki dynamics of non-gradient type,
which has the canonical Gibbs measures as its reversible measures, was
studied in \cite{VY}.  It is natural to study our problem for such dynamics 
added the Glauber part.  To do this, we need to extend the results of
\cite{FGW} to the Kawasaki dynamics associated with the canonical
Gibbs measures.  Currently, this extension is an open and difficult problem.

The case of an unbalanced reaction term was studied in \cite{FvMST-2}
when the Kawasaki part is simple.  It is natural to generalize this to the
non-gradient Kawasaki case, but this was nontrivial even for the gradient
Kawasaki part with a speed change.  The reason is that we have a PDE
problem: we need to study the sharp interface limit for an unbalanced
reaction term and nonlinear diffusion.  This remains unsolved.

\section{Proof of Theorem \ref{Theorem 1.1}}  \label{Section 2}

\subsection{Formulation of Theorem \ref{Corollary 2.1}: estimate on the 
relative entropy}  \label{Section 2.1}

To prove Theorem \ref{Theorem 1.1}, we apply the relative entropy 
method as in \cite{FUY} comparing the distribution of $\eta^N(t)$ with a
properly taken local equilibrium state 
 of second order approximation.  Its leading term is
determined from the solution $\rho_K(t,v)$ of the hydrodynamic equation 
(\ref{1.6}).  The second order term is determined by $F=F(\eta)$ which
appears in the variational formula \eqref{1.3}, and plays a similar role to
the corrector in the theory of homogenization.  Note that, for gradient
models, the second order approximation is unnecessary; see
\cite{EFHPS}, \cite{FvMST}, \cite{FvMST-2}.
 
Given a function  $\la = \la(t,v)\in C^{1,3}([0,T]\times\T^d)$ 
and a function $F=F(\eta) \in\mathcal{F}_0^d$, we define a local equilibrium 
state  $\psi_t(\eta)d\nu^N$  of second order approximation by
\begin{align}  \label{eq:2psit}
\psi_t(\eta) & \equiv \psi_{\la(t,\cdot),F}(\eta)  \\
& = Z_t^{-1} \exp\bigg\{ \sum_{x\in\T_N^d} \la(t,x/N)
\eta_x + \frac1{N} \sum_{x\in\T_N^d} \left(\partial\la(t,x/N), \tau_x F(\eta)
\right)\bigg\},   \notag
\end{align}
for  $\eta\in\mathcal{X}_N$,
where  $Z_t = Z_{\la(t,\cdot),F}$  is the normalization constant 
with respect to $\nu^N\equiv \nu_{1/2}^N$ and  $\partial \la \equiv 
\partial_v \la= \{\partial_{v_i} \la\}_{1\le i \le d}$.
We also write  $\dot{\la} = \partial \la/\partial t$, $\partial^2 \la \equiv 
\partial_v^2 \la= \left\{ \partial_{v_i}\partial_{v_j}\la \right\}_{1 \le i,j \le d}$ and 
$\partial^3 \la \equiv \partial_v^3 \la= \left\{ \partial_{v_i}\partial_{v_j}\partial_{v_k}
\la \right\}_{1 \le i,j,k \le d}$ for  $\la = \la(t,v)$.   

Denote by  $f_t(\eta)=f_t^N(\eta)$  the density of the 
distribution of  $\eta^N(t)$  on  $\mathcal{X}_N$  with respect to  $\nu^N$  
and consider the relative entropy per volume defined by
\begin{equation}
h_N(t) = N^{-d} H(f_t|\psi_t).     \label{2.2}
\end{equation}
We will show that $h_N(t)\le CN^{-\k}$, i.e.\ $H(f_t|\psi_t)\le CN^{d-\k}$,
$t\in [0,T]$  with some $C=C_T>0$
and $\k>0$ by choosing $\la(t,\cdot)$ and $F$ in $\psi_t = \psi_t^N(\eta)$
properly as in the following theorem.

\begin{thm}  \label{Corollary 2.1}
Let $\de=\de_T=\frak{c}/T>0$ with sufficiently small $\frak{c}>0$, and
let functions $F_N=F_N(\eta) \in \mathcal{F}_0^d$ 
be chosen as in Proposition \ref{prop:5.4} below, namely, $F_N=\Phi_{n(N)}$
with $\Phi_n$ satisfying \eqref{eq:Fn} (which provides the convergence rate
for a minimizing sequence $\{\Phi_n\}_n$ for the variational problem \eqref{1.3})
and taking $n(N) =[N^{a_1}]$  (i.e.\ the integer part of $N^{a_1}$)
with $a_1$ chosen sufficiently small such that
$a_1\in (0,1/(2d+5))$; cf.\ \eqref{eq:a1a2}.

Suppose a sequence $K=K(N) \to\infty \, (N\to\infty)$ is given
that satisfies $1\le K \le \overline{K}_\de(N)=\de \log N$ with $\de>0$ chosen
as above.  Determine $\la(t,v) \equiv \la_K(t,v) := \bar{\la}(\rho(t,v))$ 
from the solution  $\rho(t,v)= \rho_K(t,v)$  of the hydrodynamic equation 
(\ref{1.6}) with $K=K(N)$ and the initial value $\rho_0$ satisfying
\eqref{eq:1.12}; recall \eqref{eq:2.4} for $\bar\la(\rho)$.

Consider $h_N(t)$ defined by \eqref{2.2} with $\psi_t = \psi_{\la(t,\cdot),F_N}$
in \eqref{eq:2psit} taking $\la(t,\cdot)$ and $F_N$ as above, and assume that
$h_N(0) \le C_1N^{-\k_1}$ for some $C_1>0$ and $\k_1>0$.  

Then, there exist some
$C=C_T>0$ and $\k>0$ such that $h_N(t) \le CN^{-\k}$ holds for every $t \in [0,T]$.
\end{thm}

The proof of Theorem \ref{Corollary 2.1} will be given in Section \ref{sec:3.7}
starting from Section \ref{Section 3}.  We will calculate the time derivative 
of $h_N(t)$, in which derivatives of $\la_K(t,v)$ appear.
For the solution $\rho=\rho_K(t,v)$ of the hydrodynamic equation \eqref{1.6}
with smooth coefficients $D$ and $Kf$, we have the Schauder estimates
in terms of $K\ge 1$:
\begin{equation}\label{eq:Schauder}
\|\partial^k\rho\|_\infty, \; \|\partial_t\rho\|_\infty,\;
\|\partial \partial_t\rho\|_\infty  \le CK^{\a_0},
\quad k=1,2,3,
\end{equation}
for some $\a_0>0$ and $C=C_T>0$, where 
$\|\cdot\|_\infty = \|\cdot\|_{L^\infty([0,T]\times \T^d)}$.  These estimates
are shown by applying the results in \cite{Li96}; see Section \ref{sec:Schauder}
for details.  In particular, for $\la=\la_K(t,v)=\bar\la(\rho_K(t,v))$, 
recalling the definition \eqref{eq:2.4} of $\bar\la$ and noting that $\rho_K(t,v)$ is
uniformly away from $0$ and $1$ by Lemma \ref{lem:max} (the comparison
theorem) below, we have
\begin{align}  \label{eq:3-Schauder}
\|\partial\la\|_{3,\infty}, \;  \|\dot{\la}\|_\infty, \;
\|\partial \dot{\la}\|_\infty \le C K^{\a_0}
\end{align}
by changing $C=C_T>0$ (depending on $\|\la\|_\infty$), where we set
$$
\|\partial\la\|_{3,\infty} := \|\partial\la\|_\infty + \|\partial^2\la\|_\infty
 + \|\partial^3\la\|_\infty.
$$
To show \eqref{eq:3-Schauder} from \eqref{eq:Schauder},
note that for the function $\bar\la = \bar\la(\rho)$ 
defined by \eqref{eq:2.4} we have
\begin{align}  \label{eq:2.5-C}
\frac{\partial\bar\la}{\partial\rho}= \frac1{\chi(\rho)}, \quad
\frac{\partial^2\bar\la}{\partial\rho^2}= \frac{-1+2\rho}{\chi(\rho)^2}, \quad
\frac{\partial^3\bar\la}{\partial\rho^3}= \frac{2(1-3\rho+3\rho^2)}{\chi(\rho)^3}, 
\end{align}
where $\chi(\rho)=\rho(1-\rho)$ was introduced in \eqref{1.4}.

The Schauder estimates \eqref{eq:Schauder}, \eqref{eq:3-Schauder}
are actually used only in Section \ref{sec:6} except for \eqref{eq:2.6-a}
and \eqref{eq:2.12-A} below. 
But, as these estimates suggest and also as $e^{Kt/\de_*}$ appears in
\eqref{eq:3.hNt} after 
applying Gronwall's inequality at the end of the proof in Section \ref{sec:3.7},
at all steps in the proof of Theorem
\ref{Corollary 2.1}, we need to obtain error estimates that are strong enough 
to control diverging factors caused by  $K$.  This is used to show
the quantitative version of the hydrodynamic limit as well.

Note that, as the hydrodynamic equation, a discrete PDE was used in
\cite{EFHPS}, \cite{FvMST}, \cite{FvMST-2} to ensure an exact cancellation 
for the leading term in the entropy computation, but here we use a continuous 
PDE \eqref{1.6}, since our
error estimates are strong enough to cover the difference.

\subsection{Proof of Theorem \ref{Theorem 1.1}}
\label{Section 2.2}

The proof of Theorem \ref{Theorem 1.1} is given based on Theorem
\ref{Corollary 2.1}, i.e.\ the bound $H(f_t|\psi_t) \le C N^{d-\k}$ for the
relative entropy with respect to $\psi_t=\psi_{\la_K(t,\cdot),F_N}$,
where $\la_K(t,\cdot)$ and $F_N$ are chosen as in the statement of
Theorem \ref{Corollary 2.1}, combining with the entropy inequality 
and the concentration inequality called Hoeffding's inequality.  
We state two lemmas.  In this section,
we assume the assumptions stated in Theorem \ref{Theorem 1.1}.

We will sometimes identify $\psi_t=\psi_{\la_K(t,\cdot),F_N}$, 
$\psi_{t,0}=\psi_{\la_K(t,\cdot),0}$ and $f_t$ with $\psi_t \, d\nu^N$, 
$\psi_{t,0}d\nu^N$ and $f_t \, d\nu^N$, denoted by $P^{\psi_t}$, $P^{\psi_{t,0}}$ 
and $P^{f_t}$ (probability measures on $\mathcal{X}_N$), and the expectations under
them by $E^{\psi_t}$, $E^{\psi_{t,0}}$ and $E^{f_t}$, respectively.

\begin{lem} \label{lem:Add-2}
For every $T>0$, there exist $C=C_T>0$ and $\k>0$ such that
$$
\sup_{t\in [0,T]}
E\Big[|\lan \rho^N(t),\phi\ran- \lan \rho_K(t),\phi\ran_N|^2\Big]
\le C \|\phi\|_\infty^2 N^{-\k},
$$
where
\begin{equation}  \label{eq:2rhophiN}
\lan \rho_K(t),\phi\ran_N := \frac1{N^d} \sum_{x\in \T_N^d}
\rho_K(t,x/N) \phi(x/N).
\end{equation}
\end{lem}

\begin{proof}
For every $\de>0$, by the entropy inequality, we have
\begin{equation}  \label{eq:3-ent-Q}
E\Big[|\lan \rho^N(t),\phi\ran- \lan \rho_K(t),\phi\ran_N|^2\Big] 
\le \frac{1}{\de N^d} \log E^{\psi_{t,0}}\Big[ e^{\de X_N^2}\Big] 
+ \frac{1}{\de N^d}  H_N(f_t|\psi_{t,0}),
\end{equation}
where $\psi_{t,0}= \psi_{\la(t,\cdot),0}$ with $\la(t,\cdot)=\la_K(t,\cdot)$,
$F=0$ and $X_N$ is defined under $\psi_{t,0} d\nu^N$ by
\begin{align*}
X_N :=  N^{d/2} \{ \lan \rho^N,\phi\ran- \lan \rho_K(t),\phi\ran_N\}
=  N^{-d/2} \sum_{x\in \T_N^d} \zeta_x \phi(x/N),
\end{align*}
and $\zeta_x := \eta_x -\rho_K(t,x/N)$.

To estimate the entropy in \eqref{eq:3-ent-Q}, 
we compare $\psi_t = \psi_{\la(t,\cdot),F_N}$ and $\psi_{t,0}=\psi_{\la(t,\cdot),0}$
recalling \eqref{eq:2psit}:
$$
\psi_{\la(t,\cdot),F_N} = \frac{Z_{\la(t,\cdot),0}}{Z_{\la(t,\cdot),F_N}}
e^{R(\eta)} \psi_{\la(t,\cdot),0}, 
$$
where
$$
R(\eta) = \frac1N \sum_{x\in \T_N^d} \big( \partial \la(t,x/N), \t_x F_N(\eta)\big).
$$
We have a simple bound
\begin{align}  \label{eq:2.6-a}
|R(\eta)| \le N^{d-1} \|\partial\la(t)||_\infty \|F_N\|_\infty
  \le C_2 N^{d-\k_2}, 
\end{align}
for some $0<\k_2<1$,
by the Schauder estimates \eqref{eq:3-Schauder}, $K\le \overline{K}(N)$
 and Proposition \ref{prop:5.4} noting
that $F_N = \Phi_{n(N)}$ with $n(N)=[N^{a_1}]$, $0<a_1<1/(2d+5)$
(as in the statement of Theorem \ref{Corollary 2.1})  and
$$
0< \frac{Z_{\la(t,\cdot),0}}{Z_{\la(t,\cdot),F_N}} \le e^{\|R\|_\infty}.
$$
Thus, in \eqref{eq:3-ent-Q}, we can estimate as
\begin{align}  \label{eq:1.3-Q}
H_N(f_t|\psi_{t,0}) & = \int \log \frac{f_t}{\psi_{t,0}} \, f_t d\nu^N \\
& = \int \Big( \log \frac{f_t}{\psi_t} + \log \frac{\psi_t}{\psi_{t,0}} \Big) f_t d\nu^N
   \notag  \\
& \le H_N(f_t|\psi_t) + 2 \|R\|_\infty
 \le H_N(f_t|\psi_t) + 2C_2 N^{d-\k_2}.
\notag
\end{align}

For $H_N(f_t|\psi_t)$, by Theorem \ref{Corollary 2.1},
we have $H_N(f_t|\psi_t) \le CN^{d-\k}$, $t\in [0,T]$ with $C=C_T>0$, $\k>0$.
Indeed, the condition for the relative entropy at $t=0$ in Theorem 
\ref{Theorem 1.1} implies that for $h_N(0)$ in Theorem \ref{Corollary 2.1},
as we see 
\begin{align}\label{eq:2.6-a-Q}
|H(f_0|\psi_0) - H(f_0|\psi_{0,0})| 
  & = \Big|\int f_0 \log(\psi_{0,0}/ \psi_0) \, d\nu^N \Big| 
  \le 2 C_2 N^{d-\k_2},
\end{align}
which is shown similarly to \eqref{eq:1.3-Q} taking $t=0$.
Thus, the assumption $h_N(0) \le \bar C_1N^{-\bar\k_1}$ of Theorem 
\ref{Corollary 2.1} (where we write $\bar\k_1, \bar C_1$ instead of $\k_1, C_1$
to distinguish them from the constants in Theorem \ref{Theorem 1.1})
holds by taking $\bar \k_1 = \k_1 \wedge \k_2>0$ and 
$\bar C_1 = C_1 \vee (2C_2)>0$,
where $\k_1, C_1$ are the constants in Theorem \ref{Theorem 1.1}.

On the other hand, for the first term in \eqref{eq:3-ent-Q},
since $\{\zeta_x\}$ are independent and mean $0$
under $\psi_{t,0} d\nu^N$, by Hoeffding's inequality (see Theorem 2.8 in
\cite{BLM}, take $N^{d/2}s$ for $t$ and note $|\bar \eta_x\phi(\tfrac{x}N)|
\le \|\phi\|_\infty$), we obtain that
$$
P^{\psi_{t,0}}\big(|X_N|\ge s\big) \le 2 e^{-s^2/(2\|\phi\|_\infty^2)},
$$
for all $s>0$ and $\phi\not\equiv 0$.  Thus, we have
\begin{align}  \label{eq:1.5-Q}
E^{\psi_{t,0}}\Big[e^{\de X_N^2}\Big]
& = \sum_{i=0}^\infty E^{\psi_{t,0}} 
\Big[e^{\de X_N^2}; i \le |X_N|< i+1\Big]  \\
& \le \sum_{i=0}^\infty e^{\de (i+1)^2} P^{\psi_{t,0}}\big(|X_N|\ge i\big)  \notag  \\
& \le \sum_{i=0}^\infty e^{\de (i+1)^2}  2 e^{-i^2/(2\|\phi\|_\infty^2)}
=: C_3 <\infty,  \notag
\end{align}
uniformly for $\phi: \|\phi\|_\infty=1$, by taking $\de=1/4$.

Summarizing \eqref{eq:3-ent-Q} taking $\de = 1/4$, \eqref{eq:1.3-Q} and
$H_N(f_t|\psi_t) \le C N^{d-\k}$, we have
\begin{equation*}
E\Big[|\lan \rho^N(t),\phi\ran- \lan \rho_K(t),\phi\ran_N|^2\Big] 
\le \frac{4\log C_3}{N^d} + 4 \big(C N^{-\k} 
+ 2 C_2 N^{-\k_2} \big) 
\le C_4 N^{-\k\wedge \k_2},
\end{equation*}
where the constant $C_4 = C_{4,T}>0$ is uniform in $\phi$ such that
$\|\phi\|_\infty= 1$. 

For general $\phi (\not\equiv 0)$, taking $\phi/\|\phi\|_\infty$ for
$\phi$, we obtain
\begin{equation*}
E\Big[|\lan \rho^N(t),\phi\ran- \lan \rho_K(t),\phi\ran_N|^2\Big] 
\le C_4 \|\phi\|_\infty^2 N^{-\k},
\end{equation*}
by writing $\k\wedge\k_2$ by $\k>0$ again.
This concludes the proof of the lemma.
\end{proof}

Next, for $\b\in (0,1]$, let us consider the H\"older norms 
$\|g\|_{C^\b} = \|g\|_{C^\b(\T^d)}$ of a function $g$ on $\T^d$:
\begin{equation*}  
\|g\|_{C^\b} := \|g\|_\infty + \sup_{v_1\not=v_2} \frac{|g(v_1)-g(v_2)|}{|v_1-v_2|^\b},
\end{equation*}
where $|v|\in [0,\sqrt{d}/2]$ is  the Euclidean norm
defined for $v\in \T^d$ modulo $1$ in each component.
Later, we will take $\rho_K(t,\cdot)$ for $\rho(\cdot)$ in the next lemma.

\begin{lem} \label{lem:Add-3}
For every $\b\in (0,1]$ and $\rho=\rho(\cdot)$, $\phi=\phi(\cdot)\in C^\b(\T^d)$,
we have
$$
|\lan \rho,\phi\ran_N -\lan \rho,\phi\ran|
 \le d^{\b/2} (2N)^{-\b} \|\rho\cdot\phi\|_{C^\b(\T^d)}
$$
Recall \eqref{eq:2rhophiN} for $\lan \rho,\phi\ran_N$ and
$$
\lan \rho,\phi\ran = \int_{\T^d} \rho(v)\phi(v)dv.
$$
\end{lem}

\begin{proof}
The proof is easy:
\begin{align*}
| \lan \rho,\phi\ran_N -\lan \rho,\phi\ran |
& = \bigg| \sum_{x\in \T_N^d} \int_{B(x/N,1/N)} \Big(
\rho(x/N) \phi(x/N) -\rho(v)\phi(v)\Big) dv \bigg| \\
& \le \sum_{x\in \T_N^d}  N^{-d} \|\rho\cdot\phi\|_{C^\b(\T^d)} 
\sup_{v\in B(x/N,1/N)} \big| v - x/N\big|^\b\\
& \le \big(\sqrt{d}/2N\big)^\b \|\rho\cdot\phi\|_{C^\b(\T^d)},
\end{align*}
where $B(x/N,1/N)$ is the box in $\T^d$ with center $x/N$ and side length 
$1/N$.
\end{proof}

We are at the position to give the proof of Theorem \ref{Theorem 1.1}.
For the sake of completeness, 
let us define the Sobolev spaces $H^{\a} \equiv H^{\a}(\T^d)$ on $\T^d$
for $\a\in \R$.  We consider the Laplacian $\De$ on $\T^d$.  Its 
normalized eigenfunctions are products of 
$$
\{\sqrt{2} \sin\pi n_iv_i, 1, \sqrt{2}\cos\pi m_iv_i\}_{n_i,
m_i = 1,2,\ldots}
$$
choosing one for each $i=1,2,\ldots,d$.  We denote them by $e_{\bf n}$
for ${\bf n} = (n_1,\ldots,n_d) \in \Z^d$, i.e., we choose
$\sqrt{2} \sin\pi n_iv_i$, $1$ or $\sqrt{2}\cos\pi (-n_i)v_i$
for $i$ according to $n_i>0$,  $n_i=0$ or $n_i< 0$, respectively,
and make product to define  $e_{\bf n}(v), v\in \T^d$.
Then, $e_{\bf n}$ is an eigenfunction of $-\De$ with
eigenvalue $\la_{\bf n} := \pi^2|{\bf n}|^2 \equiv \pi^2 \sum_{i=1}^d n_i^2$, 
i.e.\ $-\De e_{\bf n} = \la_{\bf n} e_{\bf n}$.  Moreover,
$\{e_{\bf n}\}_{{\bf n}\in \Z^d}$ forms a complete orthonormal system
of $L^2(\T^d)$. 

Define the norm $\|g\|_{H^\a}$ first for $g\in C^\infty(\T^d)$ by
\begin{align}\label{eq:Ha}
\|g\|_{H^{\a}}^2 := \lan (-\De+1)^{\a}g,g\ran
\equiv \sum_{{\bf n}\in \Z^d} (\la_{\bf n}+1)^{\a} \lan g, e_{\bf n}\ran^2,
\end{align}
where $\lan\cdot,\cdot\ran$ denotes the inner product of $L^2(\T^d)$.
The Sobolev space $H^\a(\T^d)$ is the completion of $C^\infty(\T^d)$
under the norm $\|\cdot\|_{H^\a}$.  Note that discrete measures on $\T^d$
such as $\rho^N(t)$ are elements of $H^{-\a}(\T^d)$ for $\a>d/2$,
but not for $a\le d/2$ in general.  For example, we easily see
$\|\de_0\|_{H^{-\a}}=\infty$ for $\a\le d/2$.

\begin{proof}[Proof of Theorem \ref{Theorem 1.1}]
By \eqref{eq:Ha}, one can rewrite as
\begin{align}  \label{eq:2.10-A}
E\Big[\|\rho^N(t)-\rho_K(t)\|_{H^{-\a}}^2\Big]
= \sum_{{\bf n}\in \Z^d} (\la_{\bf n}+1)^{-\a} 
E\Big[\lan \rho^N(t)-\rho_K(t), e_{\bf n}\ran^2\Big].
\end{align}
Then, we estimate as
\begin{align*}
E\Big[\lan \rho^N(t)-\rho_K(t), e_{\bf n}\ran^2\Big]
\le & 2 E\Big[\big( \lan \rho^N(t), e_{\bf n}\ran-
\lan \rho_K(t), e_{\bf n}\ran_N\big)^2\Big] \\
&+ 2 \big( \lan \rho_K(t), e_{\bf n}\ran_N-
\lan \rho_K(t), e_{\bf n}\ran\big)^2.
\end{align*}
For the first term, by Lemma \ref{lem:Add-2} and noting 
$\| e_{\bf n}\|_\infty \le (\sqrt{2})^d$
for all ${\bf n}$, we see
\begin{align}  \label{eq:thm1-1}
& \sum_{{\bf n}\in \Z^d} (\la_{\bf n}+1)^{-\a} 
E\Big[\big( \lan \rho^N(t), e_{\bf n}\ran-
\lan \rho_K(t), e_{\bf n}\ran_N\big)^2\Big]  \\
& \quad \le C_1 N^{-\k} \sum_{{\bf n}\in \Z^d} (\la_{\bf n}+1)^{-\a} 
\le C_2 N^{-\k},  \notag
\end{align}
if $\a>d/2$ with $C_2=C_{2,T,\a}>0$, $\k>0$.  Note that the sum is bounded by 
$C_3 \big(\int_{\R^d} (\pi^2 |x|^{2}+1)^{-\a}dx +1\big)
\le C_4\big(\int_1^\infty r^{-2\a} r^{d-1}dr +1\big)< \infty$ if $\a>d/2$.

Next, for the second term, by Lemma \ref{lem:Add-3}, we have
\begin{align*}
& \sum_{{\bf n}\in \Z^d} (\la_{\bf n}+1)^{-\a} 
\big( \lan \rho_K(t), e_{\bf n}\ran_N-
\lan \rho_K(t), e_{\bf n}\ran\big)^2 \\
& \quad \le  \sum_{{\bf n}\in \Z^d} (\la_{\bf n}+1)^{-\a} d^\b (2N)^{-2\b }
\|\rho_K(t)e_{\bf n}\|_{C^\b(\T^d)}^2,
\end{align*}
for every $\b\in (0,1]$.
However, since $\|\rho_K(t)e_{\bf n}\|_{C^\b(\T^d)} \le 
\|\rho_K(t)\|_{C^\b(\T^d)} \|e_{\bf n}\|_{C^\b(\T^d)}$ and
$$
\|e_{\bf n}\|_{C^\b(\T^d)}\le \|e_{\bf n}\|_\infty
+ \|\nabla e_{\bf n}\|_\infty^\b (2 \|e_{\bf n}\|_\infty)^{1-\b}
\le C_5(|{\bf n}|^\b +1), 
$$
the above series is bounded by
\begin{align}  \label{eq:second}
C_6 N^{-2\b }\|\rho_K(t)\|_{C^\b(\T^d)}^2
\sum_{{\bf n}\in \Z^d} (\la_{\bf n}+1)^{-\a} 
(|{\bf n}|^\b +1)^2.
\end{align}
Note that the last series converges if $\a> \b+d/2$.  

Since $\|\rho_K(t)\|_{C^\b}\le \|\rho_K(t)\|_{C^1}
\le C_\ga N^{\ga}$ for arbitrary small $\ga>0$ by the Schauder estimates
\eqref{eq:Schauder} and $K\le \overline{K}_\de(N)$, \eqref{eq:second} is bounded by
\begin{align}  \label{eq:2.12-A}
C_7 N^{-2\b + 2\ga}
\end{align}
for any $\b\in (0,1]$ such that $\b < \a-d/2$.  However, by our assumption
$\a>d/2$, we can find $\b>0$ small enough satisfying this condition and then we take
$\ga\in (0,\b)$.  Thus, from \eqref{eq:2.10-A}, \eqref{eq:thm1-1},
\eqref{eq:second} and \eqref{eq:2.12-A}, we obtain the conclusion 
of Theorem \ref{Theorem 1.1}  by taking
$\k\wedge(\a-d/2)\wedge 1$ (with $\k$ in Lemma \ref{lem:Add-2})
for $\k=\k_\a$ in this theorem. (We can take $2((\a-d/2)\wedge 1)-\e$
for arbitrary small $\e>0$ instead of $(\a-d/2)\wedge 1$.)
\end{proof}

\subsection{Toward the proof of Theorem \ref{Corollary 2.1}}
\label{sec:2.3-B}

The goal is now to prove Theorem \ref{Corollary 2.1}, which states
$h_N(t) \le CN^{-\k}$, $t\in [0,T]$ for some $C=C_T>0$ and $\k>0$.
We will begin the proof of Theorem \ref{Corollary 2.1} in the next section and
will conclude in Section \ref{sec:3.7}. 
To clarify the main steps of the arguments, let us outline the procedure.

To give the bound on $h_N(t)$, first in Section \ref{Section 3},
we calculate its time derivative $\partial_t h_N(t)$ in
Lemma \ref{Lemma 3.1}.  In this calculation, we obtain two terms
$\Om_1(\eta)$ and $\Om_2(\eta)$.  The term $\Om_1$ is of diverging order $O(N)$, 
while $\Om_2$ is $O(1)$ in $N$ but contains the diverging factor $K$.
We formulate the refined one-block estimate to replace a microscopic
function $\Om$ in general by its ensemble average in Proposition 
\ref{One-block} and Corollary \ref{One-block-cor} for specific $\Om_2$.

The main difficulty lies in treating the diverging term
$\Om_1$.  Sections \ref{sec:3.3} and \ref{sec:5} are the core of the
non-gradient and quantitative method.  We prove the ``gradient replacement"
in Theorem \ref{Theorem 3.2}, which allows one to replace $\Om_1$
with another $O(1)$ term, given as the term involving $P_{ij}$
in Lemma \ref{Lemma 3.4}.
The proof of Theorem \ref{Theorem 3.2} is reduced to the calculation of
the so-called CLT variances via a key lemma, called Kipnis-Varadhan estimate or
It\^o-Tanaka trick, formulated in Lemma \ref{Lemma 3.3} in the
setting of Glauber-Kawasaki dynamics.

The CLT variances are calculated in Section \ref{sec:5}.
In particular, Theorem \ref{Theorem 5.1}, which concerns the term $A_\ell$,
and Proposition \ref{prop:5.4}, which studies the minimizing sequence 
$\{\Phi_n\}$ for the variational formula \eqref{1.3}, play a crucial role
in showing the quantitative result for a non-gradient model.
These are taken from \cite{FGW}.

In Section \ref{section:6}, we derive an integral estimate for $h_N(t)$
in Lemma \ref{Lemma 3.2} as a consequence of the results obtained in Sections 
\ref{Section 3}--\ref{sec:5} and, combining it with the entropy inequality
and the large deviation type upper bound formulated in Lemma \ref{Theorem 3.3},
we obtain the estimate \eqref{2.5} for $h_N(t)$ in Proposition 
\ref{Theorem 2.1}.  We note in Lemma \ref{Lem:2.3} that the main
term $g(t)$ in \eqref{2.5} can be dropped if we choose the function $\la(t,v)$ 
in the local equilibrium $\psi_t$ according to the hydrodynamic equation \eqref{1.6}.

Then, in Section \ref{sec:6}, after summarizing all error estimates, we choose
the parameter $\b$ and the mesoscopic scales $n$ and $\ell$ appearing in these
estimates, and determine the second approximation $F=F_N$ in $\psi_t$ and
find the necessary condition for $K=K(N)$.

The proof of Theorem \ref{Corollary 2.1} is concluded in 
Section \ref{sec:3.7} by applying Gronwall's inequality.
The proofs of several lemmas and propositions are postponed to
Section \ref{sec:4-C} and Appendices.

The procedure is outlined also in Section \ref{sec:1.6}.

\section{Time derivative of $h_N(t)$ and one-block estimate}  
\label{Section 3}

In this section, we first give the estimate on $\partial_t h_N(t)$ with 
a proper error estimate to be
sufficient for our purpose, that is,  to show a quantitative bound
$h_N(t) \le CN^{-\k}$ with $\k>0$; see Section \ref{sec:3.1}.
In this estimate, two terms $\Om_1=\Om_1(\eta)$ and $\Om_2=\Om_2(\eta)$ 
appear.  For the term $\Om_2$, we show a ``refined'' one-block estimate,
that is, the local averaging property due to the local ergodicity.
``Refined'' means that it provides an error estimate sufficient for our purpose
and also properly controls the diverging factor $K$ in $\Om_2$; 
see Section \ref{sec:3.2}.  

On the other hand, the term  $\Om_1$ looks of the order $O(N)$.  
We need the gradient replacement which is the main part in the non-gradient method
and will be studied in Sections \ref{sec:3.3} and \ref{sec:5}.

The results stated in this section, Lemma \ref{Lemma 3.1} and Proposition
\ref{One-block}, are the extensions of known results,
providing fine error estimates, so that the proofs are postponed to Section
\ref{sec:4-C}.  Corollary \ref{One-block-cor} is an application of
Proposition \ref{One-block} to our specific $\Om_2(\eta)$.

\subsection{Time derivative of the relative entropy  $h_N(t)$}
\label{sec:3.1}

Let us begin with the calculation of the time derivative of $h_N(t)$.  
If we take $K=0$ (i.e.\ no Glauber part), the following lemma is in
Lemma 3.1 of \cite{FUY}, but the error estimate given there was of order $o(1)$ 
and it was enough to show $h_N(t)=o(1)$.  For our purpose, this is not
sufficient.  To show $h_N(t) \le CN^{-\k}$ as in  Theorem \ref{Corollary 2.1},
we need to make it finer.  

For $\ell\in \N$ and $x\in \T_N^d$ (or $\Z^d$), define
\begin{align}  \label{eq:3.1-P}
\La_{\ell,x} := \{y\in \T_N^d \;(\text{or }\Z^d); \, |y-x|\le \ell\}
\end{align}
and $\ell_* := 2\ell +1$.  Note that, in case of $\T_N^d$, by embedding it
in $\Z^d$, the $\ell^\infty$-norm $|y-x|= \max_{1\le i \le d}|x_i-y_i|$ 
and therefore $\La_{\ell,x}$ are well-defined
at least if $\ell_*<N$.  Note that $|\La_{\ell,x}|=\ell_*^d$, which is
the volume of $\La_{\ell,x}$.

For $F=F(\eta) \equiv (F_i(\eta))_{i=1}^d \in \mathcal{F}_0^d$,
we set
\begin{align}  \label{eq:3.2-P}
\vertiii{F}_{k,\infty} = r(F)^{d+k} \| F\|_\infty,
\quad k=0, 1,2,
\end{align}
where $r(F) (\ge 1)$ denotes the radius (in $\ell^\infty$-sense, centered at $0$)
of the support of $F=F(\eta)$, that is,
\begin{align}  \label{eq:3rF}
r(F) = \min \big\{r>0; \text{supp } F \subset \La_{r,0}\big\}.
\end{align}

We use an abbreviation $\sum_{x\sim y}$  for  $\sum_{x,y\in\T_N^d: |x-y|=1}$.

\begin{lem} \label{Lemma 3.1}
Assume $\la \in C^{1,3}([0,T]\times\T^d)$ and $F \in\mathcal{F}_0^d$
for $\psi_t=\psi_{\la(t,\cdot),F}$ in \eqref{eq:2psit}.  Then, setting 
$\rho(t,v) := \bar\rho(\la(t,v))$, we have
$$
\partial_t h_N(t) \le E^{f_t}[\Om_1 + \Om_2]
   + N^{-d} \sum_{x\in\T_N^d} \dot{\la}(t,x/N)\rho(t,x/N) + Q_N^{En}(\la,F),
$$
where  $\Om_1 = \Om_1(\eta)$  and  $\Om_2 = \Om_2(\eta)$  are
defined respectively by
\begin{align}  \label{eq:Om1}
\Om_1 = & - \frac{N^{1-d}}2 \sum_{x\sim y} c_{x,y} \Om_{x,y}, 
\\   \label{eq:Om2}
\Om_2 = & - N^{-d} \sum_{x\in\T_N^d} \dot{\la}(t,x/N) \eta_x 
   + \frac{N^{-d}}4 \sum_{x\sim y} c_{x,y} \Om_{x,y}^2   
       \\  \notag
 & -\frac{N^{-d}}4 \sum_{x\sim y} c_{x,y} \sum_{i,j}
     \partial_{v_i}\partial_{v_j} \la(t,x/N) (y_i -x_i)(y_j-x_j)(\eta_y-\eta_x) 
       \\  \notag
 &   + \frac{N^{-d}}2 \sum_{x\sim y} c_{x,y} \sum_{i,j}
     \partial_{v_i}\partial_{v_j}\la(t,x/N)  \pi_{x,y}(\sum_{z\in\T_N^d}(z_j-x_j)
       \tau_z F_i)  \\ \notag
 & + N^{-d} K \sum_{x\in\T_N^d} \Big\{ \frac{c_x^+(\eta)}{\rho(t,x/N)}      
  -  \frac{c_x^-(\eta)}{1-\rho(t,x/N)}\Big\}  \{ \eta_x-\rho(t,x/N)\},       
\intertext{and}  \label{eq:Omxy}
\Om_{x,y} & \equiv \Om_{x,y}(\eta) =    \bigg(\partial\la(t,x/N),
     (y - x)(\eta_y-\eta_x)- \pi_{x,y}\Big(\sum_{z\in\T_N^d}
       \tau_z F\Big)\bigg).
\end{align}
The error term $Q_N^{En}(\la,F)$ has the estimate
\begin{align}  \label{eq:R(la,F)}
|Q_N^{En}(\la,F)| \le &
CN^{-1} \Big(1+ \|\partial\la\|_{3,\infty}\Big)^3
   \Big(1+\vertiii{F}_{2,\infty}\Big)^3 \\
   & + CN^{-1}K e^{\|\la\|_\infty}
   \|\partial\la\|_\infty \vertiii{F}_{0,\infty}   \notag  \\
   & + CN^{-1} \|\partial\dot{\la}\|_\infty \|F\|_\infty + 
   CN^{-1} \|\dot{\la}\|_\infty \|\partial\la\|_\infty\vertiii{F}_{0,\infty},  \notag
\end{align}
as long as $\la$ and $F$ satisfy the bounds
\begin{align}  \label{eq:2.F} 
N^{-1} \|\partial\la\|_\infty \le 1, \quad
N^{-1} r(F)^d \|\partial\la\|_\infty \| F\|_\infty \le 1,
\end{align}
(the constant $1$ in the right-hand side may be replaced
by any other constant).
\end{lem}

The proof is postponed to Section \ref{sec:4.1-C}.

\subsection{Refined one-block estimate}
\label{sec:3.2}

We now formulate the one-block estimate, that is, the error estimate for 
the replacement
of the microscopic function $\Om_2$ of order $O(1)$ (first four terms) and $O(K)$ 
(last term) by its ensemble average with mean $\rho=\bar\eta_x^\ell$ (see 
\eqref{eq:3.saeta} below) under the integral with respect to $P^{f_t}$ and 
also in $t\in [0,T]$.  As we pointed repeatedly, we need to 
prepare an error estimate strong enough for later use.  The classical one-block
estimate can be shown even with the diverging Glauber part as discussed in 
Remark \ref{Remark3.1} below, but it is not sufficient for our purpose.

In general for $\Om\in \mathcal{F}_0$, define the error function
in the replacement of $\t_x\Om$ by its ensemble mean at density
$\bar\eta_x^\ell$ as
$$
\hat{\Om}_x(\eta) \equiv \hat{\Om}_x^\ell(\eta) = \t_x \Om(\eta) - \lan \Om \ran(\bar\eta_x^\ell),
$$
where we denote $\lan \,\cdot\, \ran (\rho)$ for $\lan \,\cdot\, \ran_\rho
= E^{\nu_\rho}[\,\cdot\,]$,
$\rho\in [0,1]$.  Here a local sample average of $\eta$ over $\Lambda_{\ell,x}$
is defined by
\begin{align}  \label{eq:3.saeta}
\bar\eta_x^\ell \equiv \bar\eta_{\Lambda_{\ell,x}}
:= \ell_*^{-d} \sum_{y \in \Lambda_{\ell,x}} \eta_y,
\end{align}
recall \eqref{eq:3.1-P} for $\La_{\ell,x}$ and $\ell_* =2\ell+1\;(<N)$.

\begin{prop}  \label{One-block}
Let $a_{t,x}$, $t \in [0,T]$ and $x \in \T_N^d$, be 
deterministic coefficients which satisfy
\begin{equation} \label{BG:assn:atx}  
  |a_{t,x}| \leq M \quad \text{for all } t \in [0,T], \; x \in \T_N^d.
\end{equation}
We assume
\begin{equation} \label{3.26-A}
K^{1/2} \ell_*^{(d+2)/2} \le \de N
\end{equation}
for some $\de>0$ small enough (the choice of $\de$ is determined in
Lemma \ref{l:L10:6} in the proof of Proposition \ref{One-block} in Section 
\ref{sec:4.2-C})  and for all $N$ large enough.  Then, for every exponent
$a\in (0,1)$ (for \eqref{3.9-C}), there exists  $C=C_{T,a} >0$ such that for 
every $t\in [0,T]$ and $N$ large enough,
\begin{align} \label{One-block:est}
E \bigg[ \bigg |& \int_0^t  N^{-d}
\sum_{x \in \T_N^d} a_{s,x} \hat{\Om}_{x}(\eta^N(s)) ds\bigg | \bigg] \\
& \leq C M (\|\Om\|_\infty +1)\big(N^{-1}K^{1/2} \ell^{(d+2)/2} 
+ \ell^{-d} r(\Om)^d \big),
\notag
\end{align}
if $\Om\in \mathcal{F}_0$ satisfies
\begin{equation} \label{3.9-C}
r(\Om) \le \ell_*^a.
\end{equation}
Recall that $r(\Om)$ denotes the radius of the support of $\Om$ defined by
\eqref{eq:3rF}.
\end{prop}

Proposition \ref{One-block} can be applied for the specific term $\Om_2$ 
obtained in Lemma 
\ref{Lemma 3.1} to replace $\int_0^t E^{f_s}[\Om_2] \, ds$ by the time-integral
of its ensemble average with $\rho=\bar\eta_x^\ell$.  The following corollary
summarizes the error for this replacement.

\begin{cor}  \label{One-block-cor}
For $t\in [0,T]$, the integral $\int_0^t E^{f_s}[\Om_2] \, ds$ can be replaced by
\begin{align}  \label{eq:hatc-D}
\int_0^t E^{f_s} & \bigg[N^{-d} \sum_{x\in\T_N^d} \Big\{ -\dot{\la}(s,x/N)
   \bar{\eta}_{x}^\ell  \\
& \phantom{\sum_{x\in\Ga_N}} 
+ \frac12 \left(\partial\la(s,x/N),
        \widehat{c}(\bar{\eta}_{x}^\ell;F)\partial\la(s,x/N)\right)
        + K \si_G(\bar{\eta}_x^\ell;s,x/N)
           \Big\} \bigg] \, ds,  \notag
\end{align}
where $\widehat c(u;F)$ is defined in \eqref{eq:1.6-Q} and
\begin{align}  \label{eq:sigG}
\si_G(u;t,v) = \bigg( \frac{\lan c^+\ran_u}{\rho(t,v)}      
  -  \frac{\lan c^-\ran_u}{1-\rho(t,v)}\bigg)  \{ u-\rho(t,v)\}.
\end{align}
The error for this replacement is estimated by
\begin{align}  \label{3.Q-Om21}
Q_{N,\ell}^{\Om_2,1}(\la,F):= 
& C\big(N^{-1}K^{1/2}\ell^{(d+2)/2} + \ell^{-d}(1+r(F)^d)\big)  \\
& \times
    \Big( \|\dot{\la}\|_\infty+ \|\partial\la\|_\infty^2 (1+\vertiii{F}_{0,\infty})^2
        + \|\partial^2\la\|_\infty (1+\vertiii{F}_{1,\infty}) +K\Big),
        \notag
\end{align}
provided the function $F\in \mathcal{F}_0^d$ satisfies
\begin{equation} \label{3.12-C}
r(F) + r_c+1 \le \ell_*^a,
\end{equation}
for $a\in (0,1)$ (arbitrarily fixed),
where $r_c>0$ is the range of the rates $c_{x,y}$ and $c_x$ introduced in
Section \ref{Section 1.1}.
In \eqref{3.Q-Om21}, for each $\la_*>0$, the constant $C=C_{T,\la_*}$  can 
be taken uniformly for $\la=\la(t,v)$ such that $\|\la\|_\infty \le \la_*$.
\end{cor}

\begin{proof}
We apply Proposition \ref{One-block} taking each term of $\Om_2$ 
as $\Om$.  
First, note that the ensemble averages of the third and fourth terms of 
$\Om_2$ vanish:
$$
\lan c_{xy}(\eta_x-\eta_y)\ran_\rho=0, \quad \lan c_{xy}\pi_{xy}\Phi\ran_\rho=0
$$
for all $\Phi=\Phi(\eta)\in \mathcal{F}_0$. 

Then, \eqref{3.Q-Om21} is obtained by gathering all errors
given by \eqref{One-block:est} for each term of $\Om_2$.
Indeed, take $M=\|\dot{\la}\|_\infty$ for the first term of $\Om_2$, 
$M(\| \Om\|_\infty +1) \le C \|\partial\la\|_\infty^2 (1+\vertiii{F}_{0,\infty})^2$
for the second term, $M(\| \Om\|_\infty +1) \le C \|\partial^2\la\|_\infty$ 
for the third term, 
$M(\| \Om\|_\infty +1) \le C \|\partial^2\la\|_\infty(1+\vertiii{F}_{1,\infty})$
for the fourth term and $M(\| \Om\|_\infty +1)\le CK$ for the fifth term 
noting that $\rho(t,v)$ is uniformly away from $0$ and $1$:
$e^{-\la_*}/(e^{-\la_*}+1) \le \rho(t,v) \le e^{\la_*}/(e^{\la_*}+1)$
for $\la$ such that $\|\la\|_\infty \le \la_*$.

The condition \eqref{3.12-C} ensures \eqref{3.9-C} for each choice of $\Om$.
In \eqref{3.Q-Om21}, $r(F)^d$ needs to be $(r(F) + r_c+1)^d$, but
this difference is absorbed to the constant $C$ by changing it.
\end{proof}

The proof of Proposition \ref{One-block} is given in Section \ref{sec:4.2-C}
as a combination of a known estimate adjusted to our setting, which is formulated
in Lemma \ref{l:L10:6}, and the equivalence of ensembles with precise
convergence rate \eqref{eq:3EE}.
Proposition \ref{One-block} is similar to Theorem 1.4 of \cite{FvMST} but, 
differently from it, we don't
need the relative entropy on the right-hand side of the estimate 
\eqref{One-block:est}.

\begin{rem}  \label{Remark3.1}
A qualitative version of Proposition \ref{One-block} with $a_{t,x}=1$:
\begin{align} \label{eq:3.15-D}
\lim_{\ell\to\infty} \varlimsup_{N\to\infty}
E \bigg[ \bigg | \int_0^t  N^{-d}
\sum_{x \in \T_N^d} \hat{\Om}_{x}^\ell(\eta^N(s)) ds\bigg | \bigg] =0
\end{align}
can be shown by the classical argument.  Indeed, 
considering the relative entropy with respect to
the global equilibrium $\nu^N \equiv\nu_{1/2}^N$ (so we take $g=1$ in
\eqref{eq:3.16-D}), we have
\begin{align} \label{eq:3.16-D}
\frac{d}{dt} H(f_t| 1) \le -4N^2 \mathcal{D}_E(\sqrt{f_t}),
\end{align}
where $\mathcal{D}_E$ is the Dirichlet form associated with the
Kawasaki generator $L_E$:
$$
\mathcal{D}_E(f) = \frac14 \sum_{x,y\in \T_N^d}
E^{\nu^N}\Big[ c_{x,y}(\eta) \big( f(\eta^{x,y})-f(\eta)\big)^2 \Big]
\; \Big( \!\! = - (f,L_Ef)_{L^2(\mathcal{X}_N,\nu_N)}\Big).
$$
Since a classical one-block estimate (and also two-blocks estimate) was 
shown relying only on the estimate \eqref{eq:3.16-D}, we obtain 
\eqref{eq:3.15-D} even with a diverging Glauber part
in our generator $\mathcal{L}_N$; cf.\  Theorem 2.7
of \cite{F18} (only for $d=1$) or Theorem 3.1 of \cite{FUY}.

However, \eqref{eq:3.15-D} is not enough for our purpose, instead
we prepared the stronger error estimate \eqref{One-block:est} in
Proposition \ref{One-block} to show the quantitative result.

The estimate \eqref{eq:3.16-D} is standard.  For example, it is derived as
follows.  In the estimate (4.10) of \cite{F18}, taking $\nu_t =\nu^N$ and
$L=N^2 \mathcal{L}_N$,
the second term on the right-hand side vanishes and also, in the first term
$-\mathcal{D}(\sqrt{f_t};\nu^N)$, one can drop the contribution of the Glauber part 
since it is non-positive, and get $-4N^2 \mathcal{D}_E(\sqrt{f_t})$. 
This shows \eqref{eq:3.16-D}.
\end{rem}

\section{Gradient replacement}
\label{sec:3.3}

Sections \ref{sec:3.3} and \ref{sec:5} are highlights of the non-gradient 
and quantitative method.

The term  $\Om_1$ looks of the order $O(N)$.  We need the gradient replacement
to eliminate the diverging factor in this term, with proper error estimates;
see Theorem \ref{Theorem 3.2} and Lemma \ref{Lemma 3.4}.
Main estimate comes from Section \ref{sec:5}, where we state decay estimates
for the CLT variances.  In particular,  Theorem \ref{Theorem 5.1} and
Proposition \ref{prop:5.4}, taken from \cite{FGW}, are new and essential to 
show the quantitative result.  These were shown inspired by
the recent progress in the quantitative homogenization theory.

\subsection{Boltzmann-Gibbs principle for gradient replacement}

For the microscopic function $\Om_1$ which is given in Lemma \ref{Lemma 3.1}
and looks $O(N)$, we need the following
Boltzmann-Gibbs principle for the gradient replacement.
This is a refinement of Theorem 3.2 of \cite{FUY} adding the Glauber part and
giving the necessary error estimates.

\begin{thm}  \label{Theorem 3.2}
For every $T>0$, there exists  $C=C_T>0$  such that
\begin{align*}
    \int_0^t &  E^{f_s} \!  \Bigg[\Om_1 +  N^{1-d}
    \sum_{x\in \T_N^d} \left( D(\bar{\eta}_x^\ell) \partial\la(s,x/N),
        \ell_*^{-d} \t_x A_\ell \right)   \\ 
&  \qquad\qquad  -\frac{\b N^{-d}}{2}  
      \sum_{x\in\T_N^d} \left(\partial\la(s,x/N),
          R(\bar{\eta}_x^\ell;F) \partial\la(s,x/N)\right) \Bigg]  
               \, ds \\
&   \le 
               \frac{C}{\b}K + \b^2 Q_{N,\ell}^{(1)}(\la,F)
               +  \b Q_\ell^{(2)}(\la,F)
\end{align*}
for every $t\in [0,T]$ and $\b > 0$, where
\begin{equation}  \label{eq:3.Aell}
A_\ell = \{A_{\ell,i}(\eta)\}_{i=1}^d :=
   \sum_{b = \{x,y\} \in (\La(\ell))^*} (\eta_y-\eta_x) (y-x),
\end{equation}
denoting the box $\La_{\ell,0}$ with the center at $0$ by $\La(\ell)$, and
\begin{equation} \label{eq:3.R}
R(\rho;F) = \widehat{c}(\rho;F)  - \widehat{c}(\rho),
\end{equation}
recalling \eqref{1.3} and \eqref{eq:1.6-Q} for $\widehat{c}(\rho)$ and
$\widehat{c}(\rho;F)$, respectively.

The error $Q_{N,\ell}^{(1)} \equiv Q_{N,\ell}^{(1)}(\la,F)$  is estimated as
\begin{equation}  \label{3.Q1}
|Q_{N,\ell}^{(1)}| \le C N^{-1} \ell^{(2d+4)} \|\partial\la\|_\infty^3
(1+ \vertiii{F}_{0,\infty}^3).
\end{equation}
The other error $Q_\ell^{(2)} \equiv Q_\ell^{(2)}(\la,F)$ appears in the computation 
of the CLT variance (in Section \ref{sec:5}) and has a bound
\begin{equation}  \label{eq:3.26}
|Q_\ell^{(2)}| \le \|\partial\la\|_\infty^2 Q_\ell(F)
\end{equation}
with $Q_\ell(F)$ being bounded in Corollary \ref{Corollary 5.1} as
$$
Q_\ell(F) \le 
C (1+ r(F)^{2d}) (1+\vertiii{F}_{0,\infty}^2) \ell^{-1}+ C \ell^{-\a_1},
$$
for some $C>0$ and $\a_1>0$.
\end{thm}

Combining with the key lemma stated below, the proof of Theorem 
\ref{Theorem 3.2} heavily relies on a recent result 
(Theorem \ref{Theorem 5.1} and Corollary \ref{Corollary 5.1})
on the convergence rate of the CLT
variances.  This will be discussed in Section \ref{sec:5}.
In this section, we give the proof of Theorem \ref{Theorem 3.2}
assuming that Corollary \ref{Corollary 5.1} is shown.

Later in Lemma \ref{Lemma 3.4}, we will show that the term next to $\Om_1$
in the first displayed formula
in Theorem \ref{Theorem 3.2}, which still looks of diverging order $O(N)$
from the front factor $N^{1-d}$, can 
actually be rewritten into a term of $O(1)$.

\subsection{A key lemma in the setting of Glauber-Kawasaki dynamics}

We state a key lemma by which one can reduce a non-equilibrium problem 
into a static problem under the canonical equilibrium measure,
sometimes called Kipnis-Varadhan estimate or It\^o-Tanaka trick.
We introduce some notation.  For $\La \Subset \Z^d$ (i.e., finite
subset of $\Z^d$) and $m \in [0,|\La|]\cap\Z$,
\begin{align*}
&  \mathcal{X}_\La = \{0,1\}^\La, \quad
   \mathcal{X}_{\La,m} = \{ \eta\in \mathcal{X}_\La; \sum_{x\in\La} \eta_x =m \}, 
   \\
&  \nu_{\La,m} = \text{ uniform probability measure on } \mathcal{X}_{\La,m},
\end{align*}
and  $\lan\,\cdot\,\ran_{\La,m}$
denotes the expectation with respect to  $\nu_{\La,m}$. Then,
we have a result similar to Lemmas 3.3 (the case $M(\rho)=0$)
and 6.1 of \cite{FUY}, but the last term
$\frac{C}\b$ in their estimates is replaced by $\frac{C}\b K$ due to the
contribution from the Glauber part.  We also show an estimate
for the error term instead of taking the limit in $N$ as in \cite{FUY}.

Given a positive integer  $\ell$ (recall $\ell_* =2\ell+1<N$ as stated below 
\eqref{eq:3.saeta}),  we  write  $G_{\zeta} = G_{\zeta}(\xi)$  for a function  
$G = G(\eta)$  when we consider $G$ as a function of  $\xi = \eta|_{\La(\ell)}$ 
regarding $\zeta = \eta|_{\La(\ell)^c}$ as a parameter.  The following lemma
will be used for the proof of Theorem \ref{Theorem 3.2} taking $n=d$
and for the proof of Lemma \ref{Lemma 3.4} (given in Section \ref{sec:4.4-C})
 taking $n=1$.

\begin{lem}  \label{Lemma 3.3}
Let  $n\in \N$, $J(t,v) = \{J_i(t,v)\}_{i=1}^n \in C^\infty([0,T]\times\T^d, \R^n)$, 
$G(\eta) = \{ G_i(\eta)\}_{i=1}^n\in \mathcal{F}_0^n$ and  $M(\rho) = 
\{ M_{ij}(\rho)\}_{1\le i,j \le n} \in C([0,1],\R^n \otimes \R^n)$. 
Suppose that $\lan G_\zeta \ran_{\La(\ell),m} = 0$ for every  
$\zeta\in \mathcal{X}_{\La(\ell)^c}$ and $m: 0 \le m \le \ell_*^d$.  
Then, there exists  $C=C_T>0$
such that for every $t\in [0,T]$ and $\b > 0$
\begin{align*}
 \int_0^t & E^{f_s} \bigg[ N^{1-d}
  \sum_{x\in\T_N^d} J(s,x/N) \cdot \t_x G(\eta)
   - \b N^{-d} \sum_{x\in\T_N^d} J(s,x/N) \cdot M(\bar{\eta}_x^\ell)
    J(s,x/N)  \bigg] \, ds \\
&\le \b\, t \sup_{|\th| \le \|J\|_\infty } \sup_{m,\zeta\in \mathcal{X}_{\La(\ell)^c}}
  \Big[ d \,\ell_*^d \big\lan \th\cdot G_{\zeta} \,
    (-L_{\La(\ell),\zeta})^{-1} \th\cdot G_{\zeta} \big\ran_{\La(\ell),m}  
      - \th \cdot M(m/\ell_*^d)\th \Big] \\
     & \qquad + \frac{C}{\b}K + \b^2 Q_{N,\ell}^{(1)}(J,G),
\end{align*}
where  $L_{\La(\ell),\zeta}$  is the operator defined by 
\begin{align}  \label{eq:LLa}
L_{\La(\ell),\zeta} f(\xi) = \sum_{b \in (\La(\ell))^*} c_b(\xi\cdot\zeta) \pi_b f(\xi),
\end{align}
for $\xi\in \mathcal{X}_{\La(\ell)}$, $\zeta\in \mathcal{X}_{\La(\ell)^c}$,
$f$ is a function on $\mathcal{X}_{\La(\ell)}$ and $\xi\cdot\zeta$ denotes
the configuration $\eta\in \mathcal{X}$ such that
$\eta|_{\La(\ell)}=\xi$ and $\eta|_{\La(\ell)^c}=\zeta$.  The error term 
$Q_{N,\ell}^{(1)}\equiv Q_{N,\ell}^{(1)}(J,G)$ has a bound
\begin{align}  \label{eq:3QNell}
|Q_{N,\ell}^{(1)}| 
\le C N^{-1} \ell^{(2d+4)} \| G\|_\infty^3 \| J\|_\infty^3.
\end{align}
\end{lem}

Lemma \ref{Lemma 3.3} is a combination of essentially known results,
providing fine error estimates, so the proof is postponed to 
Section \ref{sec:4.3-C}.

\subsection{Proof of Theorem \ref{Theorem 3.2}}

When $K=0$ (i.e.\ no Glauber part), the proof of Theorem \ref{Theorem 3.2}
without error estimates was given in Sections 4--6 of \cite{FUY}.
Since only the Kawasaki part appears in the first term of
the estimate in Lemma \ref{Lemma 3.3}
above, the arguments in Sections 4 and 5 of \cite{FUY} are applicable
in our setting; see Section \ref{sec:5} below for more details and refined estimates.

To rewrite the function in the expectation in  Theorem \ref{Theorem 3.2},
we introduce the following three $\R^d$-valued functions  $A_\ell$, $B_\ell$,
and $H_\ell$ of $\eta$; in fact, $A_\ell$ is the same as in \eqref{eq:3.Aell}.
We denote $\Psi_i = \eta_{e_i} - \eta_0$ and  
$W_{x,y}(\eta) = c_{x,y}(\eta)(\eta_y - \eta_x)$ for the microscopic current,
where  $e_i \in \Z^d$ stands for the unit vector in the $i$-th direction.
Recalling $\xi = \eta|_{\La(\ell)}$ and $\zeta = \eta|_{\La(\ell)^c}$, 
\begin{align}  \label{eq-3.A}
 A_\ell(\eta) \equiv A_{\ell}(\xi) 
&  = \frac12 \sum_{x,y\in \La(\ell):|x-y|=1}
        (\xi_y-\xi_x)  (y-x)   
       \\  
&  = \sum_{i=1}^d \sum_{x: \t_x e^*_i \in (\La(\ell))^*}
     \t_x\Psi_i(\xi) \; e_i,   \notag
       \\    \label{eq-3.B}
 B_\ell(\eta) \equiv B_{\ell,\zeta}(\xi) 
& = \frac12 \sum_{x,y\in \La(\ell):|x-y|=1}
     W_{x,y}(\eta) (y-x)  
       \\  
& = -L_{\La(\ell),\zeta} \, \Big( \sum_{x\in\La(\ell)} x\xi_x\Big),   \notag
\intertext{and}  \label{eq-3.H}
 H_\ell(\eta) \equiv H_{\ell,\zeta,F}(\xi) 
& = \sum_{x\in \La(\ell-n)} \t_x(L_E F)(\xi\cdot\zeta) 
       \\  
& = L_{\La(\ell),\zeta} \,
  \Big( \sum_{x\in \La(\ell-n)} \t_x  F\Big)(\xi),  \notag
\end{align}
for a given function $F = (F_i)_{i=1}^d \in \mathcal{F}_0^d$ which is
$\mathcal{F}_{\La(n-1)}$-measurable with $n=r(F)+1$. 
Recall \eqref{eq:LLa} for the operator $L_{\La(\ell),\zeta}$ and
\eqref{1.1} (or its extension on $\Z^d$) for $L_E$.

\begin{proof}[Proof of Theorem \ref{Theorem 3.2}]   
We apply Lemma \ref{Lemma 3.3} with
$n = d, J(t,v) = \partial \la(t,v)$  and
\begin{align}  \label{eq:3-G}
\begin{aligned}
G(\eta) & = 
  \ell_*^{-d} \left\{D(\bar{\eta}_\ell) A_\ell - B_\ell + H_\ell \right\},  \\
M(\rho) & = \frac12\left\{ \widehat{c}(\rho;F) - \widehat{c}(\rho)\right\}
\; \Big(\!\!= \frac12 R(\rho;F)\Big).
\end{aligned}
\end{align}
Note that, under the above choice, the expressions under
$\int_0^t E^{f_s}[\cdots]ds$ in Theorem \ref{Theorem 3.2} and
Lemma \ref{Lemma 3.3} coincide, by recalling that $\Om_1$ is defined 
by \eqref{eq:Om1} with $\Om_{x,y}$ in \eqref{eq:Omxy}.

Then, we have $\| G\|_\infty \le C (1+\vertiii{F}_{0,\infty})$.  Indeed, the function
$D=D(\rho)$ is bounded from \eqref{eq:1.9},
$\| A_\ell\|_\infty \le C \ell_*^d$, $\| B_\ell\|_\infty \le C \ell_*^d$
since $W_{x,y}$ is bounded, and from
$$
H_\ell = \sum_{\{x,y\}\in \La(\ell)^*} c_{xy}(\xi\cdot\zeta) \pi_{xy}
\Big( \sum_{z\in \La(\ell-n)} \t_z F\Big),
$$
with $n=r(F)+1$, we have
$$
\| H_\ell\|_\infty \le C \ell_*^d r(F)^d \| F\|_\infty = 
C \ell_*^d \vertiii{F}_{0,\infty}. 
$$
Therefore, by \eqref{eq:3QNell}, the error $Q_{N,\ell}^{(1)}(\la,F) 
:= Q_{N,\ell}^{(1)}(J,G)$ with $J$ and $G$ determined as above has the estimate 
\eqref{3.Q1}.

Thus, the proof of Theorem \ref{Theorem 3.2} is concluded 
with the help of Corollary \ref{Corollary 5.1}, by observing that
\begin{align}  \label{eq:4.11-P}
& Q_\ell^{(2)}(\|J\|_\infty,F) \\
& \;  :=
 \sup_{|\th| \le \|J\|_\infty } \sup_{m,\zeta\in \mathcal{X}_{\La(\ell)^c}} 
\Big\{ \ell_*^d \lan \th\cdot G_{\zeta}, 
    (-L_{\La(\ell),\zeta})^{-1} \th\cdot G_{\zeta} \ran_{\La(\ell),m}
   - \th \cdot M(m/\ell_*^d)\th \Big\}  \notag
\end{align}
appearing from Lemma  \ref{Lemma 3.3}
is bounded by $\| J\|_\infty^2 Q_\ell(F)$ with $Q_\ell(F)$ given in 
Corollary \ref{Corollary 5.1}.  Taking $J=\partial\la$ and
denoting $Q_\ell^{(2)}(\la,F) := Q_\ell^{(2)}(\|\partial\la\|_\infty,F)$, 
we obtain \eqref{eq:3.26} in Theorem \ref{Theorem 3.2}.
\end{proof}

\subsection{Error estimate to rewrite $O(N)$-looking term in 
Theorem \ref{Theorem 3.2}}

Next, we show the following lemma, in which we evaluate the error to rewrite
the term appearing in Theorem \ref{Theorem 3.2} from $\Om_1$ into another
term of order $O(1)$.

\begin{lem} \label{Lemma 3.4}
Set
\begin{align*}
Q_{N,\ell}^{(3)}(\la) =
\int_0^t E^{f_s} &  \bigg[
  N^{1-d} \sum_{x\in \T_N^d} \bigg( D(\bar{\eta}_x^\ell) \partial \la(s,x/N),
    \ell_*^{-d} \t_x A_\ell(\eta) \bigg)    \\ 
& \qquad  + N^{-d} \sum_{x\in\T_N^d} \sum_{i,j=1}^d \partial_{v_i}\partial_{v_j}
\la(s,x/N)
     P_{ij}(\bar{\eta}_x^\ell) \bigg] \, ds,
\end{align*}
where $P(\rho) \equiv \{P_{ij}(\rho)\}_{1\le i,j \le d} := \int_0^\rho D(\rho')d\rho'$,
$\rho \in [0,1]$.
Then, there exists $C=C_T>0$ such that for every  $t\in [0,T]$ and $\b>0$
the following upper bound holds for $Q_{N,\ell}^{(3)} \equiv
Q_{N,\ell}^{(3)}(\la)$:
$$
Q_{N,\ell}^{(3)}\le C (\ell^{-1}\|\partial^2\la\|_\infty +N^{-1}\|\partial^3\la\|_\infty)
 +C\b \ell^{-1} \|\partial\la\|_\infty^2   + 
\frac{C}\b K + C\b^2 N^{-1}\ell^{2d-2}  \|\partial\la\|_\infty^3. 
$$
\end{lem}

The proof is given in Section \ref{sec:4.4-C} by carefully reviewing the
proof of Lemma 3.4 of \cite{FUY} and applying Lemma \ref{Lemma 3.3},
i.e.\ the key lemma in the setting of Glauber-Kawasaki dynamics.

\section{CLT variances and their error estimates}
\label{sec:5}

We now present the error estimates for the variational quantity 
such as $Q_\ell^{(2)}(\|J\|_\infty,F)$ in \eqref{eq:4.11-P} appearing from 
Lemma \ref{Lemma 3.3}.  This contains the so-called CLT
(central limit theorem) variance defined by \eqref{eq:CLT-variance}.
This quantity is defined by means 
of the localized Kawasaki generator $L_{\La(\ell),\zeta}$ given in
\eqref{eq:LLa} and defined on the box $\La(\ell) = \La_{\ell,0}$ centered at $0$,
where we denoted $\zeta=\eta|_{\La(\ell)^c}$ and $\xi=\eta|_{\La(\ell)}$
for $\eta\in \mathcal{X}$.  In particular, the Glauber part plays no role and the
arguments in Sections 4 (characterization of closed forms) and 5 (computation
of variances) of \cite{FUY} are valid as is in our setting.  

However, these are not
sufficient for our purpose to derive appropriate decay rates for the error terms.
In particular, for the CLT variance of $A_\ell$, a qualitative convergence 
result as $\ell\to\infty$ (see \eqref{4.1.1} below)
is known and shown based on the 
characterization of the closed forms, initiated by Varadhan \cite{11} 
for the Ginzburg-Landau type lattice model.  But, this is weak for our purpose and, 
to fill the gap,
Theorem \ref{Theorem 5.1} below was shown in \cite{FGW}
by a new technique inspired by the method of the quantitative homogenization.
This theorem provides an appropriate decay rate of  the CLT variance of $A_\ell$
and, in particular, the characterization of the closed forms was not used
for the proof of this theorem.

The CLT variance of $B_\ell-H_\ell$ and the CLT covariance between $A_\ell$ 
and $B_\ell-H_\ell$ are calculated in Proposition \ref{Proposition 5.1}, providing
a refined error estimate.  Theorem \ref{Theorem 5.1} and 
Proposition \ref{Proposition 5.1} are summarized into Corollary \ref{Corollary 5.1},
which was used for the proof of Theorem \ref{Theorem 3.2}.
Proposition \ref{prop:5.4}, which is taken
from \cite{FGW} and gives a decay rate for a minimizing sequence for the
variational formula \eqref{1.3}, also plays a key role in our study to derive
the quantitative result in Theorem \ref{Corollary 2.1}.

\subsection{Definition of CLT covariances and variances}

Motivated by the variational quantity appearing in Lemma \ref{Lemma 3.3},
for  $f,g : \mathcal{X}_{\La(\ell),m} \to \R$  which satisfy 
$\lan f \ran_{\La(\ell),m} = \lan g \ran_{\La(\ell),m} = 0,$
we define
\begin{align}  \label{eq:CLT-variance}
\begin{aligned}
& \Delta_{\ell,m,\zeta} (f,g) = \lan f (-L_{\La(\ell),\zeta})^{-1} g 
\ran_{\La(\ell),m},   \\
& \Delta_{\ell,m,\zeta} (f) = \Delta_{\ell,m,\zeta} (f,f),
\end{aligned}
\end{align}
where $m\in [0,\ell_*^d]\cap\Z$  and  $\zeta\in\mathcal{X}_{\La(\ell)^c}$.
Recall three $\R^d$-valued functions  $A_\ell, B_\ell$  and
$H_\ell$, which are defined by \eqref{eq-3.A}, \eqref{eq-3.B}, and \eqref{eq-3.H},
respectively, and $n=r(F)+1$ in $H_\ell$.

\subsection{CLT variance of $A_\ell$}

For the CLT variance of $A_\ell$, it was shown in Theorem 5.1 of \cite{FUY}
that
\begin{align} \label{4.1.1}
\lim_{\ell, m \to \infty, m/\ell_*^d \to \rho} 
\ell_*^{-d} \Delta_{\ell,m,\zeta} (\th\cdot A_\ell)
   = 2 (\th\cdot \widehat{c}^{-1}(\rho)\th) \chi^2(\rho),
\end{align}
uniformly in $\rho\in [0,1]$ and $\eta\in \mathcal{X}$,
where $\zeta=\eta|_{\La(\ell)^c}$. 
This is the key in the gradient replacement and shown based on a
characterization of closed forms given in Corollary
4.1 of \cite{FUY}.  However, this is not sufficient for our purpose.
Then, recently more detailed analysis has been developed 
by applying the method in the quantitative
homogenization theory and, as a by-product, 
a decay rate in the limit \eqref{4.1.1} was obtained as in the following theorem;
see (7.12) in \cite{FGW}.

\begin{thm}  \label{Theorem 5.1} {\rm (\cite{FGW})}
For $\th \in \R^d: |\th|=1$, $m\in [0,\ell_*^d]\cap\Z$
and $\zeta \in \mathcal{X}_{\La(\ell)^c}$, set
$$
Q_\ell^{(4)}(\th,m,\zeta) =  \ell_*^{-d} \Delta_{\ell,m,\zeta} (\th\cdot A_\ell)
   - 2 (\th\cdot \widehat{c}^{-1}(m/\ell_*^d)\th) \chi^2(m/\ell_*^d).
$$
Then, there exist $C>0$ and $\a_1>0$ such that
\begin{align}  \label{eq:Q(4)}
|Q_\ell^{(4)}(\th,m,\zeta)|\le C\ell^{-\a_1}
\end{align}
uniformly in  $\th \in \R^d: |\th|=1$, $m\in [0,\ell_*^d]\cap\Z$ and
$\zeta \in \mathcal{X}_{\La(\ell)^c}$.
\end{thm}

\subsection{CLT variance of $B_\ell-H_\ell$ and covariance of
$A_\ell$ and $B_\ell-H_\ell$}
\label{sec:5.3-D}

We calculate the CLT variance of $B_\ell-H_\ell$ and the CLT covariance 
between $A_\ell$ and $B_\ell-H_\ell$, and provide refined error estimates
for them.  These are put together in Corollary \ref{Corollary 5.1} below.

\begin{prop} \label{Proposition 5.1}
For  $F\in \mathcal{F}_0^d$ and $\th, \tilde{\th} \in \R^d: |\th| = |\tilde\th|=1$,
set
\begin{align*}
Q_\ell^{(5)}(F) & \equiv Q_\ell^{(5)}(F;\th,m,\zeta) \\
& :=   \ell_*^{-d} \Delta_{\ell,m,\zeta} (\th\cdot(B_{\ell,\zeta}-H_{\ell,\zeta,F}))
   - \frac12 \th\cdot \widehat{c}(m/\ell_*^d;F)\th,  \\  
Q_\ell^{(6)}(F) & \equiv Q_\ell^{(6)}(F;\th,\tilde\th,m,\zeta) \\
& :=  \ell_*^{-d} \Delta_{\ell,m,\zeta} (\tilde{\th}\cdot A_{\ell}, 
    \th\cdot(B_{\ell,\zeta}-H_{\ell,\zeta,F}))
   -  (\th\cdot \tilde{\th}) \chi(m/\ell_*^d).
\end{align*}
Then, we have
\begin{align}
&  |Q_\ell^{(5)}(F)| \le  C r(F)^d (1+\vertiii{F}_{0,\infty}^2) \ell^{-1},  \label{5.1} 
       \\  
&  |Q_\ell^{(6)}(F)| \le C r(F)^{2d} \| F\|_\infty \ell^{-d} +C\ell^{-1},
    \label{5.2}
\end{align}
uniformly in  $\th, \tilde{\th} \in \R^d: |\th| = |\tilde\th|=1$,
$m\in [0,\ell_*^d]\cap\Z$  and  $\zeta\in\mathcal{X}_{\La(\ell)^c}$, 
where $\th\cdot\tilde{\th}$  stands for the inner product of  $\th$
and  $\tilde{\th}\in \R^d$  which has been denoted by
$(\th,\tilde{\th})$  in the previous sections.  Recall \eqref{eq:3.2-P}
for $\vertiii{F}_{0,\infty}$.
\end{prop}

The proof of Proposition \ref{Proposition 5.1} is given based on known results
for Kawasaki dynamics in \cite{FUY}, providing careful error estimates.  
The proof is simpler than that for $A_\ell$, since
$B_{\ell,\zeta}-H_{\ell,\zeta,F}$ is represented as $L_{\La(\ell),\zeta} G$
for some function $G$ and this makes a cancellation with
$(-L_{\La(\ell),\zeta})^{-1}$ in the CLT variance and the CLT covariance.
In particular, the
quantitative result from the recent homogenization theory is unnecessary 
for the proof of this proposition.  It is postponed to Section \ref{sec:9.5}.

\subsection{CLT variance of $D A_\ell - (B_\ell-H_\ell)$}

The following corollary is shown by combining Theorem \ref{Theorem 5.1} and
Proposition \ref{Proposition 5.1}, and gives a quantitative refinement of
Corollary 5.1 of \cite{FUY}.

\begin{cor}  \label{Corollary 5.1}
Taking the supremum  over all $\th \in \R^d: |\th|=1$,
$m \in [0,\ell_*^d]\cap\Z$ and $\zeta \in \mathcal{X}_{\La(\ell)^c}$, set
\begin{align*}
Q_\ell(F) := \sup_{|\th|=1, m, \zeta}
\bigg| \ell_*^{-d} & \Delta_{\ell,m,\zeta} \big(\th\cdot 
   \{ D(m/\ell_*^d)A_{\ell} -(B_{\ell,\zeta}-H_{\ell,\zeta,F})\}\big) \\
&   
- \frac12 \th\cdot R(m/\ell_*^d;F)\th \bigg|, 
\end{align*}
where  $D(\rho)= \widehat{c}(\rho)/2\chi(\rho)$  is the diffusion matrix
defined in \eqref{1.5} and
$R(\rho;F) = \widehat{c}(\rho;F) - \widehat{c}(\rho)$ is the error defined in 
\eqref{eq:3.R} for $\widehat{c}(\rho)$, which is determined by
\eqref{1.3}.  Then, $Q_\ell(F)$ is estimated as
$$
Q_\ell(F) \le 
C (1+ r(F)^{2d}) (1+\vertiii{F}_{0,\infty}^2) \ell^{-1}+ C \ell^{-\a_1},
$$
for some $C>0$ and $\a_1>0$.
\end{cor}

\begin{proof}
Omitting $m/\ell_*^d$, we simply write $D, \widehat c,
\widehat c(F)$, $R(F)$ and $\chi$ for $D(m/\ell_*^d), \widehat c(m/\ell_*^d),
\widehat c(m/\ell_*^d;F)$, $R(m/\ell_*^d;F)$ and $\chi(m/\ell_*^d)$, respectively.
Then, noting that $D$ is symmetric, we have
\begin{align*}
\ell_*^{-d} & \Delta_{\ell,m,\zeta} \big(\th\cdot 
   \{ DA_{\ell} -(B_{\ell,\zeta}-H_{\ell,\zeta,F})\}\big) \\
& = \ell_*^{-d} \Big\{ \Delta_{\ell,m,\zeta} \big(D\th\cdot A_\ell \big)
-2 \Delta_{\ell,m,\zeta} \big(D\th\cdot A_\ell,
 \th\cdot (B_{\ell,\zeta}-H_{\ell,\zeta,F})  \big)  \\
& \hskip 40mm
+ \Delta_{\ell,m,\zeta} \big(\th\cdot (B_{\ell,\zeta}-H_{\ell,\zeta,F}) \big) \Big\}
\\
& = \big\{Q_\ell^{(4)}(D\th, m,\zeta) + 2 \big(D\th\cdot \widehat c^{-1}D\th\big) 
  \, \chi^2\big\}  - 2 \big\{ Q_\ell^{(6)}(F;\th,D\th, m,\zeta) 
  + \big(\th\cdot D \th\big) \, \chi
\big\} \\
&  \quad + \big\{ Q_\ell^{(5)}(F;\th, m,\zeta) + \frac12 \big(\th\cdot \widehat c(F)
\th\big)\big\}.
\end{align*}
However, by $D=  \widehat c/2\chi$ and $\widehat c(F)- \widehat c
=R(F)$, the term except for $Q^{(4)}$--$Q^{(6)}$ on the right-hand side
is rewritten as
\begin{align*}
 2 \big(D\th\cdot \widehat c^{-1}D\th\big) 
  \, \chi^2 -2 \big(\th\cdot D\th\big) \, \chi
+ \frac12 \big(\th\cdot \widehat c(F)\th\big)
= \frac12 \th \cdot R(F)\th.
\end{align*}
Therefore, the bound for $Q_\ell(F)$ follows from \eqref{eq:Q(4)},
\eqref{5.1} and \eqref{5.2} by changing the constant $C>0$ if necessary.
\end{proof}

\subsection{Minimizing sequence $\{\Phi_n\}$ for $\widehat{c}(\rho)$ and
its estimate}

The following proposition is shown from Lemma 4.4 and Proposition 6.9 of 
\cite{FGW} (see also Theorem 1.2 and Section 7 of \cite{FGW}).
Recall \eqref{eq:3rF} for $r(F)$, the variational formula \eqref{1.3}
for $\widehat{c}(\rho)$ and \eqref{eq:3.R} for $R(\rho;F)$.

\begin{prop}  {\rm (\cite{FGW})}\label{prop:5.4}
There exists a sequence of local functions $\Phi_n=\Phi_n(\eta)\in 
\mathcal{F}_0^d$ on $\mathcal{X}$  such that
\begin{equation}  \label{eq:Fn}
r(\Phi_n)\le n, \quad \sup_{\rho\in [0,1]} \| R(\rho; \Phi_n) \| \le C_2 n^{-\a_2},
\quad \|\Phi_n\|_\infty \le C_2 n^2 \log n,
\end{equation}
for some $C_2>0$ and $\a_2>0$.
\end{prop}

This proposition is a refinement of Lemma 2.1 of \cite{FUY}:
$$
\inf_{F\in {\mathcal{F}}_0^d}~\sup_{\rho\in [0,1]} \Vert R(\rho;F)\Vert = 0,
$$
which is uniform in $\rho$.  Proposition \ref{prop:5.4}
provides the decay rate in $n$ of $R(\rho;F)$ approximated by $F$ 
of support size at most $n$: $r(F)\le n$.

\section{Bound on $h_N(t)$ combining with large deviation estimate}
\label{section:6}

In this section, we first summarize the estimates obtained in
Sections \ref{Section 3}--\ref{sec:5} in the form of Lemma \ref{Lemma 3.2}.
Then, we apply the entropy inequality and a large deviation type upper bound 
under $\psi_t$ with fine error estimates given in Lemma \ref{Theorem 3.3}.
This shows Proposition \ref{Theorem 2.1}.
One can observe that the main term called $g(t) =g_{\de,K}(t)$ in the estimate
in this proposition can be dropped if we determine the leading term of $\psi_t$
according to the hydrodynamic equation \eqref{1.6} and take $\de>0$ small; 
see Lemma \ref{Lem:2.3}.
The equation \eqref{1.6} is used only for this lemma.  

\subsection{Consequence of the results of Sections \ref{Section 3},
 \ref{sec:3.3} and \ref{sec:5}}

Lemma \ref{Lemma 3.1} (calculation of $\partial_t h_N(t)$),
Corollary \ref{One-block-cor} (refined one-block estimate)
and Theorem \ref{Theorem 3.2} (gradient
replacement, relying on the estimates for the CLT variances in Section
\ref{sec:5}) are summarized in the following lemma.
Recall $h_N(t) = N^{-d} H(f_t|\psi_t)$ in \eqref{2.2} and 
$\psi_t=\psi_{\la(t,\cdot),F}$ in \eqref{eq:2psit}.
The constant $C=C_T>0$ may depend on $T$.

\begin{lem}  \label{Lemma 3.2} 
Assume $\la \in C^{1,3}([0,T]\times\T^d)$, $F \in\mathcal{F}_0^d$
and the conditions \eqref{eq:2.F}, \eqref{3.26-A} and \eqref{3.12-C}.
Then, setting $\rho(t,v) = \bar\rho(\la(t,v))$, we have
\begin{align}  \label{eq:6.1-h}
h_N(t) &  \le  h_N(0)+ \int_0^t  E^{f_s}[W_s]\,ds + 
Q_{N,\ell,\b,K}^{\Om_1}(\la,F)   \\
 &  + C (\b+1) \,\|\partial\la\|_\infty^2 \, \sup_{\rho} \|R(\rho; F)\|
    +    Q_N^{En}(\la,F)+ Q_{N,\ell}^{\Om_2}(\la,F),   \notag
\end{align}
for every $\b>0$ and $1\le \ell \le N/2$,
where  $\| R \|$  denotes the operator norm of matrix  $R$
(recall \eqref{eq:3.R} for $R(\rho;F)$), $Q_N^{En}(\la,F)$ is as in 
Lemma \ref{Lemma 3.1} (multiplied by $T$) and
\begin{align}  \label{3.QOM1}
Q_{N,\ell,\b,K}^{\Om_1}(\la,F) := \frac{C}{\b}K+\b^2 Q_{N,\ell}^{(1)}(\la,F)
 +\b Q_\ell^{(2)}(\la,F)+Q_{N,\ell}^{(3)} (\la),
\end{align}
see Theorem \ref{Theorem 3.2} for $Q_{N,\ell}^{(1)}$, $Q_\ell^{(2)}$ 
and Lemma \ref{Lemma 3.4} for $Q_{N,\ell}^{(3)}$.
The function $W_t$ is defined by
\begin{align*}
W_t(\eta) = & -  N^{-d}  \sum_{x\in\T_N^d} \dot{\la}(t,x/N)
   \{ \bar{\eta}_x^\ell - \rho(t,x/N)\}   
       \\  
 &  +  N^{-d} \sum_{x\in\T_N^d} \Tr \left(\partial^2\la(t,x/N)
   \{ P(\bar{\eta}_x^\ell) - P(\rho(t,x/N))\}\right)  
       \\  
 &  +  \frac{N^{-d}}{2}  \sum_{x\in\T_N^d} \left(\partial\la(t,x/N),
        \{\widehat{c}(\bar{\eta}_x^\ell) - \widehat{c}(\rho(t,x/N))\}
            \partial\la(t,x/N)\right)\\
            & + N^{-d} K \sum_{x\in\T_N^d} \si_G(\bar{\eta}_x^\ell;t,x/N).
\end{align*}
Recall \eqref{eq:sigG} for the function $\si_G=\si_G(u;t,v)$.
The error $Q_{N,\ell}^{\Om_2}(\la,F)$ is estimated as
\begin{align}  \label{3.QOM2}
|Q_{N,\ell}^{\Om_2}(\la,F)| \le & C\big(N^{-1}K^{1/2}\ell^{(d+2)/2} 
+ \ell^{-d}(1+r(F)^d)\big) \\
& \quad \times
    \Big( \|\dot{\la}\|_\infty+ \|\partial\la\|_\infty^2 (1+\vertiii{F}_{0,\infty})^2
        + \|\partial^2\la\|_\infty (1+\vertiii{F}_{1,\infty}) +K \Big)     \notag \\
& + C N^{-1}(\|\partial^3\la\|_\infty
+ \|\partial^2\la\|_\infty\|\partial\la\|_\infty
+ \|\partial\la\|_\infty^2 \|\partial \la \|_\infty).  \notag
\end{align}
In the above estimate, for each $\la_*>0$, the constant $C=C_{\la_*}$  can 
be taken uniformly for $\la=\la(t,v)$ such that $\|\la\|_\infty \le \la_*$.
\end{lem}

When we apply this lemma later in Lemma \ref{Lem:2.3}, we take
$\la(t,v) = \bar\la(\rho_K(t,v))$ with a solution $\rho_K$ of the
hydrodynamic equation \eqref{1.6}.
Under the condition \eqref{eq:1.12} for the initial value $\rho_0$, by
Lemma \ref{lem:max} (the comparison theorem for $\rho_K$)
and recalling \eqref{eq:2.4}, $\la_*$ can be taken to be uniform in $K$.

\begin{proof}
The proof is similar to Lemma 3.2 of \cite{FUY}, but compared to it,
we have the Glauber part and also
require to derive an appropriate error estimate.  

First, we integrate the estimate for $\partial_t h_N(t)$ shown in Lemma
\ref{Lemma 3.1} on the time integral $[0,t]$.  Then, for $\Om_1$, we apply 
Theorem \ref{Theorem 3.2} (with an error $\frac{C}{\b}K+\b^2 Q_{N,\ell}^{(1)}
 +\b Q_\ell^{(2)}$) and Lemma \ref{Lemma 3.4}
 (with an error $Q_{N,\ell}^{(3)}$), and obtain a 
replacement of $\int_0^t E^{f_s}[ \Om_1] \, ds$
by the following integral:
\begin{align*}
\int_0^t &  \bigg\{  E^{f_s} \bigg[ N^{-d}
\sum_{x\in\T_N^d} \Tr \left( \partial^2\la(s,x/N) P(\bar{\eta}_{x}^\ell)\right)
     \bigg]   \\
  & \quad   +  E^{f_s} \bigg[  \frac{\b N^{-d}}{2}
\sum_{x\in\T_N^d} \left(\partial\la(s,x/N), 
   R(\bar{\eta}_{x}^\ell;F)\partial\la(s,x/N)
     \right) \bigg]  \bigg\} \, ds
\end{align*}
with an error $Q_{N,\ell,\b,K}^{\Om_1}(\la,F)$ defined in \eqref{3.QOM1}.
The second term in the above formula is bounded by 
$C \b\,\|\partial\la\|_\infty^2 \, \sup_{\rho} \|R(\rho; F)\|$ and this is counted
on the right-hand side of \eqref{eq:6.1-h} for $h_N(t)$.

For $\Om_2$, by Corollary \ref{One-block-cor}, $\int_0^t E^{f_s}[\Om_2] \, ds$ 
can be replaced by the expression in \eqref{eq:hatc-D}:
\begin{align*}  
\int_0^t E^{f_s} & \bigg[N^{-d} \sum_{x\in\T_N^d} \Big\{ -\dot{\la}(s,x/N)
   \bar{\eta}_{x}^\ell  \\
& \phantom{\sum_{x\in\Ga_N}} 
+ \frac12 \left(\partial\la(s,x/N),
        \widehat{c}(\bar{\eta}_{x}^\ell;F)\partial\la(s,x/N)\right)
        + K \si_G(\bar{\eta}_x^\ell;s,x/N)
           \Big\} \bigg] \, ds   \notag
\end{align*}
with the error $Q_{N,\ell}^{\Om_2,1}(\la,F)$ defined in \eqref{3.Q-Om21}.
This gives the first part
in the error estimate of $|Q_{N,\ell}^{\Om_2}(\la,F)|$.
Note that $\widehat{c}(\bar{\eta}_{x}^\ell;F)$ in the above expression
(i.e.\ \eqref{eq:hatc-D}) can be
replaced by $\widehat{c}(\bar{\eta}_{x}^\ell)$ with an error 
$C \,\|\partial\la\|_\infty^2 \, \sup_{\rho} \|R(\rho; F)\|$.

The reason that we have the second part in the error estimate
of $|Q_{N,\ell}^{\Om_2}(\la,F)|$ in \eqref{3.QOM2} is as follows.  
Recalling that $\rho(t,v) = \bar{\rho}(\la(t,v))$, we have the identity
\begin{align}  \label{3.IBP}
-\int_{\T^d} \Tr \big(\partial^2\la(t,v) P(\rho(t,v))\big) dv  
=  \frac1{2} \int_{\T^d} \big(\partial\la(t,v),\widehat{c}(\rho(t,v))
 \partial\la(t,v)\big)dv,
\end{align}
This is shown by integration by parts and noting 
\begin{equation}
\frac{\partial \bar{\rho}}{\partial\la} = \chi(\rho),   \label{2.9}
\end{equation}
cf.\ \eqref{eq:2.5-C}.
The terms with $\rho(t,x/N)$ in the second and third terms of $W$
cancels by this identity with an error bounded by
$$
C N^{-1}(\|\partial^3\la\|_\infty+
\|\partial^2\la\|_\infty \|\partial\rho\|_\infty 
+ \|\partial^2\la\|_\infty\|\partial\la\|_\infty
+ \|\partial\la\|_\infty^2 \|\partial \rho \|_\infty),
$$
which appears when we discretize the above two integrals in \eqref{3.IBP}; 
note that $P, \widehat{c}$ are bounded and $P, \widehat{c} \in C^\infty((0,1))$.
Note that $\|\partial \rho \|_\infty$ is uniformly bounded by 
$\|\partial \la \|_\infty$ for $\la$ such that $\|\la\|_\infty \le \la_*$.
\end{proof}

\subsection{Entropy inequality and large deviation error estimate}

In Lemma \ref{Lemma 3.2}, $E^{f_t}[W_t]$ is an expectation under 
$P^{f_t} = f_t\,d\nu^N$.
We apply the entropy inequality to reduce it to the expectation under
$P^{\psi_t} = \psi_t\,d\nu^N$:
\begin{equation}  \label{eq:3-ent}
E^{f_t}[W_t] \le \frac{K}{\de N^d} \log E^{\psi_t}[ e^{\de N^d K^{-1} W_t}] 
+ \frac{K}{\de}  h_N(t),
\end{equation}
for every $\de > 0$ (different from $\de>0$ in $\overline{K}_\de(N)$).  
We study the asymptotic behavior of
the first term on the right-hand side via a large deviation type upper 
bound with an appropriate error estimate under
$\psi_t\, d\nu^N$; see Lemma \ref{Theorem 3.3} below.

For  $\la(\cdot) \in C^1(\T^d)$  and  $F = F(\eta) \in \mathcal{F}_0^d$,
dropping the $t$-dependence in \eqref{eq:2psit}, the local equilibrium state  
$\psi_{\la(\cdot),F}^N(\eta) d\nu^N$ of second order approximation
is a probability measure on  $\mathcal{X}_N$ defined by
\begin{equation} \label{B.8}
\psi_{\la(\cdot),F}^N (\eta) = Z^{-1}   
   \exp\bigg\{ \sum_{x\in \T_N^d} \la(x/N) \eta_x
      + \frac1{N} \sum_{x\in\T_N^d} \left(\partial\la(x/N), \tau_x F(\eta)
           \right) \bigg\}, 
\end{equation}
for  $\eta\in\mathcal{X}_N$,  where   $Z = Z_{\la(\cdot),F,N}$  is the 
normalization constant with respect to $\nu^N$.  Then we
have the following large deviation type upper bound for  
$\psi_{\la(\cdot),F}^N\, d\nu^N$.
A similar result is shown in Lemmas 5 and 7 of \cite{14} in the case of $F \equiv 0$
and Theorem 3.3 of \cite{FUY}. But here, we
give its error estimate.

The Bernoulli measure on $\mathcal{X}= \{0,1\}^{\Z^d}$ 
associated with the chemical potential  $\la \in \R$  is denoted by
$\nu_\la$, that is, $\nu_\la= \bar\nu_\la^{\otimes \Z^d}$, where
$\bar\nu_\la$ is a probability measure on $\{0,1\}$ defined by
$\bar\nu_\la(\eta) := e^{\la\eta}/(e^\la+1)$, $\eta=0, 1$.
We abuse the notation $\nu_\la$ and $\nu_\rho$ (defined in Section
\ref{Section 1.1}) which have different
parametrizations but are clearly distinguishable.  Indeed, $\nu_\la = 
\nu_{\bar{\rho}(\la)}$, i.e.\ $\nu_\rho$ with $\rho=\bar\rho(\la)$
defined by \eqref{2.3}.

The rate function for the large deviation principle for the Bernoulli measure
$\nu_{\la}$  is denoted by  $I(u;\la)$, namely, for $u\in [0,1]$,
\begin{align}\label{2.4} 
\begin{aligned}
& I(u;\la) = - \la  u - q(u) + p(\la),    \\
 & p(\la) = \log (e^\la + 1), \quad
 q(u)   = -\{u\log u +(1-u)\log(1-u)\},
\end{aligned}
\end{align}
see Appendix B for more details.

When we apply the following lemma later, we take $K=K(N)$, but in this lemma,
we consider $K$ as a parameter in the function $G(v,\rho)$,
 such that $\frac1K\in (0,1]$, and derive
an estimate which is uniform in $\frac1K$.

\begin{lem} \label{Theorem 3.3}
For every  $G_1(v,\rho), G_2(v,\rho) \in C^1(\T^d \times [0,1])$, we have
\begin{align*} 
& N^{-d}  \log E^{\psi_{\la(\cdot),F}^N} [ \exp \widetilde{G}(\eta) ]  \\
 \le &  \sup_{\rho(v) \in C(\T^d ; [0,1])}
    \int_{\T^d} \{ G(v,\rho(v)) 
    -I(\rho(v);\la(v))\} \, dv + Q_{N,\ell}^{LD}(\la,F;G),
\end{align*}
where $G(v,\rho)= G_1(v,\rho)+\frac1K G_2(v,\rho)$ and
\begin{equation} \label{eq:3.tildeG}
\widetilde{G}(\eta) \equiv \widetilde{G}_{N,\ell}(\eta)
= \sum_{x\in \T_N^d}  G(x/N,\bar{\eta}_x^{\ell}).
\end{equation}
The error term $Q_{N,\ell}^{LD}(\la,F;G)$ has an estimate which is uniform 
in $\frac1K\in (0,1]$:
\begin{align*}
|Q_{N,\ell}^{LD}(\la,F;G)| \le & C N^{-1} \|\partial\la\|_\infty \| F\|_\infty
 +  C N^{-1}\ell  \Big( \|\partial\la\|_\infty 
+  \|\partial G_1\|_\infty
    + \|\partial G_2\|_\infty \Big) \\
    & + C \ell^{-d} \Big(\log \ell +  \|\la\|_\infty
+  \|\partial_\rho G_1\|_\infty
    + \|\partial_\rho G_2\|_\infty \Big), 
\end{align*}
where $\partial G_i = \partial_v G_i$ for $i=1,2$.
\end{lem}

The proof of this lemma is rather standard but, for the sake of completeness, 
it will be given in Appendix B. 
We will apply this lemma for the first term on the right-hand
side of \eqref{eq:3-ent}, which determines $G_1$ and $G_2$ as
\begin{equation}  \label{3.deKW}
\de N^d K^{-1}W(\eta) = \sum_{x\in \T_N^d}
\big\{G_1(x/N,\bar\eta_x^\ell) + K^{-1}G_2(x/N,\bar\eta_x^\ell)\big\},
\end{equation}
see Proposition \ref{Theorem 2.1} below.  The definition of $\widetilde{G}$ 
was slightly different from \eqref{eq:3.tildeG} in Theorem 3.3 of \cite{FUY},
but \eqref{eq:3.tildeG} is more convenient for our application.

\subsection{Application to the estimate for $h_N(t)$}
\label{sec:6.3-D}

The next proposition is immediate by combing Lemma \ref{Lemma 3.2}, 
the entropy inequality 
\eqref{eq:3-ent} and Lemma \ref{Theorem 3.3}.  A corresponding result
was given for $h(t) = \varlimsup_{N\to\infty} h_N(t)$ in Theorem 2.1
of \cite{FUY} (the error estimate of $o(1)$ was sufficient in \cite{FUY}).  But, 
in the present setting, for the reason already noted, we need a bound
for $h_N(t)$ with an appropriate error estimate.

To state the proposition, let us define $g(t)$ as follows:
\begin{align*}
& g(t) \equiv g_{\de, K}(t) = \sup_{u(v) \in C(\T^d;[0,1])} 
   \int_{\T^d}
   \{\frac{\de}K \cdot \si(u(v);t,v) - I(u(v);\la(t,v)) \}\, dv,  
        \\
&  \si(u;t,v) \equiv \si_K(u;t,v)=  -\dot{\la}(t,v)\{ u-\rho(t,v)\} +
   \Tr \left(\partial^2\la(t,v) \{ P(u) - P(\rho(t,v))\}\right)
          \phantom{\frac12}      \\
&  \qquad\qquad\qquad \qquad\qquad 
    + \frac12 \left(\partial\la(t,v),
        \{\widehat{c}(u) - \widehat{c}(\rho(t,v))\} \partial\la(t,v)
    \right)  + K \si_G(u;t,v).
\end{align*}
Recall $\rho(t,v)= \bar\rho(\la(t,v))$, Lemma \ref{Lemma 3.4} for $P(\rho)$, 
\eqref{eq:sigG} for $\si_G(u;t,v)$ and \eqref{2.4} for $I(u;\la)$.

In view of \eqref{3.deKW} and $W_t$ given in Lemma \ref{Lemma 3.2}, we take
\begin{align}
G_1(v,\rho) & = \de\cdot \si_G(\rho;t,v),  \notag \\
G_2(v,\rho) & = \de \Big[ -\dot{\la}(t,v) \{ \rho - \rho(t,v)\}   
       +  \Tr \left(\partial^2\la(t,v) \{ P(\rho) - P(\rho(t,v))\}\right)   
       \label{eq:3-G-2}\\
   & \qquad\qquad    +  \frac12 \left(\partial\la(t,v),
        \{\widehat{c}(\rho) - \widehat{c}(\rho(t,v))\}
            \partial\la(t,v)\right) \Big].  \notag
\end{align}
The corresponding error $\int_0^T Q_{N,\ell}^{LD}(\la(s),F;G) ds$ 
from Lemma \ref{Theorem 3.3}
integrated in $s$ will be denoted by $Q_{N,\ell,\de}^{LD}(\la,F)$
with $\de$ from $G_1$ and $G_2$.

\begin{prop}  \label{Theorem 2.1}
Assume $\la \in C^{1,3}([0,T]\times\T^d)$, $F \in\mathcal{F}_0^d$
and the conditions \eqref{eq:2.F}, \eqref{3.26-A}  and \eqref{3.12-C}.
Then, there exist  $\de_0, C > 0$ such
that for  any $0<\de<\de_0$, $\b > 0$ and $1\le \ell \le N/2$, we have
\begin{align}  \label{2.5} 
 h_N(t) \le & h_N(0) + \frac{K}{\de} \int_0^t g(s)\,ds
                + \frac{K}{\de} \int_0^t h_N(s)\,ds  \\                
              &  +Q_{N,\ell,\b,K}^{\Om_1}(\la,F)               
               +Q_{N,\ell,\de}^{LD}(\la,F)
            \notag         \\
  & + C (\b+1) \, \|\partial\la\|_\infty^2
          \sup_{\rho\in [0,1]} \|R(\rho; F)\| 
    +   Q_N^{En}(\la,F) + Q_{N,\ell}^{\Om_2}(\la,F).  \notag
\end{align}
The estimate is uniform in $\la$ such that $\|\la\|_\infty \le \la_*$ for
each $\la_*>0$ as we noted in Lemma \ref{Lemma 3.2}.
\end{prop}

In Section \ref{sec:6}, we analyze the error terms that appear in the 
estimate \eqref{2.5} in Proposition \ref{Theorem 2.1}.

\begin{rem}
Compared to Theorem 2.1 of \cite{FUY},
the term $K\si_G$, which results from the Glauber part, is new.
Since $K\si_G$ and therefore $\si$ contains a diverging factor $K$,
we change $\de\cdot\si$ to $\frac{\de}K \cdot\si$ in the definition of
$g$.  From this (and when applying the entropy inequality), we have 
$\frac{K}\de$ in front of $\int_0^t g ds$
in the estimate in Proposition \ref{Theorem 2.1}.
The reason for having $\frac{K}\de$ in front of $\int_0^t h_N ds$
in the estimate in Proposition \ref{Theorem 2.1} is that, when we apply
the entropy inequality, we need to pay for $K$ in $K\si_G$ by the entropy factor;
recall \eqref{eq:3-ent}.
\end{rem}

\subsection{$g(t)\le 0$ for small $\de>0$}

In Proposition \ref{Theorem 2.1}, $\la=\la(t,v)\in C^{1,3}([0,T]\times \T^d)$,
which determines $h_N(t)$ through $\psi_t=\psi_{\la(t,\cdot),F}$,
can be arbitrary.  However, in the next lemma,
we choose $\la(t,v)$ as $\la(t,v) = \bar\la(\rho_K(t,v))$ from the solution 
$\rho_K(t,v)$ of the hydrodynamic equation \eqref{1.6} and $\bar\la(\rho)$ 
in \eqref{eq:2.4}.  Under this choice of $\la(t,v)$ and if we take $\de>0$
small enough, the integral $\int_0^t g(s)ds$ in \eqref{2.5} in 
Proposition \ref{Theorem 2.1} becomes non-positive, so it can be dropped.
Recall the definition of $g(t)=g_{\de,N}(t)$ given
at the beginning of Section \ref{sec:6.3-D}, especially, it depends on $\de>0$.
This special choice of $\la(t,v)$ is used only for this lemma.  

\begin{lem}  \label{Lem:2.3}
Let $\rho(t,v)\equiv \rho_K(t,v)$ be the solution of the hydrodynamic equation
\eqref{1.6} with $\rho_0$ satisfying the condition \eqref{eq:1.12}, and define
$\la(t,v)\equiv \la_K(t,v)$ as $\la(t,v) = \bar\la(\rho(t,v))$ with
$\bar\la(\rho)$ in \eqref{eq:2.4}.  Then, we have
\begin{equation}
g(t) \equiv g_{\de_*, K}(t) \le 0,    \label{2.6}
\end{equation}
for all $K\ge 1$ by making  $\de_* > 0$  small enough.  
\end{lem}

When $K=0$, this lemma is shown in (2.6) in \cite{FUY} .
We can extend it to the case where $K\ge 1$, since we put $\frac{\de}K$
in front of $\si$ in $g(t)$.  The proof is postponed to Section \ref{sec:9.6}.

\section{Error analysis, choice of $\b, \ell, F$ and range of $K$}
  \label{sec:6}

In the estimate \eqref{2.5} obtained in Proposition \ref{Theorem 2.1}, we have
five error terms $Q_{N,\ell,\b,K}^{\Om_1}(\la,F)$, $Q_{N,\ell,\de}^{LD}(\la,F)$,
$Q_N^{En}(\la,F)$, $Q_{N,\ell}^{\Om_2}(\la,F)$ and 
$(\b+1)  \|\partial\la\|_\infty^2 \sup_{\rho\in [0,1]} \|R(\rho; F)\|$.
We choose $\la=\la(t,v)\equiv \la_K(t,v)$ and $\de=\de_*>0$ as in
Lemma \ref{Lem:2.3}.
We assumed the conditions \eqref{eq:2.F},  \eqref{3.26-A}  and \eqref{3.12-C}
to prove this proposition.  Recall that $\ell\in\N$ is taken as $\ell_* =2\ell+1<N$ 
as stated below \eqref{eq:3.saeta}.
This section cares all these quantities and conditions, and choose
$\b, \ell, F$ in an appropriate way, and determine the possible range of $K$.

\subsection{Summary of the error estimates applying Schauder estimate}  
\label{sec:error-summary}

Here we summarize the estimates on the error terms appearing in
\eqref{2.5} in Proposition \ref{Theorem 2.1}.  Note that $\la(t,v)\equiv \la_K(t,v)$
is now determined from the 
hydrodynamic equation \eqref{1.6} as in Lemma \ref{Lem:2.3}
and, in particular, the Schauder estimate \eqref{eq:3-Schauder}
is applicable.  We use the exponent 
\begin{align}  \label{eq:7-gamma}
\ga := \max\{3\a_0, \a_0+1\} >0
\end{align}
determined from $\a_0>0$ in \eqref{eq:3-Schauder}.  
In the following, the constant $C=C_T>0$ changes from line to line and depends 
on $T$.  Recall \eqref{eq:3.2-P} for $\vertiii{F}_{k,\infty}, k=0, 1,2$.

The term $Q_{N,\ell,\b,K}^{\Om_1}(\la,F)$  is the error for $\Om_1$ in
\eqref{eq:Om1} with
the non-gradient microscopic current function.  It is given in \eqref{3.QOM1} as
\begin{align}  \label{5.om1}
Q_{N,\ell,\b,K}^{\Om_1}(\la,F) = \frac{C}{\b}K+\b^2 Q_{N,\ell}^{(1)}(\la,F)
 +\b Q_\ell^{(2)}(\la,F)+Q_{N,\ell}^{(3)} (\la),
\end{align}
where the first three terms are obtained in 
Theorem \ref{Theorem 3.2} and the last one in Lemma \ref{Lemma 3.4}.
Here, the term $Q_{N,\ell}^{(1)} \equiv Q_{N,\ell}^{(1)}(\la,F)$ is estimated in \eqref{3.Q1}, and together with the Schauder estimate \eqref{eq:3-Schauder}
recalling the exponent $\ga>0$ in \eqref{eq:7-gamma}, we have
\begin{align} \label{eq:Q.2}
|Q_{N,\ell}^{(1)}(\la,F)| \le C N^{-1} K^\ga \ell^{2d+4}(1+\vertiii{F}_{0,\infty}^3).
\end{align}
The term $Q_\ell^{(2)} \equiv Q_\ell^{(2)}(\la,F)$ is for the gradient
replacement and is estimated in \eqref{eq:3.26} together with 
\eqref{eq:3-Schauder} as
\begin{align} \label{5.Q2ell}
|Q_\ell^{(2)}(\la,F)| \le C K^\ga Q_\ell(F)
\end{align}
and by Corollary \ref{Corollary 5.1},
\begin{align} \label{eq:Q.3}
Q_\ell(F)
\le C (1+ r(F)^{2d}) (1+\vertiii{F}_{0,\infty}^2) \ell^{-1}+ C \ell^{-\a_1},
\end{align}
for some $\a_1>0$.  
The term $Q_{N,\ell}^{(3)} \equiv Q_{N,\ell}^{(3)}(\la)$ given in Lemma \ref{Lemma 3.4}
is independent of $F$ and it is estimated from above as
\begin{align} \label{eq:Q.5}
Q_{N,\ell}^{(3)}(\la)\le C K^\ga (\ell^{-1}+N^{-1}) + C \b K^\ga \ell^{-1} + 
\frac{C}\b K + C \b^2 K^\ga N^{-1}\ell^{2d-2}.
\end{align}

The term $Q_{N,\ell,\de_*}^{LD}(\la,F)$ is defined just above 
Proposition \ref{Theorem 2.1}.  We choose $\de=\de_*>0$. It
appears in the large deviation estimate
(see Lemma \ref{Theorem 3.3}) after applying 
the entropy inequality and, by noting that we have
$\|\partial_v G_i\|_\infty \le C K^\ga$,  $\|\partial_\rho G_i\|_\infty \le C K^\ga$
for $i=1,2$ by Schauder estimates \eqref{eq:Schauder},
\eqref{eq:3-Schauder} and the comparison theorem 
(Lemma \ref{lem:max}) used for $\|\la\|_\infty$, it is estimated as
\begin{align} \label{eq:Q.6}
|Q_{N,\ell,\de_*}^{LD}(\la,F)| \le C N^{-1} K^\ga (\ell+ \| F\|_\infty)+
C \ell^{-d} ( \log \ell+K^\ga).
\end{align}

The term $Q_N^{En}(\la,F)$ is the error in the entropy calculation and
is estimated as in \eqref{eq:R(la,F)} (multiplied by $T$).  Indeed, together with
the Schauder estimate \eqref{eq:3-Schauder}
and the comparison theorem (Lemma \ref{lem:max}), we have
\begin{align} \label{eq:R.1}
|Q_N^{En}(\la,F)| \le C N^{-1} K^\ga (1+\vertiii{F}_{2,\infty})^3,
\end{align}
if the condition \eqref{eq:2.F},
in particular,
\begin{align} \label{eq:R.1-A}
N^{-1} K^\ga \vertiii{F}_{0,\infty} \le 1
\end{align}
is satisfied.  

The term $Q_{N,\ell}^{\Om_2}(\la,F)$ estimated in \eqref{3.QOM2} is 
the error for the one-block estimate \eqref{One-block:est} applied for the 
gradient term $\Om_2$ in \eqref{eq:Om2}, adding the discretization error
of \eqref{3.IBP}.  Applying the Schauder estimate \eqref{eq:3-Schauder},
it is bounded as
\begin{align} \label{eq:Q.5-b}
|Q_{N,\ell}^{\Om_2}(\la,F)|
\le C K^\ga & \big(N^{-1}K^{1/2} \ell^{(d+2)/2} + \ell^{-d} (1+r(F)^d)\big)\\
& \times
 (1+\vertiii{F}_{0,\infty}^2 + \vertiii{F}_{1,\infty}) + C K^\ga N^{-1}. \notag
\end{align}

\subsection{Choice of $\b, \ell, F$ and range of $K$}
  \label{sec:6.2}

Related to the error term involving $R(\rho; F)$, we recall
Proposition \ref{prop:5.4} (shown in \cite{FGW}). By this proposition,
one can construct a sequence 
of functions $\Phi_n=\Phi_n(\eta)\in \mathcal{F}_0^d$ on $\mathcal{X}$
satisfying the condition \eqref{eq:Fn}, that is,
\begin{equation*}  
r(\Phi_n)\le n, \quad \sup_{\rho\in [0,1]} \| R(\rho; \Phi_n) \| \le C_2 n^{-\a_2},
\quad \|\Phi_n\|_\infty \le C_2 n^2 \log n,
\end{equation*}
for some $C_2>0$ and $\a_2>0$.  Recall \eqref{eq:3rF} for $r(\Phi_n)$ and
\eqref{eq:3.R} for $R(\rho;\Phi_n)$.

First, we choose $n=n(N)$ and $\ell=\ell(N)$ such that 
$1\le n=[N^{a_1}] \ll \ell = [N^{a_2}] \ll N$ ($[\,\cdot\,]$ denotes the integer part)
with exponents
\begin{equation}  \label{eq:a1a2}
0<a_1<a_2:= 1/(2d+5)<1.  
\end{equation}
Then, we determine $\b(N) := \b(n(N))$ from $\b(n) := n^{\a_2/2}$ with
$\a_2>0$ as in \eqref{eq:Fn} cited above, and $F_N := \Phi_{n(N)}$
(i.e.\ $\Phi_n$ with $n=n(N)$), respectively.
In particular, the radius
$r(\Phi_n)$ of the support of $\Phi_n$ is much smaller than $\ell$
so that the arguments in Section \ref{sec:5} work smoothly for $F=F_N$;
note that $n$ in $H_\ell(\eta)$ in \eqref{eq-3.H} is taken as $n=r(F_N)+1
\le n(N)+1$.  Note that this choice of $F_N$ coincides with that indicated
in Theorem \ref{Corollary 2.1}.

For $K=K(N)\ge 1$, we assume
\begin{align} \label{eq:overlineK}
1 \le K \le \overline{K}_\de(N) := \de \log N
\end{align}
and take $\de=\frak{c}/T>0$ with sufficiently small $\frak{c}>0$
as explained below (and also as in Theorem \ref{Corollary 2.1}).
In particular, by \eqref{eq:Fn}, \eqref{eq:a1a2} and
\eqref{eq:overlineK}, the conditions \eqref{3.26-A} and \eqref{3.12-C} hold. 
We may take $a\in (a_1/a_2,1)$ for \eqref{3.12-C} noting $r(F_N)\le N^{a_1}$
and $\ell_*(N)^a \le (2N^{a_2}+1)^a$.  (Note that $\de>0$ in \eqref{3.26-A} 
is different from that in \eqref{eq:overlineK}.) Moreover, by \eqref{eq:Fn}, 
the condition \eqref{eq:R.1-A} also holds.

Later in Section \ref{sec:3.7}, to obtain the bound on $h_N(t)$, we apply
Gronwall's inequality for the estimate \eqref{2.5} in Proposition \ref{Theorem 2.1}
taking $\de=\de_*>0$ (chosen as in Lemma \ref{Lem:2.3})
so dropping the term involving $g(s)$.  Then, from
$K/\de_*$ in front of $\int_0^t h_N(s)ds$, a diverging factor $e^{Kt/\de_*}$
appears in the concluding estimate for $h_N(t)$; cf.\ \eqref{eq:3.hNt}.
This means that all error terms are multiplied by this factor.  We need to study
them multiplying $e^{Kt/\de_*}$.

Now, we give  bounds for the eight error terms (recall that 
$Q_{N,\ell,\b,K}^{\Om_1}(\la,F_N)$ consists of four terms as in  \eqref{5.om1})
and the initial relative entropy $h_N(0)$, all multiplied by $e^{c_*K}$.
Setting $c_*:= T/\de_*$, from \eqref{eq:overlineK}, the factor $e^{Kt/\de_*}$ 
is bounded by $e^{c_*K} \le N^{c_*\de}$ for $t\in [0,T]$, where $\de>0$ is
in \eqref{eq:overlineK}.
We also bound $K, K^\ga \le CN^{c_*\de}$ for simplicity.  
Noting $c_*\de = T\de/\de_*$, we set $\frak{c}=T\de$ as in 
Theorem \ref{Corollary 2.1}. In particular, $N^{c_*\de}= N^{\frak{c}/\de_*}$.

In the  derivation of the following nine estimates, we take $\frak{c}>0$ small enough. 
Also choose $a_1>0$ in \eqref{eq:a1a2} for $n= [N^{a_1}]$ small enough.

 The constants $\bar C$ and $\k$ in these estimates may change from 
line to line but they can always be taken as $\bar C=\bar C_T>0$ and $\k>0$
($\bar C$ large enough depending on $T$, while $\k$ small enough depending
on $\frak{c}$ but independently of $T$ and $\de$).

\begin{enumerate}
\item The term
$\b^2 Q_{N,\ell}^{(1)}(\la,F) \cdot e^{c_*K}$ in \eqref{5.om1} with $e^{c_*K}$
is bounded, from \eqref{eq:Q.2} and \eqref{eq:Fn}, by
\begin{align*} 
C n^{\a_2} N^{-1} \ell^{2d+4} n^{3d+7} N^{2c_*\de} \le \bar C N^{-\k}.
\end{align*}
Note that $\ell^{2d+4}=N^{(2d+4)/(2d+5)}$ by \eqref{eq:a1a2}.

\item The term  $\b Q_{\ell}^{(2)}(\la,F) \cdot e^{c_*K}$ in \eqref{5.om1}
is bounded, from \eqref{5.Q2ell}, \eqref{eq:Q.3} and \eqref{eq:Fn}, by
\begin{align*} 
C n^{\a_2/2} \big( n^{4d+5} \ell^{-1}  + \ell^{-\a_1}\big) \, N^{2c_*\de} 
\le \bar C N^{-\k}.
\end{align*}

\item The upper bound of $Q_{N,\ell}^{(3)}(\la) \cdot e^{c_*K}$ 
in \eqref{5.om1} is given, from \eqref{eq:Q.5}, by
\begin{align*} 
C \Big\{ \ell^{-1}+N^{-1}  +n^{\a_2/2} \ell^{-1} + n^{-\a_2/2}
+ n^{\a_2} N^{-1}\ell^{2d-2} \Big\} N^{2c_*\de} \le \bar C N^{-\k}.
\end{align*}

\item The term $\frac{C}\b K\cdot  e^{c_*K}$ in \eqref{5.om1}
 with $e^{c_*K}$ is bounded by
\begin{align*} 
C n^{-\a_2/2} N^{2c_*\de}\le \bar C N^{-\k},
\end{align*}
by choosing $\de >0$ small enough.

\item The term $|Q_{N,\ell,\de_*}^{LD}(\la,F)| \cdot e^{c_*K}$ is bounded, 
from \eqref{eq:Q.6} and \eqref{eq:Fn}, by
\begin{align*} 
C \Big\{ N^{-1} (\ell+ n^3) +
\ell^{-d} \log \ell\Big\} N^{2c_*\de}\le \bar C N^{-\k}.
\end{align*}

\item Recalling that \eqref{eq:R.1-A} is satisfied, the term 
$Q_N^{En}(\la,F)\cdot e^{c_*K}$ is bounded,
from \eqref{eq:R.1} and \eqref{eq:Fn}, by
\begin{align*} 
C N^{-1} n^{3d+13} N^{2c_*\de} \le \bar C N^{-\k}.
\end{align*}

\item The term $|Q_{N,\ell}^{\Om_2}(\la,F)|\cdot e^{c_*K}$ is bounded, 
from \eqref{eq:Q.5-b} and \eqref{eq:Fn}, by
\begin{align*}
C \big(N^{-1} \ell^{(d+2)/2} + \ell^{-d} n^d \big)
n^{2d+5} N^{3c_*\de}+ C N^{2c_*\de} N^{-1} \le \bar C N^{-\k}.
\end{align*}

\item The term $(\b+1) K^\ga \sup_{\rho\in [0,1]} \|R(\rho; F)\|\cdot e^{c_*K}$
 is bounded, from \eqref{eq:Fn}, by
\begin{align*} 
C n^{\a_2/2} N^{2c_*\de} n^{-\a_2} \le \bar C N^{-\k},
\end{align*}
by choosing $\de>0$ small, once $a_1>0$ is determined in the above procedure.

\item Finally, by the assumption $h_N(0) \le C_1 N^{-\k_1}$, $C_1, \k_1>0$
of Theorem \ref{Corollary 2.1}, we have
\begin{align*} 
h_N(0) e^{c_*K} \le \bar C N^{-\k}.
\end{align*}
\end{enumerate}

In summary, the nine terms are all bounded by $\bar C N^{-\k}$.

\section{Proof of Theorem \ref{Corollary 2.1}}
\label{sec:3.7}

In this section, we complete the proof of Theorem \ref{Corollary 2.1}.
It follows from Proposition \ref{prop:5.4},
Proposition \ref{Theorem 2.1} and Lemma \ref{Lem:2.3}.  In fact, by (\ref{2.5}) 
in Proposition \ref{Theorem 2.1} taking $\de=\de_*>0$ and $\la(t,v)$ determined 
from \eqref{1.6} as in Lemma \ref{Lem:2.3}, $F=F_N$ as given at the beginning of
Section \ref{sec:6.2} (as noted, this is the same as in the statement of
Theorem \ref{Corollary 2.1})
and $\la_* = \log(1-c)/c>0$ with $c=\min\{\rho_-\wedge \a_-, 
1-\rho_+\vee \a_+\} \, (<1/2)$ with $\rho_\pm$ and $\a_\pm$ as in Lemma 
\ref{lem:max}, we have
\begin{align*}
h_N(t) \le  & h_N(0) + \frac{K}{\de_*} \int_0^t h_N(s)\,ds
+Q_{N,\ell,\b,K}^{\Om_1}(\la,F_N)
+Q_{N,\ell,\de_*}^{LD}(\la,F_N) \\
& + C (\b+1) \, \|\partial\la\|_\infty^2
          \sup_{\rho\in [0,1]} \|R(\rho; F_N)\| 
    +   Q_N^{En}(\la,F_N) + Q_{N,\ell}^{\Om_2}(\la,F_N).
\end{align*}
Therefore, with the help of Gronwall's inequality, we have
\begin{align}  \label{eq:3.hNt}
0\le  h_N(t) \le  e^{Kt/\de_*} \Big( &h_N(0) +
Q_{N,\ell,\b,K}^{\Om_1}(\la,F_N)
 +Q_{N,\ell,\de_*}^{LD}(\la,F_N) \\
 & +C (\b+1) \, \|\partial\la\|_\infty^2
          \sup_{\rho\in [0,1]} \|R(\rho; F_N)\|   \notag  \\  \notag
&    +   Q_N^{En}(\la,F_N) + Q_{N,\ell}^{\Om_2}(\la,F_N)
\Big).
\end{align}
Choose $\b=\b(N)$ and $\ell=\ell(N)$ as given at the beginning of 
Section \ref{sec:6.2}, and assume the condition \eqref{eq:overlineK} for 
$K=K(N)$ taking $\de=\frak{c}/T$ with
$\frak{c}>0$ sufficiently small.  Then, as we show in Section \ref{sec:6.2}, 
each of the nine terms (recall that $Q_{N,\ell,\b,K}^{\Om_1}(\la,F_N)$
consists of four terms as in \eqref{5.om1}) multiplied by $e^{KT/\de_*}$ 
on the right-hand side of \eqref{eq:3.hNt} is bounded by $\bar C N^{-\k}$
for some $\bar C=\bar C_T>0$ and $\k>0$.  Therefore, we conclude 
$h_N(t) \le C N^{-\k}$, $t\in [0,T]$ for  $C=9\bar C$.
This completes the proof of Theorem \ref{Corollary 2.1}.

\section{Proofs of lemmas and propositions in Sections \ref{Section 3}
-- \ref{section:6}}
\label{sec:4-C}

Here we complete the proofs of 
Lemma \ref{Lemma 3.1} (calculation of $\partial_t h_N(t)$), 
Proposition \ref{One-block} (refined one-block estimate),
Lemma \ref{Lemma 3.3} (a key lemma), 
Lemma \ref{Lemma 3.4} (rewriting $O(N)$-looking term to $O(1)$),
Proposition \ref{Proposition 5.1} (CLT variance)
and Lemma \ref{Lem:2.3} ($g(t)\le 0$)
postponed from Sections \ref{Section 3}, \ref{sec:3.3},
\ref{sec:5} and \ref{section:6}.  
These are gathered here, since they have a common feature that
they are shown based on known results with some modification
and refinement, providing careful error estimates.

\subsection{Proof of Lemma \ref{Lemma 3.1} (calculation of $\partial_t h_N(t)$)}
\label{sec:4.1-C}

Recall that $h_N(t)$ is the relative entropy per volume defined by \eqref{2.2} 
with respect to the local equilibrium $\psi_t = \psi_{\la(t,\cdot),F}$
in \eqref{eq:2psit}.  We apply the estimate
\begin{equation}
\partial_t h_N(t) \le N^{-d}
  \int \psi_t^{-1} \left(\mathcal{L}_N^* \psi_t - \partial_t \psi_t
         \right) \cdot f_t\, d\nu^N     \label{3.1}
\end{equation}
which holds for a large class of Markovian models, where $\mathcal{L}_N^*$ 
denotes the dual operator of the generator  $ \mathcal{L}_N$  with respect to 
the measure $ \nu^N$; see \cite{14} Lemma 1, \cite{7} Lemma 3.1.
The proof is divided into four steps.

{\it Step 1.} (Kawasaki part) 
The term in the integrand on the right-hand side of  (\ref{3.1}) derived from
the Kawasaki part was already computed in the proof of Lemma 3.1 of \cite{FUY},
but only with the error estimate $o(1)$.  Let us record the computation (3.2) in
\cite{FUY} here to make
the error term clear.  Noting the symmetry of $L_E$:  $L_E^*=L_E$,
\begin{align}  \label{3.2}
  & \qquad  N^{-d} \psi_t^{-1} N^2 L_E \psi_t    \\
 = &  \frac{N^{2-d}}2  \sum_{x\sim y} c_{x,y} \bigg[ \exp 
   \Big\{(\la(t,x/N) - \la(t,y/N)) (\eta_y-\eta_x)
           \phantom{\sum_{z\in\T_N^d}}   \notag  \\
  &\qquad\qquad \qquad\qquad \qquad      
   + \frac1N \pi_{x,y} \sum_{z\in\T_N^d} \left(\partial\la(t,z/N),\tau_z F
     \right) \Big\} - 1 \bigg]    
        \notag \\  
 = & - \frac{N^{1-d}}2 \sum_{x\sim y} c_{x,y} \Om_{x,y}
     + \frac{N^{-d}}{4} \sum_{x\sim y} c_{x,y} \Om_{x,y}^2  \notag \\
 & - \frac{N^{-d}}{4} \sum_{x\sim y}  c_{x,y} \sum_{i,j}
     \partial_{v_i}\partial_{v_j}\la(t,x/N) (y_i -x_i)(y_j-x_j)(\eta_y-\eta_x) \notag \\
     & + \frac{N^{-d}}2 \sum_{x\sim y} c_{x,y} \sum_{i,j}
     \partial_{v_i}\partial_{v_j}\la(t,x/N)  \pi_{x,y}(\sum_{z\in\T_N^d}(z_j-x_j)
       \tau_z F_i)          + Q_{1,N},   \notag
\end{align}
where the error term $Q_{1,N}\equiv Q_{1,N}(\la,F)$ is defined as the difference
between the expressions in the middle two lines and the last three lines
in \eqref{3.2}.  Recall \eqref{eq:Omxy} for $\Om_{x,y}$.

It is estimated as 
\begin{align}  \label{eq:R(la,F)-1}
|Q_{1,N}| \le &
CN^{-1} \Big(1+ \|\partial\la\|_{3,\infty}\Big)^3
   \Big(1+\vertiii{F}_{2,\infty}\Big)^3,
\end{align}
if the condition \eqref{eq:2.F} 
and accordingly $|\e_N| \le C$ below are satisfied for some $C>0$.
Indeed, to show \eqref{eq:R(la,F)-1}, we expand the second to the third lines
of \eqref{3.2} as
$$
\exp\{\e_N\}-1 = \e_N+ \tfrac12\e_N^2 + O(\e_N^3), \quad |\e_N| \le C,
$$
for $\e_N\equiv \e_{N,x,y}$ given in the braces.  Then, by Taylor's formula, we have
\begin{align*}
\la(t,x/N) - \la(t,y/N)= -&\tfrac1N \big( \partial\la(t,x/N), (y-x)\big)\\
& -\tfrac1{2N^2}(y-x)\cdot \partial^2\la(t,x/N) (y-x)+ 
O\Big( N^{-3} {\|\partial^3\la\|_\infty} \Big),
\end{align*}
and therefore
$$
\e_N= -\tfrac1N \Om_{xy} + \tfrac1N \bar\Om_{xy}
- \tfrac1{2N^2}(y-x)\cdot \partial^2\la(t,x/N)(y-x)
(\eta_y-\eta_x) + O\Big( N^{-3} {\|\partial^3\la\|_\infty}\Big),
$$
where
\begin{align*}
& \bar \Om_{x,y}  \equiv \bar\Om_{x,y}(\eta) 
 =  \pi_{x,y}\sum_{z\in\T_N^d}\Big(\partial\la(t,z/N) 
   - \partial\la(t,x/N), \tau_z F\Big) \\
& = N^{-1} \Big( \partial^2\la(t,x/N)\pi_{x,y}\sum_{z\in\T_N^d}(z-x),\tau_z F\Big) 
+ O\Big(N^{-2} \|\partial^3\la\|_\infty \Big|\pi_{xy}\sum_{z\in\T_N^d}
 (z-x)^{\otimes 2},
\t_z F\Big|\Big).
\end{align*}
Thus, $Q_{1,N}\equiv Q_{1,N}(\la,F)$ is given by
\begin{align*}
Q_{1,N} = & \frac{N^{2-d}}2 \sum_{x\sim y} c_{x,y} 
O\Big( N^{-3} \|\partial^3\la\|_\infty  (1+\vertiii{F}_{2,\infty})\Big)  \\
& + \frac{N^{2-d}}4 \sum_{x\sim y} c_{x,y}  
\big(\e_N^2- N^{-2}\Om_{x,y}^2 \big)
+ O\Big(N^{2-d} \sum_{x\sim y} c_{x,y} \e_N^3 \Big).
\end{align*}
The first term is $O\Big( N^{-1} \|\partial^3\la\|_\infty(1+\vertiii{F}_{2,\infty})\Big)$.
By noting that $|\Om_{xy}|\le \|\partial\la\|_\infty 
(1+\vertiii{F}_{0,\infty})$,
the second is bounded by
\begin{align*}
\Big| & \frac{N^{2-d}}4 \sum_{x\sim y} c_{x,y}  
\big(\e_N+ \tfrac1{N}\Om_{x,y} \big) \big(\e_N- \tfrac1{N}\Om_{x,y} \big) \Big| \\
& \le CN^2 \Big( \tfrac1{N^2} \|\partial^2\la\|_\infty \vertiii{F}_{1,\infty}
+ \tfrac1{N^3} \|\partial^3\la\|_\infty \vertiii{F}_{2,\infty}
+ \tfrac1{N^2} \|\partial^2\la\|_\infty
+ \tfrac1{N^3} \|\partial^3\la\|_\infty \Big) \\
& \qquad 
\times \bigg( \tfrac2{N} \|\partial\la\|_\infty \big(1+\vertiii{F}_{0,\infty}\big)
+\tfrac1{N^2} \|\partial^2\la\|_\infty \vertiii{F}_{1,\infty}  \\
& \qquad\qquad + \tfrac1{N^3} \|\partial^3\la\|_\infty \vertiii{F}_{2,\infty}
+ \tfrac1{N^2} \|\partial^2\la\|_\infty
+ \tfrac1{N^3} \|\partial^3\la\|_\infty \bigg) \\
& \le CN^{-1} \Big( 1+ \|\partial\la\|_{3,\infty}\Big)^2
\Big(1+\vertiii{F}_{2,\infty}\Big)^2.
\end{align*}
Similarly, the third is bounded by
\begin{align*}
CN^2 & \bigg\{ \tfrac1{N^3} \sup_{x,y} |\Om_{xy}|^3 +
\Big( \tfrac1{N^2} \|\partial^2\la\|_\infty \vertiii{F}_{1,\infty}
+ \tfrac1{N^3} \|\partial^3\la\|_\infty \vertiii{F}_{2,\infty}\Big)^3 \\
& \qquad  + \big( \tfrac1{N^2} \|\partial^2\la\|_\infty\big)^3 
+ \big( \tfrac1{N^3} \|\partial^3\la\|_\infty\big)^3 \bigg\} \\
& \le CN^{-1} \Big( 1+ \|\partial\la\|_{3,\infty}\Big)^3
\Big(1+\vertiii{F}_{2,\infty}\Big)^3.
\end{align*}
Summarizing these estimates for three terms of $Q_{1,N}(\la,F)$, we obtain
\eqref{eq:R(la,F)-1} for $Q_{1,N}(\la,F)$.

{\it Step 2.} (Glauber part) 
Next, we compute the contribution from the Glauber part.
First we note that the dual operator $L_G^*$ of $L_G$ with respect to
$\nu^N$ is given by
$$
L_G^* g(\eta) = \sum_{x\in \T_N^d}
 \big\{ c_x(\eta^x)g(\eta^x) - c_x(\eta) g(\eta)\big\}.
$$
Therefore, we have
\begin{align}  \label{eq:3.LG}
  & \qquad  N^{-d} K \psi_t^{-1} L_G^* \psi_t    \\
 = &  N^{-d} K \sum_{x\in \T_N^d}  \Bigg[ c_{x}(\eta^x) \exp \Big\{ \pi_x
 \Big( \sum_{z\in \T_N^d} \la(t,z/N) \eta_z + \frac1N \sum_{z\in\T_N^d} \left(\partial\la(t,z/N),\tau_z F(\eta)\right) \Big)\Big\} - c_x(\eta) \Bigg]    \notag \\
   = &  N^{-d} K \sum_{x\in \T_N^d}  \Bigg[ c_{x}(\eta^x) \exp \Big\{ 
 \la(t,x/N) (1-2\eta_x) + \frac1N \sum_{z\in\T_N^d} \left(\partial\la(t,z/N),\pi_x\tau_z F(\eta)\right) \Big\} - c_x(\eta) \Bigg]    \notag  \\
 = & N^{-d} K \sum_{x\in \T_N^d}  \left[ c_{x}(\eta^x) \exp \Big\{ 
 \la(t,x/N) (1-2\eta_x)  \Big\} - c_x(\eta) \right] 
 + Q_{2,N},  \notag
\end{align}
where the error term $Q_{2,N} \equiv Q_{2,N}(\la,F)$ is defined as the
difference between two expressions in the third and the fourth lines.
However, the leading term can be rewritten as the fifth term of $\Om_2$ noting
\eqref{eq:c+-} and \eqref{2.3}.  Indeed, writing $\la=\la(t,x/N), \rho=\rho(t,x/N)$
and $c_x^\pm=c_x^\pm(\eta)$ for simplicity, the term inside the brackets in the last sum
is rewritten as
\begin{align}  \label{eq:8.4-C}
\big(c_x^+e^{-\la}1_{\{\eta_x=1\}} +  c_x^- e^{\la}1_{\{\eta_x=0\}}\big)
- \big(c_x^+1_{\{\eta_x=0\}} +  c_x^- 1_{\{\eta_x=1\}}\big).
\end{align}
Then, using $e^{-\la} = (1-\rho)/\rho$, $e^\la = \rho/(1-\rho)$,
$ 1_{\{\eta_x=1\}} =(\eta_x-\rho)+\rho$ and $ 1_{\{\eta_x=0\}} =(1-\rho)-
(\eta_x-\rho)$, we easily see that \eqref{eq:8.4-C} is further rewritten as
$$
\Big( \frac{c_x^+}{\rho} - \frac{c_x^-}{1-\rho}\Big) (\eta_x-\rho),
$$
and gives the fifth term of $\Om_2$.  Recall \eqref{eq:Om2} for $\Om_2$.
(This coincides with the right-hand side in Lemma 2.3 of \cite{FvMST}.)

The error term $Q_{2,N}$ in \eqref{eq:3.LG} is estimated as 
\begin{align}  \label{eq:R(la,F)-2}
|Q_{2,N}| \le &
CN^{-1} K e^{\|\la\|_\infty} \|\partial\la\|_\infty \vertiii{F}_{0,\infty},
\end{align}
if the second bound in the condition \eqref{eq:2.F} is satisfied.
Indeed, 
$$
\exp \{ \e_N\} = 1+O(\e_N), \quad |\e_N|\le C,
$$
for $\e_N\equiv \e_{N,x} = \frac1N \sum_{z\in\T_N^d} 
\left(\partial\la(t,z/N),\pi_x\tau_z F(\eta)\right)$.
Thus,
\begin{align*}
Q_{2,N}=  N^{-d} K \sum_{x\in \T_N^d}  c_{x}(\eta^x) &  \exp \Big\{ 
 \la(t,x/N) (1-2\eta_x)  \Big\} \\
& \times
 O\bigg(\frac1N \sum_{z\in\T_N^d} \big(\partial\la(t,z/N),\pi_x\tau_z F(\eta)\big)
 \bigg)
\end{align*}
and therefore we obtain \eqref{eq:R(la,F)-2}.

{\it Step 3.} (Second term in  \eqref{3.1})
Recalling \eqref{eq:2psit} for $\psi_t$, the second term in the integrand
on the right-hand side of \eqref{3.1} is rewritten as
\begin{align}  \label{3.3}
N^{-d} \psi_t^{-1}  \partial_t \psi_t
 =  &- N^{-d} Z_t^{-1} \partial_t Z_t  \\
&  + N^{-d} \partial_t
  \left\{ \sum_{x\in\T_N^d} \la(t,x/N) \eta_x 
  + \frac1{N} \sum_{x\in\T_N^d} \left(\partial\la(t,x/N),
     \tau_x F \right) \right\}.    \notag
\end{align}
From \eqref{3.2}, \eqref{eq:3.LG} and \eqref{3.3}, it follows that
\begin{equation}
N^{-d} \psi_t^{-1} \left(\mathcal{L}_N^* \psi_t - \partial_t \psi_t \right)
= \Om_1(\eta) + \Om_2(\eta) + a(t) +\sum_{i=1}^3 Q_{i,N},  \label{3.4}
\end{equation}
where 
$$
Q_{3,N} \equiv Q_{3,N}(\la,F) = -N^{-d-1}\sum_{x\in \T_N^d} (\partial\dot{\la}(t,x/N), \t_x F),
$$
and
\begin{equation}
a(t) := N^{-d} Z_t^{-1} \partial_t Z_t
  = E^{\psi_t}\left[ N^{-d} \sum_{x\in\T_N^d} \dot{\la}(t,x/N)\eta_x \right]
    + O\Big(N^{-1} \|\partial\dot{\la}\|_\infty \|F\|_\infty\Big).     \label{3.5}
\end{equation}
Recall \eqref{eq:Om1} for $\Om_1$.
The last equality for $a(t)$  is seen from \eqref{3.3} by noting that
$$
E^{\psi_t}\left[ N^{-d} \psi_t^{-1} \partial_t\psi_t \right] = 0
$$
and
\begin{equation}  \label{eq:R3}
| Q_{3,N}| \le N^{-1} \|\partial\dot{\la}\|_\infty \|F\|_\infty.
\end{equation}
Thus,
$$
a(t) = N^{-d} \sum_{x\in\T_N^d} \dot{\la}(t,x/N)\rho(t,x/N)+Q_{4,N} +
O\Big(N^{-1} \|\partial\dot{\la}\|_\infty \|F\|_\infty\Big),
$$
where the error $Q_{4,N} \equiv Q_{4,N}(\la,F)$ is given by
$$
Q_{4,N} = N^{-d} \sum_{x\in\T_N^d} \dot{\la}(t,x/N)
\Big( E^{\psi_t}\left[ \eta_x \right]-\rho(t,x/N)\Big).
$$
We only need to estimate $|E^{\psi_t}\left[ \eta_x \right] 
-\rho(t,x/N)|$.

{\it Step 4.} (Proof of the error estimate  \eqref{eq:R(la,F)})
To give an estimate for $Q_{4,N} \equiv Q_{4,N}(\la,F)$, we note the 
$r$-Markov property of $\psi_t$ (or $P^{\psi_t}$).
Indeed, taking $\La = \{x\}$, $r=r(F)$  and $\bar\La = \La_{r(F), x}$
in Lemma \ref{lem:3.2-A} in Appendix A, we see
\begin{equation}  \label{m-etax}
E^{\psi_t}[\eta_x] = E^{\psi_t}\big[E^{P_{\{ x\}}^{\om,F}} [\eta_x] \big]
\end{equation}
and
\begin{align}\label{mm-etax}
E^{P_{\{ x\}}^{\om,F}} [\eta_x] 
& := \widetilde Z_{\om,F}^{-1} \sum_{\eta_x \in \{0,1\}} \eta_x
e^{\la(t,x/N) \eta_x + \frac1N \sum_{y \in \T_N^d:\, \text{supp}\, \t_y F\subset \bar\La}
(\partial\la(t,y/N), \t_y F (\eta_x\cdot\om))} \\
& \; =: \widetilde Z_{\om,F}^{-1} A_{\om,F},   \notag
\end{align}
for $\om \in \{0,1\}^{\bar\La \setminus \{ x\}}$, where
\begin{equation}\label{tildeZ}
\widetilde Z_{\om,F}:= \sum_{\eta_x\in \{0,1\}} 
e^{\la(t,x/N) \eta_x + \frac1N \sum_{y \in \T_N^d:\, \text{supp}\, \t_y F\subset \bar\La}
(\partial\la(t,y/N), \t_y F (\eta_x\cdot\om))}.
\end{equation}
Note that $\sum_{y\in \bar\La \setminus \{ x\}} \la(y/N) \om_y$
in $H_{\bar \La}(\eta\cdot\om)$ cancels with that in $Z_\om$
and gives the formula \eqref{mm-etax} with $\widetilde Z_{\om,F}$ determined
by \eqref{tildeZ}.

Noting that $E^{P_{\{ x\}}^{\om,0}} [\eta_x] = \rho(t,x/N)$ taking $F\equiv 0$,
we have
$$
E^{P_{\{ x\}}^{\om,F}} [\eta_x] - \rho(t,x/N)
=  \widetilde Z_{\om,F}^{-1} A_{\om,F}  -  \widetilde Z_{\om,0}^{-1} A_{\om,0}.
$$
A simple estimate shows
\begin{align*}
& \widetilde Z_0 e^{-\frac{c}N r(F)^d \|\partial\la(t,\cdot)\|_\infty \| F\|_\infty}
\le \widetilde Z_{\om,F} \le 
\widetilde Z_0 e^{\frac{c}N r(F)^d \|\partial\la(t,\cdot)\|_\infty \| F\|_\infty}, \\
& A_0 e^{-\frac{c}N r(F)^d \|\partial\la(t,\cdot)\|_\infty \| F\|_\infty}
\le A_{\om,F} \le 
A_0 e^{\frac{c}N r(F)^d \|\partial\la(t,\cdot)\|_\infty \| F\|_\infty},
\end{align*}
for $A_0 =A_{\om,0}$, $\widetilde Z_0 = \widetilde Z_{\om,0}$ and
some $c>0$.
Therefore, taking the expectation in $\om$ under $P^{\psi_t}$, we obtain
by \eqref{m-etax}
\begin{align*}
|E^{\psi_t}\left[ \eta_x \right] -\rho(t,x/N)|
& = \bigg| E^{\psi_t}\Big[ \frac{A_0}{\widetilde Z_{\om,F}}
\Big( \frac{A_{\om,F}}{A_0} - \frac{\widetilde Z_{\om,F}}{\widetilde Z_0}\Big) \Big] \Big| \\
& \le C_1 \Big| e^{\frac{c}N r(F)^d \|\partial\la(t,\cdot)\|_\infty \| F\|_\infty}
-1 \bigg| \\
& \le C_2 N^{-1} r(F)^d \|\partial\la\|_\infty \| F\|_\infty,
\end{align*}
if $F$ satisfies the condition \eqref{eq:2.F}: $N^{-1} r(F)^d 
\|\partial\la\|_\infty \| F\|_\infty \le 1$; note that $0\le A_0/\widetilde Z_{\om,F}
= A_0/\widetilde Z_0 \cdot \widetilde Z_0/\widetilde Z_{\om,F}\le e^c$
under this condition.  This leads to the bound on 
$Q_{4,N} \equiv Q_{4,N}(\la,F)$:
\begin{align}  \label{eq:3.R4}
|Q_{4,N}| \le C N^{-1}  \|\dot{\la}\|_\infty \|\partial\la\|_\infty \vertiii{F}_{0,\infty}.
\end{align}
By \eqref{eq:R(la,F)-1}, \eqref{eq:R(la,F)-2}, \eqref{eq:R3} and \eqref{eq:3.R4},
we obtain \eqref{eq:R(la,F)}, and this concludes the proof of Lemma \ref{Lemma 3.1}.

\subsection{Proof of Proposition \ref{One-block} (refined one-block estimate)}
\label{sec:4.2-C}

We follow the argument in Section 3 of \cite{FvMST}, where the same Glauber-Kawasaki dynamics $\eta^N(t)$ was considered.  Note that although the gradient condition was assumed in \cite{FvMST}, it was not used in Section 3.  
It was used only for rewriting the term $I_1$ in p.18 of \cite{FvMST}
into a form of $O(1)$ and $I_1$ corresponds to our $\Om_1$. (For this
correspondence, note \eqref{eq:2.5-C} for $\frac1{\chi(u_x)}$ appearing in $I_1$.)

We first decompose as 
\begin{align} \label{BG:pf:0}
  E \bigg[ \bigg |\int_0^t N^{-d} \sum_{x \in \T_N^d} a_{s,x} \hat{\Om}_x 
  ds \bigg |\bigg]
  \leq & E\bigg[ \bigg |\int_0^t N^{-d} 
  \sum_{x \in \T_N^d} a_{s,x} m_x ds \bigg |\bigg]  \\
 & + E\bigg[ \bigg |\int_0^t N^{-d} \sum_{x \in \T_N^d} a_{s,x} E^{\nu_{1/2}}
 [\hat{\Om}_x \mid \bar\eta_x^\ell] ds \bigg |\bigg],  \notag
\end{align}
where $\nu_{1/2}$ is the Bernoulli measure with 
$\rho=1/2$ on $\mathcal{X}=\{0,1\}^{\Z^d}$ (recall Section \ref{Section 1.1})
and
\begin{align*}
  m_x \equiv m_x(\eta)
  := \hat{\Om}_x - E^{\nu_{1/2}}[\hat{\Om}_x \mid \bar\eta_x^\ell].
\end{align*}
Precisely saying, the three expectations in \eqref{BG:pf:0} are taken for the
process $\eta^N(\cdot)$, so that $\hat{\Om}_x$ in the left-hand side
and $m_x$ in the right-hand side are understood as $\hat{\Om}_x(\eta^N(s))$ and
$m_x(\eta^N(s))$, respectively, and also we replace $\eta$ to
$\eta^N(s)$ in the conditional expectation, which is a function of
$\eta$, in the last expectation.

Lemma 3.1 of \cite{FvMST} applies for the first term in \eqref{BG:pf:0}.
Recall we assume \eqref{BG:assn:atx}  for $a_{t,x}$.

\begin{lem} {\rm (Lemma 3.1 of \cite{FvMST})} \label{l:L10:6}
Let $\gamma = \gamma(N) > 0$. If $\gamma M \ell_*^{d+2} \le \de N^2$
for some $\de>0$ small enough, then for every $t\in [0,T]$ and $N$ large enough
\begin{equation}  \label{eq:8.15-C}
      E\bigg[ \bigg |\int_0^t N^{-d} \sum_{x \in \T_N^d} a_{s,x} m_x ds \bigg |\bigg]
      \leq C \Big( \frac K\gamma 
      + \frac{ \gamma  \ell_*^{d+2} }{ N^2 } M^{2} \| \Om\|_\infty^2\Big),
\end{equation}
for some $C=C_T>0$.  

In particular, choosing 
$\ga= NK^{1/2}\ell_*^{-(d+2)/2}/ M(\| \Om\|_\infty+1)$, the right-hand side is
bounded by
$$
2CN^{-1} K^{1/2}\ell_*^{(d+2)/2} M(\| \Om\|_\infty+1).
$$
For this choice of $\ga$, the above condition for $\ga$ is satisfied if 
\eqref{3.26-A} stated in Proposition \ref{One-block} holds.
\end{lem}

\begin{proof}
Lemma 3.1 of \cite{FvMST} holds in our setting, but we need a small adjustment
in the proof.  We point out the necessary changes.

Instead of the assumption (1.19) of \cite{FvMST}
for $a_{t,x}$ (i.e.\ $|a_{t,x}| \le CK^\th/[t(T-t)]^\Theta$ for some $\th\ge 0,
\Theta\in [0,1)$), we have \eqref{BG:assn:atx}:
$$
|a_{t,x}| \le M \quad \text{for all } t\in [0,T], \; x \in \T_N^d.
$$
In other words, we may take ``$\Theta=0$ and $M$ for $CK^\th$ in
(1.19) of \cite{FvMST}.  Furthermore, letting $\e= \e(N)\downarrow 0$,
the first term in the concluding estimate in  Lemma 3.1 of \cite{FvMST}
becomes negligible, since $\e^{1-\Theta}K^\th=\e M/C\to 0$.  

For the third term in the concluding estimate in  Lemma 3.1
($\e^{2\Theta}=\e^0=1$), at the last part of its proof, we had
$$
  \langle (-L_{\ell, x, j})^{-1} (a_{t,x} m_x), a_{t,x} m_x \rangle_{\nu_{\ell, x, j}}
  \leq \frac1{\operatorname{gap}(j,\ell)} \sup_{t,x,\eta} | a_{t,x} m_x |^2, 
$$
where $\operatorname{gap}(j,\ell) \geq C / \ell_*^2$ as indicated above this
estimate in \cite{FvMST}.  Thus, in our situation, estimating as
$$
\sup_{t,x,\eta} |a_{t,x}m_x|^2 \le (2M)^2 \|\Om\|_\infty^2,
$$
we obtain the estimate \eqref{eq:8.15-C}.  Summarizing again, \eqref{eq:8.15-C}
is obtained from the concluding estimate in  Lemma 3.1 of  \cite{FvMST},
by neglecting the first term (by letting $\e\downarrow 0$) and replacing 
$CK^{2\th}$ by $(2M)^2 \|\Om\|_\infty^2$ and $\e^{2\Theta}$ by $1$.

Note that the condition `$\gamma M \ell_*^{d+2} \le \de N^2$
for some $\de>0$ small enough' was used for the
denominator in (3.10) of \cite{FvMST} (with $M$ for $CK^\th$ and 
$\e^{-\Theta}=1$), derived by Rayleigh estimate,
to stay uniformly positive, say $\ge 1/2$.  This determines how small
$\de>0$ needs to be.  This is the condition (3.2) in \cite{FvMST} with
$\e^{-\Theta}= \e^0=1$ and $M/C$ instead of $K^\th$.

The second assertion with the special choice of $\ga$ is straightforward.
\end{proof}

For the second term in \eqref{BG:pf:0}, we apply Theorem 4.1 of \cite{CM} 
(taking $\La=\La(\ell)$) which provides the equivalence of ensembles
for general Gibbs measures with precise convergence rate:
For any $\e\in (0,1/4)$, there exists $C=C_\e > 0$ such that for all $N$ 
sufficiently large
\begin{equation} \label{eq:3EE}
      \sup_{\rho\in [0,1]} \max_{x \in \T_N^d} \sup_{\eta \in \mathcal X_N} \left| 
      E^{\nu_{\rho}}[\t_x \Om \mid \bar\eta_x^\ell]
      - \lan \Om \ran (\bar\eta_x^\ell)  \right| 
      \leq \frac C{\ell^{d}} |\De| \, \| \Om\|_\infty,
\end{equation}
where the local function $\Om$ has a support in $\De$ such that 
$|\De|\le \ell_*^{d(1-4\e)}$.

Combining the two estimates given in Lemma \ref{l:L10:6} and 
\eqref{eq:3EE} with $\rho=1/2$ and $|\De| \le r(\Om)^d \le \ell_*^{d(1-4\e)}$
and writing the exponent as  $a= 1-4\e\in (0,1)$, 
the proof of Proposition \ref{One-block} is concluded.

\subsection{Proof of Lemma \ref{Lemma 3.3} (a key lemma)}
\label{sec:4.3-C}

Recall that $J(t,v) \in C^\infty([0,T]\times\T^d, \R^n)$, 
$G(\eta) \in \mathcal{F}_0^n$ and  $M(\rho) \in C([0,1],\R^n \otimes \R^n)$
are given.  For a box $\La(\ell)$, we decompose $\eta$ as 
$\xi = \eta|_{\La(\ell)}$ and $\zeta = \eta|_{\La(\ell)^c}$.
We write $G$ as $G(\eta)= G_{\zeta}(\xi)$ regarding it as a function of
$\xi$.  We assume $\lan G_\zeta \ran_{\La(\ell),m} = 0$ for every  
$\zeta\in \mathcal{X}_{\La(\ell)^c}$ and $m: 0 \le m \le \ell_*^d$.  
The proof is divided into
three steps.  It is given by combining that of Lemma 6.1 of \cite{FUY}
and the calculation developed in that of Lemma 3.1
of \cite{FvMST} for Glauber-Kawasaki dynamics.

{\it Step 1.} (Application of entropy inequality and Feynman-Kac formula) 
Setting 
$$ 
W_{N,t}(\eta) = N \sum_{x \in \T_N^d} J(t,x/N) \cdot \tau_xG - 
   \b \sum_{x \in \T_N^d} J(t,x/N) \cdot M(\bar{\eta}_x^\ell)J(t,x/N),
$$ 
by the entropy inequality 
noting that $H(P_{f_0}|P_{\nu^N}) \le C N^d$ ($P_{f_0}$ denotes the distribution
on the path space $D([0,T],\mathcal{X}_N)$ of $\eta^N(\cdot)$ with the
initial distribution $f_0\, d\nu^N$) at the process level, we have 
\begin{align*}
 \int_0^t E^{f_s} [ N^{-d}  W_{N,s}(\eta) ]ds 
 & \le \frac1{\b} N^{-d} \log {E_{\nu^N} \left[e^{\int_0^t \b W_{N,s}(\eta^N(s))ds} 
  \right]}  + \frac{C}{\b},
\end{align*}
for every $\b>0$; recall $\nu^N=\nu_{1/2}^N$ and $E_{\nu^N}[\,\cdot\,]$
means the expectation under the process $\eta^N(\cdot)$ with the initial
distribution $\nu^N$.
Then, by Feynman-Kac formula, similar to (3.6) of \cite{FvMST}
(with $\e=0, \e^{1-\Theta}=0$ in the sense that we explained in the proof of 
Lemma \ref{l:L10:6} in Section \ref{sec:4.2-C})  shown for the
Glauber-Kawasaki dynamics, the right-hand side is bounded by
\begin{align}  \label{eq:9.18-E}
\frac1{\b} N^{-d}\int_0^t  \Om_{N,\b}(s)ds + \frac{C}{\b},
\end{align}
where 
\begin{equation}  \label{eq:9.19-Om}
\Om_{N,\b}(t) 
 = \sup_{\psi: \int\psi^2d\nu^N=1}
     \Big\{E^{\nu^N} [\b W_{N,t} \psi^2] - E^{\nu^N} [\psi (-\mathcal{L}_N \psi)]
     \Big\}.
\end{equation}
Note that
\begin{equation*}
 E^{\nu^N} [\psi (-\mathcal{L}_N \psi)]
 = N^2 E^{\nu^N} [\psi (-L_E \psi)] + K E^{\nu^N} [\psi (-L_G \psi)].
\end{equation*}

{\it Step 2.} (Contribution of Glauber part)
For the Glauber part, it was shown in the proof of Lemma 3.1 of \cite{FvMST} that
\begin{equation*}
 E^{\nu^N} [\psi (-L_G \psi)] \ge -C N^d.
\end{equation*}
This bound gives $\frac{C}\b K$ in the estimate stated in Lemma \ref{Lemma 3.3}.

{\it Step 3.} (Precise error estimate for Kawasaki part)
Once we detach $K E^{\nu^N} [\psi (-L_G \psi)]$ from the expression
on the right-hand side of \eqref{eq:9.19-Om}, then $\Om_{N,\b}(t)$ becomes
the same quantity as (6.3) of \cite{FUY} defined only from the Kawasaki part.
Then, the same bound as \cite{FUY} holds 
for $\Om_{N,\b}(t)$ detached the term from Glauber part 
in the present setting until the very end of the proof of Lemma 6.1 of \cite{FUY}.
The only difference is that,
in (6.7) of \cite{FUY}, the limit was taken as $N\to\infty$, but here we have to
derive a quantitative error estimate. 

Indeed, the first term in \eqref{eq:9.18-E} dropped the contribution of the
Glauber part is bounded by $t$ times the expression on line 4, p.28 of
\cite{FUY}, i.e.\ by adjusting the notation to our setting, it is bounded by
\begin{equation}  \label{eq:9.19-Q}
 t \sup_{|\th| \le \|J\|_\infty } \sup_{\zeta\in \mathcal{X}_{\La(\ell)^c}}
 \sup_{m: 0\le m \le \ell_*^d}
  \Big[ \frac{N^2}{\b |(\La(\ell))^*|} \Om_{N,\b,\th,G}^{m,\ell,\zeta}- \b  
      \, \th \cdot M(m/\ell_*^d)\th \Big], 
\end{equation}
where $\Om_{N,\b,\th,G}^{m,\ell,\zeta}$ is the largest eigenvalue of the symmetric
operator $L_{\La(\ell),\zeta}+V$ with $V=N^{-1}\b |(\La(\ell))^*|  \, 
\th\cdot G_\zeta$ in $L^2(\mathcal{X}_{\La(\ell),m}, \nu_{\La(\ell),m})$. 

By Theorem 6.1 of \cite{FUY} (Rayleigh-Schr\"odinger bound for 
general Markov generator), $\Om_{N,\b,\th,G}^{m,\ell,\zeta}$ is bounded as
\begin{equation}  \label{eq:9.21-P}
\Om_{N,\b,\th,G}^{m,\ell,\zeta} \le
\lan V,  (-L_{\La(\ell),\zeta})^{-1} V\ran + Q_{N,\ell,\b}(V) 
\end{equation}
and the upper bound of the error is given by 
$$
Q_{N,\ell,\b}(V) := 
4 \Vert V \Vert_{\infty}^3
   \lbr \sup_{f:\lan f\ran_{\La(\ell),m} =0}
   \frac {\lan f^2\ran_{\La(\ell),m}}{\lan -f L_{\La(\ell),\zeta} f\ran_{\La(\ell),m}}\rbr^2.
$$
Since we have a spectral gap estimate for the Kawasaki generator:
\begin{align}  \label{eq:3-SG}
\lan -f L_{\La(\ell),\zeta} f\ran_{\La(\ell),m} \ge C_0 \ell^{-2}
\lan f^2\ran_{\La(\ell),m},
\end{align}
with a constant $C_0>0$ independent of $\zeta, m, \ell$ (see \cite{6} or
Appendix A of \cite{FUY}), recalling the definition of $V$,
we have a bound for the error term:
\begin{align*}
Q_{N,\ell,\b}(V)
\le 4N^{-3}\b^3 |\La(\ell)^*|^3 \|\th\cdot G\|_\infty^3 
    \, \big(C_0^{-1} \ell^2\big)^2.
\end{align*}
Therefore, in \eqref{eq:9.19-Q},     
\begin{align*}
t \frac{N^2}{\b |(\La(\ell))^*|} Q_{N,\ell,\b}(V)    
\le C_T N^{-1}\b^2 \ell^{(2d+4)} \| G\|_\infty^3 \| J\|_\infty^3,
\end{align*}
for $\th: |\th| \le \|J\|_\infty$.  This is the estimate \eqref{eq:3QNell}
for $Q_{N,\ell}^{(1)}(J,G)$ in Lemma \ref{Lemma 3.3} multiplied by $\b^2$.  
The rest in \eqref{eq:9.19-Q} is, by \eqref{eq:9.21-P},
\begin{align*}
&t  \sup_{|\th| \le \|J\|_\infty } \sup_{\zeta\in \mathcal{X}_{\La(\ell)^c}}
 \sup_{m: 0\le m \le \ell_*^d}
  \Big[ \frac{N^2}{\b |(\La(\ell))^*|} \lan V,  (-L_{\La(\ell),\zeta})^{-1} V\ran- \b  
      \, \th \cdot M(m/\ell_*^d)\th \Big].
\end{align*}
This coincides with the variational term in Lemma \ref{Lemma 3.3} 
recalling the definition of $V$ and noting that 
$|(\La(\ell))^*| = d |\La(\ell)| = d \ell_*^d$.  Thus, we complete the proof 
of Lemma \ref{Lemma 3.3}.

\subsection{Proof of Lemma \ref{Lemma 3.4} (rewriting $O(N)$-looking term 
to $O(1)$)}
\label{sec:4.4-C}

 We essentially follow the proof of Lemma 3.4 of \cite{FUY}, 
but make it finer.  Denote the first and second terms in the expectation
$E^{f_t}[\cdots]$ 
by  $I^N_1 = I^N_1(t,\eta)$  and $ I^N_2 = I^N_2(t,\eta)$, respectively.
Then, similar to the proof of Lemma 3.4 of \cite{FUY}  (but, with the $-$ sign), 
we have the decomposition of  $I_1^N$:
\begin{equation}  \label{eq:3.IJ}
I_1^N = J^N_+ - J^N_-,
\end{equation}
where
$$
J^N_\pm = N^{1-d} \ell_*^{-d} \sum_{x\in\T_N^d} \sum_{i,j=1}^d D_{ij}(\bar{\eta}_{x}^\ell)
   \partial_{v_j} \la(t,x/N) \t_{x\pm \ell e_i} \widetilde{A}_{\ell,i}(\eta),
$$
and
$$
\widetilde{A} _{\ell,i} = \sum_{x \in \Z^d: x_i =0, |x_j| \le \ell
  \text{ for }  j\not= i} \eta_x.
$$

However, $J_-^N$  can be rewritten as
\begin{align*}
J_-^N & =  N^{1-d} \ell_*^{-d} \sum_{x\in\T_N^d} \sum_{i,j=1}^d
   D_{ij}(\bar{\eta}_{x-e_i}^\ell)
   \partial_{v_j} \la(t,(x-e_i)/N) \t_{x- (\ell+1)e_i} \widetilde{A}_{\ell,i} 
       \\  
 & =  N^{1-d} \ell_*^{-d} \sum_{x\in\T_N^d} \sum_{i,j=1}^d D_{ij}(\bar{\eta}_{x-e_i}^\ell)
   \partial_{v_j} \la(t,x/N) \t_{x- (\ell+1)e_i} \widetilde{A}_{\ell,i} 
    + Q_{1,\ell},
\end{align*}
with an error $Q_{1,\ell} \equiv Q_{1,\ell}(\la)$ bounded as
$$
|Q_{1,\ell}| \le C_1   \ell^{-1} \|\partial^2\la\|_\infty,
$$
by noting  $\partial_{v_j} \la(t,(x-e_i)/N) = \partial_{v_j} \la(t,x/N)
+ O(\|\partial^2 \la \|_\infty/N)$ 
and $\t_{x- (\ell+1)e_i} \widetilde{A}_{\ell,i} = O(\ell^{d-1})$.  
Therefore, by \eqref{eq:3.IJ}, we obtain
\begin{align*}
I_1^N = & N^{1-d} \ell_*^{-d} \sum_{x\in\T_N^d}  \sum_{i,j=1}^d \left\{
D_{ij}(\bar{\eta}_{x}^\ell) \t_{x+\ell e_i} \widetilde{A}_{\ell,i} \right.   \\
  &  \qquad\qquad \left. - D_{ij}(\bar{\eta}_{x-e_i}^\ell) \t_{x- (\ell+1)e_i} 
    \widetilde{A}_{\ell,i} \right\} \partial_{v_j} \la(t,x/N)  - Q_{1,\ell}.
\end{align*}

On the other hand, since $\partial_{v_i}\partial_{v_j} \la(t,x/N) = N \{
\partial_{v_j} \la(t,(x+e_i)/N) - \partial_{v_j} \la(t,x/N)\} + O(\|\partial^3 \la \|_\infty/N)$,
$$
I_2^N =  - N^{1-d} \sum_{x\in\T_N^d} \sum_{i,j=1}^d \left\{
P_{ij}(\bar{\eta}_{x}^\ell)  - P_{ij}(\bar{\eta}_{x-e_i}^\ell) \right\}
   \partial_{v_j} \la(t,x/N)  + Q_{2,N},    
$$
with an error $Q_{2,N} \equiv Q_{2,N}(\la)$ bounded as
$$
|Q_{2,N}| \le C_2  N^{-1} \|\partial^3\la\|_\infty.
$$
Note that $P_{ij}(\rho)$ is bounded.

From these two equalities for  $I_1^N$  and  $I_2^N$, $Q_{N,\ell}^{(3)}(\la)$ 
in the statement of the lemma is rewritten as
\begin{equation}
\int_0^T E^{f_t} \left[   N^{1-d} \sum_{x\in\T_N^d} \sum_{i,j=1}^d 
\partial_{v_j} \la(t,x/N) Q_{ij;\ell}(x,\eta) -Q_{1,\ell}+Q_{2,N}\right]
   \, dt,   \label{3.6-A}
\end{equation}
where
\begin{align*}
Q_{ij;\ell}(x,\eta) 
   = & - P_{ij}(\bar{\eta}_{x}^\ell) + P_{ij}(\bar{\eta}_{x-e_i}^\ell) \\
& + \ell_*^{-d}  \left\{
D_{ij}(\bar{\eta}_{x}^\ell) \t_{x+\ell e_i} \widetilde{A}_{\ell,i} -
D_{ij}(\bar{\eta}_{x-e_i}^\ell) \t_{x- (\ell+1)e_i} \widetilde{A}_{\ell,i} \right\}.
\end{align*}

Now, we apply Lemma \ref{Lemma 3.3} (the key lemma in the setting of 
Glauber-Kawasaki dynamics) to estimate (\ref{3.6-A})
 (except $Q_{1,\ell}$ and $Q_{2,N}$) for each fixed $i$ and $j$ from the above;
namely, take  $n = 1$, $J(t,v) = \partial_{v_j} \la(t,v), G(\eta) \equiv G_{ij}(\eta)
= Q_{ij;\ell}(0,\eta)$, $M(\rho)=0$
and apply Lemma \ref{Lemma 3.3} with  $\ell+1$  instead of  $\ell$ by
noting that  $\lan G \ran_{\La(\ell+1), m} = 0$  for any
$0 \le  m \le (\ell+1)_*^d \equiv
|\La(\ell+1)|$.  Then, noting $\| J\|_\infty\le \|\partial \la\|_\infty$,
we see that (\ref{3.6-A})  (except for $Q_{1,\ell}$ and $Q_{2,N}$) is bounded by
the sum in $i,j$ (note $G=G_{ij}$) of
\begin{align}  \label{eq:3.7-0}
& \b T \|\partial \la\|_\infty^2 d (\ell+1)_*^d  
\sup_{m,\zeta\in \mathcal{X}_{\La(\ell+1)^c} }
   \lan G, (-L_{\La(\ell+1),\zeta})^{-1} G 
\ran_{\La(\ell+1), m}  \\
& \qquad \qquad
+ \frac{C}{\b}K + \b^2 Q_{N,\ell+1}^{(1)}(\partial\la,G).
\notag
\end{align}
Therefore, recalling the spectral gap of  $L_{\La(\ell),\zeta}$ given in 
\eqref{eq:3-SG}, that is, $(- L_{\La(\ell+1),\zeta})^{-1}  \le C_0^{-1}(\ell+1)^2$, 
the first term of (\ref{eq:3.7-0}) is bounded by
\begin{equation}  \label{3.7-A}
C\b \|\partial \la\|_\infty^2 \ell^{d+2}  \lan G^2\ran_{\La(\ell+1), m}.                                                 
\end{equation}
However, by applying  Taylor's formula to  $P_{ij}$, we have
$$
P_{ij}(\bar{\eta}_{0}^\ell) - P_{ij}(\bar{\eta}_{-e_i}^\ell)
 = D_{ij}(\bar{\eta}_{0}^\ell) \fa_{i,\ell}(\eta) 
+ \frac12 D_{ij}'(\rho^*) \fa_{i,\ell}(\eta)^2
$$
with some  $\rho^* \in [0,1]$, where
\begin{align}  \label{eq:3fa}
\fa_{i,\ell}(\eta)  & = \bar{\eta}_{0}^\ell - \bar{\eta}_{-e_i}^\ell  \\
& = \ell_*^{-d}  \left\{  \t_{\ell e_i} \widetilde{A}_{\ell,i} -
\t_{- (\ell+1)e_i} \widetilde{A}_{\ell,i} \right\}.  \notag
\end{align}
Therefore, after cancelling common terms 
$\ell_*^{-d}  D_{ij}(\bar{\eta}_{0}^\ell) \t_{\ell e_i} \widetilde{A}_{\ell,i}$,
we have
\begin{align}  \label{eq:G-E}
G(\eta) & = \ell_*^{-d} \left\{  D_{ij}(\bar{\eta}_{0}^\ell)  -
D_{ij}(\bar{\eta}_{-e_i}^\ell) \right\} \t_{- (\ell+1)e_i} \widetilde{A}_{\ell,i}
- \frac12 D_{ij}'(\rho^*) \fa_{i,\ell}(\eta)^2 
       \\  
& = \ell_*^{-d} D_{ij}'(\rho^{**}) \fa_{i,\ell}(\eta)
    \t_{- (\ell+1)e_i} \widetilde{A}_{\ell,i}  
    - \frac12 D_{ij}'(\rho^*) \fa_{i,\ell}(\eta)^2
    \notag   \\
& = \ell_*^{-d} D_{ij}'(\rho^{**}) \fa_{i,\ell}(\eta)
    \big(\t_{- (\ell+1)e_i} \widetilde{A}_{\ell,i}  - \bar A_{\ell,m}\big)
         \notag  \\
& \hskip 10mm
+ \ell_*^{-d} D_{ij}'(\rho^{**}) \fa_{i,\ell}(\eta)
     \bar A_{\ell,m}  
    - \frac12 D_{ij}'(\rho^*) \fa_{i,\ell}(\eta)^2,
    \notag    
\end{align}
with some  $\rho^{**} \in [0,1]$, where $\bar A_{\ell,m} := 
\lan \widetilde{A}_{\ell,i}\ran_{\La(\ell+1),m} \in [0, \ell_*^{d-1}]$.
We have used  Taylor's
formula again.   Since  $D_{ij}'(\rho)$  is bounded due to 
\eqref{1.D} and 
\begin{equation}
\lan \{ \t_{- (\ell+1)e_i} \widetilde{A}_{\ell,i} - \bar A_{\ell,m}\}^4 \ran_{\La(\ell+1),m} 
   \le C_3 \ell^{2(d-1)},                                \label{3.8-A}
\end{equation}
\begin{equation}
\lan \fa_{i,\ell}(\eta)^4 \ran_{\La(\ell+1),m} \le C_3 \ell^{-2(d+1)},
                                                 \label{3.9-A}
\end{equation}
we obtain that (\ref{3.7-A}) is bonded by
$$
C\b \|\partial \la\|_\infty^2 \ell^{d+2} \{ \ell^{-2d} \sqrt{\ell^{-2(d+1)} \ell^{2(d-1)} }
+ \ell^{-d-3} + \ell^{-2(d+1)}\}
\le 3 C \b \|\partial \la\|_\infty^2  \ell^{-1}.
$$
Here, \eqref{3.8-A} (i.e.\ in CLT scaling) follows by using 
Lemma A.2 in Appendix of \cite{FUY} taking $n=(\ell+1)_*^d$ (indeed, the sum
of the terms of the form $\lan(\eta_1-m/n)^2(\eta_2-m/n)^2\ran (\le 1)$ gives 
the leading order),
\eqref{3.9-A} follows from \eqref{eq:3fa} and \eqref{3.8-A}, and the expectation
under $\nu_{\La(\ell+1),m}$ of the square of the middle term on the
right-hand side of \eqref{eq:G-E} is $O(\ell^{-d-3})$ from \eqref{3.9-A}.

For $Q_{N,\ell+1}^{(1)} \equiv Q_{N,\ell+1}^{(1)}(\partial\la,G)$ 
in \eqref{eq:3.7-0}, since $\| G\|_\infty \le C_4 \ell^{-2}$ 
from \eqref{eq:G-E}, $0\le \widetilde{A}_{\ell,i}\le \ell_*^{d-1}$ and
$|\fa_{i,\ell}(\eta)|\le \ell_*^{-1}$,  we have by \eqref{eq:3QNell}
in Lemma \ref{Lemma 3.3}
$$
|Q_{N,\ell+1}^{(1)}| \le C N^{-1} \ell^{(2d+4)} \ell^{-6} \|\partial \la\|_\infty^3
= C N^{-1} \ell^{(2d-2)} \|\partial \la\|_\infty^3.
$$

Summarizing these, we obtain the desired upper bound for 
$Q_{N,\ell}^{(3)}(\la)$.

\subsection{Proof of Proposition \ref{Proposition 5.1} (CLT variance)}
\label{sec:9.5}

Recall \eqref{eq-3.A}, \eqref{eq-3.B}, \eqref{eq-3.H}
for $A_\ell$, $B_\ell$, $H_\ell$ and $\Psi_i = \eta_{e_i} - \eta_0$ 
given above \eqref{eq-3.A}.  Recall also $F\in \mathcal{F}_0^d$, 
$\th, \tilde{\th} \in \R^d: |\th| = |\tilde\th|=1$,
$m\in [0,\ell_*^d]\cap\Z$  and  $\zeta\in\mathcal{X}_{\La(\ell)^c}$.
The proof is divided into three steps.

{\it Step 1.} (Known formulas)
We use the following two identities shown in the proof of Proposition 5.1 
of \cite{FUY}:
\begin{align}  \label{5.3}
   & \ell_*^{-d}\Delta_{\ell,m,\zeta} \big(\th\cdot(B_\ell-H_\ell)\big)   \\       
 & =  \frac1{2}\ell_*^{-d} \sum_{b=\tau_x e^*_i \in (\La(\ell))^*} 
\;\;\bigg\lan c_b \Big[ \th\cdot e_i (\xi_{x+e_i} - \xi_x) 
   - \pi_b \Big( \sum_{y\in\La(\ell-n)} \t_y F \cdot\th \Big) \Big]^2 
      \bigg\ran_{\!\!\La(\ell),m},
      \notag
\end{align}
and
\begin{align}  \label{5.4}
& \ell_*^{-d} \Delta_{\ell,m,\zeta} \big(\tilde{\th}\cdot A_\ell, 
    \th\cdot(B_\ell-H_\ell)\big)         \\  
& = \ell_*^{-d} \bigg\lan 
\Big\{\sum_{i=1}^d \tilde{\th}_i  \sum_{y:\tau_y e^*_i\in (\La(\ell))^*}
  \t_y\Psi_i \Big\}    
   \Big\{ \sum_{x\in\La(\ell)} (\th\cdot x)\xi_x + \sum_{x\in\La(\ell-n)} 
     \t_x (\th\cdot F)\Big\}
      \bigg\ran_{\La(\ell),m}       \notag   \\
& = \ell_*^{-d}
   \sum_{i=1}^d \tilde{\th}_i  \sum_{y:\tau_y e^*_i\in (\La(\ell))^*}
   \sum_{x\in\La(\ell)}
\Big\lan  \t_y\Psi_i 
   \big\{  (\th\cdot x)\xi_x + 1_{x\in \La(\ell-n)} 
     \t_x (\th\cdot F)\big\}
      \Big\ran_{\La(\ell),m}.        \notag
\end{align}
As we noted below Proposition \ref{Proposition 5.1} in Section \ref{sec:5.3-D},
these are shown noting that $B_{\ell,\zeta}-H_{\ell,\zeta,F}$ is represented as
$L_{\La(\ell),\zeta} G$ for some function $G$ and therefore the operator 
$L_{\La(\ell),\zeta}$ cancels with $(-L_{\La(\ell),\zeta})^{-1}$ in the 
CLT variance and the CLT covariance.

{\it Step 2.} (Proof of \eqref{5.1})
To show \eqref{5.1}, we first replace the expectation on the right-hand side
of \eqref{5.3} under the canonical equilibrium measure $\nu_{\La(\ell),m}$
by that under the grandcanonical equilibrium measure $\nu_\rho$ with
$\rho=m/\ell_*^d$:
\begin{align}  \label{5.5}
   \frac1{2} \ell_*^{-d} \sum_{b=\tau_x e^*_i \in (\La(\ell))^*} 
\;\;\bigg\lan c_b \Big[ \th\cdot e_i (\xi_{x+e_i} - \xi_x) 
   - \pi_b \Big( \sum_{y\in\La(\ell-n)} \t_y F \cdot\th \Big) \Big]^2 
      \bigg\ran(m/\ell_*^d).
\end{align}
By the equivalence of ensembles \eqref{eq:3EE} (with $|\De| \le r(F)^d$)
and recalling $|\th|=1$, 
the error for this replacement is bounded by
$$
Cr(F)^d(1+\vertiii{F}_{0,\infty}^2) \ell^{-d}.
$$
Then, if $b=\tau_x e^*_i \in (\La(\ell-2n))^*$, the sum in $y\in \La(\ell-n)$ 
in \eqref{5.5} can be replaced by the sum in $y\in \Z^d$,
Thus, by the translation invariance of $\nu_\rho$, 
\eqref{5.5} is equal to
\begin{align*}
   \frac1{2}\ell_*^{-d} \sum_{i=1}^d \sharp& \{b\in (\La(\ell-2n))^*: i\text{th directed}\}  \\
   & \times\bigg\lan c_{e_i^*} \Big[ \th\cdot e_i (\xi_{e_i} - \xi_0) 
   - \pi_{e_i^*} \Big( \sum_{y\in\Z^d} \t_y F \cdot\th \Big) \Big]^2 
      \bigg\ran(m/\ell_*^d),
\end{align*}
with an error bounded by
$$
C(1+\vertiii{F}_{0,\infty}^2)
\frac{|(\La(\ell))^* \setminus (\La(\ell-2n))^*|}{\ell_*^d} 
\le C'(1+\vertiii{F}_{0,\infty}^2)n \ell^{-1}.
$$
Finally, noting that $\sharp\{b\in (\La(\ell-2n))^*: i\text{th directed}\}/\ell_*^d
= \frac{(2(\ell-2n))^d}{(2\ell+1)^d} = 1+O(\frac{n}\ell)$ and recalling
$n=r(F)+1$, the above expression
can be replaced by
\begin{align}  \label{5.6}
   \frac1{2} \sum_{i=1}^d 
    \;\;\bigg\lan c_{e_i^*} \Big[ \th\cdot e_i (\xi_{e_i} - \xi_0) 
   - \pi_{e_i^*} \Big( \sum_{y\in\Z^d} \t_y F \cdot\th \Big) \Big]^2 
      \bigg\ran(m/\ell_*^d),
\end{align}
with an error bounded by
$$
Cr(F) (1+\vertiii{F}_{0,\infty}^2) \ell^{-1}.
$$
However, by the identity at the bottom of p.\ \!22 of \cite{FUY}, \eqref{5.6}
is equal to $\frac12 \th\cdot \widehat{c}(m/\ell_*^d;F)\th$ and we obtain the estimate
\eqref{5.1}.

{\it Step 3.} (Proof of \eqref{5.2})
To show \eqref{5.2}, we first note the exchangeability under the canonical
measure $\nu_{\La(\ell),m}$, that is, $\eta$ and $\eta^{x,y}$ have the same
distributions under $\nu_{\La(\ell),m}$ for $x,y\in \La(\ell)$,
$x\not=y$.  This is easily seen from the uniformity of $\nu_{\La(\ell),m}$ or
the invariance of $\sum_{z\in \La(\ell)} \eta_z$
under the transform $\eta\mapsto \eta^{x,y}$.
In the expression  \eqref{5.4}, the right-hand side looks like $O(\ell^{2d}/\ell^d)$, 
but due to the exchangeability under $\eta\mapsto \eta^{y,y+e_i}$, we have
$$
\lan  \t_y\Psi_i \, \xi_x \ran_{\La(\ell),m} 
= \lan  (\eta_{y+e_i}-\eta_y) \, \xi_x \ran_{\La(\ell),m} =0
$$
if $x\not= y, y+e_i$ and
$$
\lan  \t_y\Psi_i \, \t_x (\th\cdot F)\ran_{\La(\ell),m} = 0
$$
if $y, y+e_i \notin \t_x (\text{supp} F)$, i.e., $y-x, y-x+e_i \notin \text{supp} F$,
especially, if $|y-x|\ge r(F)+1$.

In particular, the first part of the sum \eqref{5.4} (related to $\xi_x$)
is equal to
\begin{align*}
 \ell_*^{-d} &
   \sum_{i=1}^d \tilde{\th}_i  \sum_{y:\tau_y e^*_i\in (\La(\ell))^*}
   \Big\lan  \t_y\Psi_i 
   \big\{  (\th\cdot y)\xi_y + (\th\cdot y+e_i)\xi_{y+e_i} \big\}
      \Big\ran_{\La(\ell),m} \\
  &= \ell_*^{-d}
   \sum_{i=1}^d \tilde{\th}_i  \sum_{y:\tau_y e^*_i\in (\La(\ell))^*}
      \big\{  -(\th\cdot y) + (\th\cdot y+e_i) \big\}
      \chi_{\La(\ell),m} \\ 
& = \frac{(2\ell)^d}{(2\ell+1)^d} (\tilde{\th}\cdot \th) \chi_{\La(\ell),m},
\end{align*}
by noting that both $x=y$ and $y+e_i$ satisfy $x\in \La(\ell)$ and then
recalling the uniformity of $\nu_{\La(\ell),m}$, where
$$
\chi_{\La(\ell),m} :=- \lan \Psi_i \xi_0\ran_{\La(\ell),m}
\; \big( =-\lan (\xi_{e_i}- \xi_0)\xi_0 \ran_{\La(\ell),m}\big).
$$
However, by the equivalence of ensembles \eqref{eq:3EE}, we have
$$
|\chi_{\La(\ell),m} - \chi(m/\ell_*^d)| \le C \ell^{-d}.
$$
Therefore, the first part of the sum \eqref{5.4} behaves as
$$
 (\tilde{\th}\cdot \th) \chi(m/\ell_*^d) + O(1/\ell).
$$
 
On the other hand, as we pointed out above, the second part of the sum \eqref{5.4}
(related to $F$)
can be restricted to the terms of $x\in \La(\ell): |x-y|\le r(F)$.  Then, applying
the equivalence of ensembles
\eqref{eq:3EE} (with $|\De| \le r(F)^d$), we can replace it by
\begin{align} \label{5.8}
\ell_*^{-d}
   \sum_{i=1}^d \tilde{\th}_i  \sum_{y:\tau_y e^*_i\in (\La(\ell))^*}
   \sum_{x\in\La(\ell): |x-y|\le r(F)}
\Big\lan  \t_y\Psi_i \,
   1_{x\in \La(\ell-n)} 
     \t_x (\th\cdot F)
      \Big\ran (m/\ell_*^d)
\end{align}
with an error bounded by
$$
C r(F)^{2d} \| F\|_\infty \ell^{-d}.
$$
Denoting $\rho=m/\ell_*^d$, since $\lan  \t_y\Psi_i \; \t_x (\th\cdot F)\ran_\rho =0$
if $|x-y|\ge r(F)+1$, we can drop the condition $|x-y|\le r(F)$ in the sum \eqref{5.8}
and obtain for each $i$
\begin{align*}
\ell_*^{-d} 
 &    \sum_{y:\tau_y e^*_i\in (\La(\ell))^*}
   \sum_{x\in\La(\ell-n): |x-y|\le r(F)}
\lan  \t_y\Psi_i \; \t_x (\th\cdot F)\ran_\rho  \\
& = \ell_*^{-d} 
     \sum_{y:\tau_y e^*_i\in (\La(\ell))^*}
   \sum_{x\in\La(\ell-n)}
\lan  \t_y\Psi_i \; \t_x (\th\cdot F)\ran_\rho
=0.
\end{align*}
Here, to see that the last sum vanishes, we first observe
$$
\sum_{y:\tau_y e^*_i\in (\La(\ell))^*} \t_y\Psi_i 
= \sum_{y:\tau_y e^*_i\in (\La(\ell))^*} (\eta_{y+e_i}-\eta_y)
= \sum_{z_+\in \partial_+^i\La(\ell)} \eta_{z_+}
-\sum_{z_-\in \partial_-^i\La(\ell)} \eta_{z_-},
$$
where $\partial_\pm^i\La(\ell) = \{z=(z_j)_{j=1}^d; z_i = \pm\ell, 
|z_j|\le \ell \, (j\not=i)\}$,
and $\t_x (\th\cdot F)$ is $\La(\ell-1)$-measurable for 
$x\in \La(\ell-n)$.
In particular, these variables are independent under $\nu_\rho$.
Since $\lan \eta_{z_+} -\eta_{z_-}\ran_\rho=0$, we see that the 
last sum vanishes.

Summing up these calculations and bounds, we get \eqref{5.2}.
This completes the proof of Proposition \ref{Proposition 5.1}.

\subsection{Proof of Lemma \ref{Lem:2.3} ($g(t)\le 0$ for small $\de>0$)}
\label{sec:9.6}

Recall the definition of $g(t) = g_{\de,K}(t)$ given at the beginning of
Section \ref{sec:6.3-D}.  In this definition, 
the large deviation rate function $I(u;\la)$ has a bound
\begin{equation}  \label{eq:I-I}
I(u;\la) \ge C \{u-\bar\rho(\la)\}^2,\quad \rho\in [0,1],
\end{equation}
for some $C>0$, since 
$\frac{\partial^2 I}{\partial u^2}(u;\la) = \frac1{\chi(u)}> 0$ 
and $I(u;\la) = 0$ if and only if  $u=\bar\rho(\la)$. 
Note that $I(u;\la)$ is determined only
from $\nu_\la$ and especially it doesn't depend on $K$.
In particular, we have 
\begin{equation}
I(u;\la(t,v)) \ge C \{u-\rho(t,v)\}^2, \label{2.7}
\end{equation}
for some $C>0$.

On the other hand, we have  $\si(\rho(t,v);t,v) = 0$  and moreover,
noting that $\frac{\partial}{\partial u} \si_G(\rho(t,v);t,v)
= \frac{1}{\chi(\rho(t,v))}f(\rho(t,v))$,
\begin{align}
\frac{\partial \si}{\partial u} (\rho(t,v);t,v)   &
  =  - \dot{\la}(t,v) + \Tr \left(\partial^2 \la(t,v) D(\rho(t,v))
   \right)   \label{2.8}  \\
& + \frac12 \left(\partial\la(t,v), \widehat{c}'(\rho(t,v))
     \partial\la(t,v) \right) + \frac{K}{\chi(\rho(t,v))}f(\rho(t,v))= 0,  \notag
\end{align}
where  $\widehat{c}'$  denotes the matrix obtained by differentiating
each element of  $\widehat{c}$ in $\rho$.  In fact, one can easily 
recognize that (\ref{2.8}) is equivalent to the hydrodynamic equation (\ref{1.6}) 
itself by noting \eqref{2.9}.
Therefore, by Taylor's formula
$$
\si(u;t,v) = \frac12 \frac{\partial^2 \si}{\partial u^2}(u_0;t,v) 
  \{u-\rho(t,v)\}^2
$$
for some  $u_0 \in [0,1]$; note that \eqref{1.D} implies
$\si \in C^\infty$  as a function of  $u$.  However, $\frac{\partial^2 }{\partial u^2}
(\si-K\si_G)(u;t,v)$ is bounded on $[0,1]\times[0,T]\times \T^d$ and also
independent of $K$, while
\begin{align*}
\Big|K\frac{\partial^2 \si_G}{\partial u^2}(u;t,v)\Big|
= \Bigg| K \frac{\partial^2 }{\partial u^2}  \bigg[
\Big( \frac{\lan c^+\ran_u}{\rho(t,v)}      
  -  \frac{\lan c^-\ran_u}{1-\rho(t,v)}\Big)\{ u-\rho(t,v)\} \bigg]\Bigg|
\le C_1 K,
\end{align*}
for some $C_1>0$, since $0<c\le \rho(t,v)\le 1-c<1$ for some
$c\in (0,1/2)$ by Lemma \ref{lem:max} below; also recall the comment
below  Lemma \ref{Lemma 3.2}.  Thus we obtain
$$
\frac{\de}K \cdot|\si(u;t,v)| \le \de \Big(C_1+\frac{C_2}K\Big) \{u-\rho(t,v)\}^2.
$$
This combined with (\ref{2.7}) 
completes the proof of (\ref{2.6}) for $g(t)$ recalling $K\ge 1$.

\section{Schauder estimates}  \label{sec:Schauder}

Here we give the proof of the Schauder estimates \eqref{eq:Schauder} for the 
solution $\rho=\rho(t,v)\equiv \rho_K(t,v)$ of \eqref{1.6}.  We apply the results 
given in \cite{Li96} for linear second order parabolic partial differential 
equations (PDEs).  See also 
\cite{LSU}, but an explicit dependence of the constants in 
the estimates on several data of the equation is not clearly indicated 
as in \eqref{7.2}, \eqref{7.3} below.

Before showing the Schauder estimates, we briefly note that the equation
\eqref{1.6} has a unique classical solution $\rho(t,v)\in C^{1,3}([0,\infty) \times
\T^d)$ under the condition $\rho_0\in C^4(\T^d)$ for its initial value;
recall \eqref{eq:1.12} in which $\rho_0\in C^5(\T^d)$ is assumed.
In fact, noting that \eqref{1.6} is a quasi-linear equation
in divergence form (0.1) in Chapter V of \cite{LSU}, p.\ \!417, by Theorem 6.1 of
\cite{LSU}, p.\ \!452, we see that it has a unique classical solution 
$\rho(t,v)\in C^{1,2}([0,\infty) \times\T^d)$  if $\rho_0\in C^3(\T^d)$.
Note that the consistency condition (6.3) assumed in Theorem 6.1 is
unnecessary in our setting, since $\T^d$ has no boundary.  To obtain
$C^{1,3}$-property, we consider the equation for the derivative
$\rho_{(k)}:= \partial_{v_k}\rho(t,v)$ for each fixed $k: 1\le k \le d$.
By a simple calculation, we see that $\rho_{(k)}$ satisfies the equation (0.1) of
\cite{LSU} with
\begin{align*}
& a_i(v,t,p) = \sum_{j=1}^d D_{ij}(\rho(t,v)) p_j
+ \sum_{j=1}^d D_{ij}'(\rho(t,v)) \partial_{v_k}\rho(t,v) \partial_{v_j}\rho(t,v), \\
& a(v,t) = -f'(\rho(t,v)) \partial_{v_k}\rho(t,v),
\end{align*}
where we regard that $\rho(t,v)$, $\partial_{v_k}\rho(t,v)$ and $\partial_{v_j}\rho(t,v)$ 
are all given.  Then, applying Theorem 6.1 of \cite{LSU} again, we see that
$\rho_{(k)}\in C^{1,2}([0,\infty) \times\T^d)$ if $\partial_{v_k}\rho_0\in C^3(\T^d)$
for each $k$.  This shows $\rho(t,v)\in C^{1,3}([0,\infty) \times
\T^d)$ under the condition $\rho_0\in C^4(\T^d)$.

Now we show the Schauder estimates.
To apply the PDE results of \cite{Li96}, i.e.\ the estimate \eqref{7.2} below
which requires the H\"older properties of several data in the PDE,
the first step is to show the H\"older continuity of the solution of the 
nonlinear PDE \eqref{1.6}. 

For this purpose, first in \eqref{1.6} dropping the term $K f(\rho(t,v))$ 
and replacing $D(\rho(t,v))$ by $D(t,v) = \{D_{ij}(t,v)\}_{1\le i,j \le d}$
which is symmetric matrix-valued Borel measurable function on
$[0,T]\times \T^d$ satisfying a bound similar to \eqref{eq:1.9}:
$ c_* |\th|^2 \le (\th,D(t,v)\th) \le c^*|\th|^2$ for $\th\in \R^d$ with $c_*, c^*>0$, 
we consider a solution $u^{(0)}=u^{(0)}(t,v)$ of the linear parabolic PDE
\begin{align}  \label{eq:7D}
\partial_t u^{(0)} = \sum_{i,j=1}^d \partial_{v_i}\{ D_{ij}(t,v) \partial_{v_j} u^{(0)}\}, 
\quad t\ge 0, v \in \T^d,
\end{align}
with initial value $u_0$.  Let $p(t,v;s,\xi)$, $v,\xi\in \T^d$,
$0\le s <t\le T$  be the associated fundamental solution.
Then, we have the following Nash H\"older estimates; cf.\ \cite{St},
\cite{QX} and related estimates
can also be found in \cite{LSU}.

\begin{lem}  \label{lem:7.1}
There exist $\si\in (0,1]$ and $C>0$ such that
\begin{align}  \label{7.0-A}
&  |u^{(0)}(t_1, v_1) - u^{(0)}(t_2, v_2)| 
\le C \| u_0\|_\infty \Big( \frac{|t_1-t_2|^{1/2} \vee
 |v_1-v_2|}{(t_1\wedge t_2)^{1/2}} \Big)^\si,
\intertext{for $t_1, t_2 \in (0,T], \; v_1, v_2 \in \T^d$ and }
&  |p(t_1, v_1; 0,\xi_1) - p(t_2, v_2;0,\xi_2)| \le \frac{C}{(t_1\wedge t_2)^{d/2}}
\Big( \frac{ |t_1-t_2|^{1/2} \vee  |v_1-v_2| \vee |\xi_1-\xi_2|}
{(t_1\wedge t_2)^{1/2}} \Big)^\si,
\label{7.0-B}
\end{align}
for $t_1, t_2 \in (0,T], \; v_1, v_2, \xi_1, \xi_2 \in \T^d$.
\end{lem}

\begin{proof}
To show \eqref{7.0-A}, we separate into two cases: $(t_1 \vee t_2) \ge 2
(t_1 \wedge t_2)$ and $(t_1 \vee t_2) < 2(t_1 \wedge t_2)$.
In the case $(t_1 \vee t_2) \ge 2(t_1 \wedge t_2)$, since $|t_1-t_2|= 
(t_1 \vee t_2) - (t_1 \wedge t_2) \ge (t_1 \wedge t_2)$, \eqref{7.0-A} 
is obvious with $C=2$, noting $\|u^{(0)}(t)\|_\infty \le \|u_0\|_\infty$ by the 
maximum principle; see Lemma \ref{lem:max} below with $f\equiv 0$.

In the other case $(t_1 \vee t_2) <2 (t_1 \wedge t_2)$, we take $t_0 = 
2(t_1 \wedge t_2)$, $R= \sqrt{2} (t_1 \wedge t_2)^{1/2}$, $\de = \sqrt{2/3}$ and
$x_0=v_1$ in Theorem 4.1 of \cite{QX} (or Theorem II.1.8 of \cite{St},
although it is stated only for the time-homogeneous case).  
Note that $t_0-R^2= 0$ and $t_1, t_2 \in (t_0-(\de R)^2, t_0) = 
\big( 2/3(t_1 \wedge t_2),2(t_1 \wedge t_2)\big)$.  Then, by  Theorem 4.1 
of \cite{QX}, \eqref{7.0-A} holds if $|v_1-v_2| \le \sqrt{4/3}(t_1 \wedge t_2)^{1/2}$.
If $|v_1-v_2| \ge \sqrt{4/3}(t_1 \wedge t_2)^{1/2}$, \eqref{7.0-A} is
obvious with $C=2$.

To show \eqref{7.0-B}, taking $\de=(t_1\wedge t_2)^{1/2}$ in Corollary 4.2 
of \cite{QX} (or Corollary II.1.9 of \cite{St}), this holds if 
$|v_1-v_2| \vee |\xi_1-\xi_2| \le (t_1\wedge t_2)^{1/2}$.
Conversely, if $|v_1-v_2| \vee |\xi_1-\xi_2| \ge (t_1\wedge t_2)^{1/2}$,
by Aronson's Gaussian upper bound (see Theorem 1.2 of \cite{QX}), we have
\begin{align*}
|p(t_1, v_1; 0,\xi_1) - p(t_2, v_2;0,\xi_2)| 
& \le p(t_1, v_1; 0,\xi_1) + p(t_2, v_2;0,\xi_2) \\
& \le C t_1^{-d/2} +  C t_2^{-d/2} \le 2C (t_1\wedge t_2)^{-d/2}
\end{align*}
and this is bounded by the right-hand side of \eqref{7.0-B} (with $C$ replaced
by $2C$).
\end{proof}

This lemma is used to show the H\"older continuity of the solution
$\rho(t,v)\equiv \rho_K(t,v), K\ge 1$ of \eqref{1.6} with $\rho(0)= \rho_0$.
The regularity of $\rho_0$ allows to remove the singularity of $\rho(t,v)$ 
near $t=0$; see also the comment above Lemma \ref{lem:6.3}.

\begin{lem}  \label{lem:7.2}
Assume the condition \eqref{eq:1.12} for $\rho_0$.
Then, we have
\begin{align}  \label{7.1}
|\rho(t_1, v_1) - \rho(t_2, v_2)| \le CK\big( |t_1-t_2|^{\si/2} + |v_1-v_2|^\si\big),
\quad t_1, t_2 \in [0,T], \; v_1, v_2 \in \T^d,
\end{align}
for some $\si \in (0,1]$.  Since $D(\rho)=\{D_{ij}(\rho)\}_{1\le i,j\le d} 
\in C^\infty([0,1])$,
we have a similar H\"older estimate for $D_{ij}(\rho(t,v))$.
\end{lem}

\begin{proof}
The proof is similar to that of Corollary 2.6 of \cite{FS} (in the discrete setting).
First, let $D(t,v)$ be as above and let $g(t,v)$ be a bounded Borel measurable 
function.  We consider the linear PDE
\begin{align*}
\partial_t u = \sum_{i,j=1}^d \partial_{v_i}\{ D_{ij}(t,v) \partial_{v_j} u\}
+ g(t,v), \quad t\ge 0, v \in \T^d,
\end{align*}
with $u(0)=u_0$.  Then, by Duhamel's formula, we have
$$
u(t,v) = u^{(0)}(t,v)+ \int_0^t ds \int_{\T^d} g(s,\xi)p(t,v;s,\xi) d\xi,
$$
where $u^{(0)}(t,v)$ is the solution of \eqref{eq:7D}. Then, by Lemma \ref{lem:7.1}
(actually only by \eqref{7.0-A}),
similar to the proof of Theorem 2.2 of \cite{FS}, we get
\begin{align} \label{7.H-Q}
|u(t_1, v_1) - u(t_2, v_2)| \le C (\|u_0\|_\infty + \| g\|_\infty)
 \Big( \frac{|t_1-t_2|^{1/2} +
 |v_1-v_2|}{(t_1\wedge t_2)^{1/2}} \Big)^\si,
\end{align}
for $t_1, t_2 \in (0,T], \; v_1, v_2 \in \T^d$.

To improve the regularity near $t=0$, we apply a  trick similar to the one used 
in the proof of 
Theorem 2.5 of \cite{FS}. We supplementarily consider the heat equation on $\T^d$
$$
\partial_s U = \De U, \quad s\in (0,1]
$$
with $U(0)=\rho_0$, and set $\hat U(t) := U(1-t)$ and $\hat h(t,v) := 
- \De \hat U(t,v)$ for $t\in [0,1)$.  Then, $\hat U(t,v)$ satisfies the equation
$$
\partial_t \hat U = - \De \hat U = \De \hat U + 2 \hat h.
$$
Since $h(t):= \hat h(1-t)= - \De U(t)$ satisfies $\partial_t h(t) = \De h(t)$
with $h(0)= -\De U(0) = -\De \rho_0$, we see by the maximum principle that
$\| h(t)\|_\infty \le \|\De\rho_0\|_\infty < \infty$ by \eqref{eq:1.12}.

To apply \eqref{7.H-Q} for our equation \eqref{1.6},
define $\hat \rho(t,v)$, $\{\hat D_{ij}(t,v)\}$ and $\hat g(t,v)$, respectively, by
\begin{align*}
& \hat \rho(t,v) = \left\{
\begin{aligned}
\rho(t-1,v),& \quad t\in [1,T+1], \\
\hat U(t,v),& \quad t \in [0,1).
\end{aligned} \right. \\
& \hat D_{ij}(t,v) = \left\{
\begin{aligned}
D_{ij}(\rho(t-1,v)),& \quad t\in [1,T+1], \\
\de_{ij},& \quad t \in [0,1),
\end{aligned} \right.  \\
& \hat g(t,v) = \left\{
\begin{aligned}
K f(\rho(t-1,v)),& \quad t\in [1,T+1], \\
2\hat h(t,v),& \quad t \in [0,1),
\end{aligned} \right.
\end{align*}
where $\rho(t,v)\equiv \rho_K(t,v)$ is the solution  of  \eqref{1.6}
with initial value $\rho_0$.  Then, these functions satisfy
$$
\partial_t \hat \rho = \sum_{i,j=1}^d 
\partial_{v_i}(\hat D_{ij}(t,v) \partial_{v_i}\hat \rho) + \hat g(t,v),
$$
with $\hat \rho(0,v) = U(1)$.  Note that $\|U(1)\|_\infty \le \|\rho_0\|_\infty$.
So, noting that $\|\hat g\|_\infty \le K\|f\|_\infty+ 2 \|\De\rho_0\|_\infty$ 
and $\{\hat D_{ij}\}$ satisfies the uniform positive-definiteness and boundedness
conditions, from \eqref{7.H-Q} with $T$ replaced by $T+1$, we have
\begin{align*}
|\hat\rho(t_1, v_1) - \hat\rho(t_2, v_2)| \le C (\|\rho_0\|_\infty + \| \hat g\|_\infty)
 \Big( \frac{|t_1-t_2|^{1/2} +
 |v_1-v_2|}{(t_1\wedge t_2)^{1/2}} \Big)^\si,
\end{align*}
for $t_1, t_2 \in (0,T+1], \; v_1, v_2 \in \T^d$.
The desired bound \eqref{7.1} is obtained by restricting the above estimate 
to $[1,T+1]$ and shifting it by $1$ in time.
\end{proof}

Now we apply Theorem 4.8 of \cite{Li96}, in particular the detailed estimate given 
on p.\ \!59 for a solution $u$ of the linear equation (4.22) of \cite{Li96} cited below.
Before explaining it, let us first recall several H\"older norms for functions 
on $\Om =[0,T]\times \T^d$
from  pp.\ \!46--47 of \cite{Li96}.  The sup-norm on $\Om$ is denoted by
$\|\cdot\|_\infty = |\cdot|_0\equiv |\cdot|_0^{(0)}$ and let $\a\in (0,1]$. 
The (semi) norms $[u]_\a^{(b)}, b=0,1, |u|_\a^{(1)}$ and 
$|u|_{1+\a}^*\equiv |u|_{1+\a}^{(0)}$ 
are defined for functions $u$ on $\Om$ which may have a singularity like $t^{-1/2}$
or others near $t=0$, respectively, by
\begin{align*}
& [u]_\a^{(b)}= \sup_{X\not= Y} (t\wedge s)^{(\a+b)/2} \frac{|u(X)- u(Y)|}
{|X-Y|^\a}, \quad b=0,1, \\
& |u|_\a^{(1)} = |u|_0^{(1)}+ [u]_\a^{(1)},   \quad |u|_0^{(1)} = \sup_\Om t^{1/2} |u(X)|, \\
& |u|_{1+\a}^{(0)} 
 := \| u\|_\infty + |\partial_v u|_0^{(1)} + [u]_{1+\a}^{(0)}+ \lan u \ran_{1+\a}^{(0)} \\
& \hskip 11.4mm
\equiv  \| u\|_\infty + \sup_\Om t^{1/2} |\partial_v u(X)|
+ \sup_{X\not= Y} (t\wedge s)^{(1+\a)/2} \frac{|\partial_v u(X)-  \partial_v u(Y)|}
{|X-Y|^\a} \\
& \hskip 20mm
+ \sup_{X\not= Y, x=y} (t\wedge s)^{(1+\a)/2} \frac{|u(X)- u(Y)|}
{|X-Y|^{1+\a}},
\end{align*}
where $X=(t,v),Y=(s,v') \in \Om$ and $|X-Y| = \max\{ |x-y|, |t-s|^{1/2}\}$.
Note that $d(X,Y) = \min\{d(X),d(Y)\}= \min\{t^{1/2},s^{1/2}\}= (t\wedge s)^{1/2}$
and $\lan u\ran_\a^{(1)}=0$ in the norm $|u|_\a^{(1)}$ on p.\ \!47 of \cite{Li96}.
The norm $|u|_{1+\a}$ for functions $u$ without singularity near $t=0$ is defined 
on p.\ \!46 of \cite{Li96} as 
\begin{align*}  
|u|_{1+\a}
 := \| u\|_\infty + \|\partial_v u\|_\infty
+ \sup_{X\not= Y} \frac{|\partial_v u(X)-  \partial_v u(Y)|}
{|X-Y|^\a} 
+ \sup_{X\not= Y, x=y} \frac{|u(X)- u(Y)|}{|X-Y|^{1+\a}}.
\end{align*}

The detailed estimate mentioned above 
clarifies the dependence of the constants in the estimate on 
various data; see \eqref{7.2} below.  The equation (4.22) of \cite{Li96} is defined 
as follows on $\Om$ and written by dropping the sums in $i, j$ as
\begin{align}  \tag{4.22}
- \partial_t u + \partial_{v_i} \big( a^{ij} \partial_{v_j}u + b^i  \partial_{v_i} u \big)
+ c^i  \partial_{v_i} u +c^0 u =  \partial_{v_i}  f^i+g,
\end{align}
where $a^{ij}, b^i, f^i$ are $\a$-H\"older continuous and $c^i, c^0, g$ are
bounded measurable functions on $\Om$ that satisfy
\begin{align}  \tag{4.20a}
& \la |\xi|^2\le a^{ij}\xi_i\xi_j \le \La |\xi|^2,
\quad [a^{ij}]_\a^{(0)} \le A,\\
& |b^i|_\a^{(1)} \le B,   \tag{4.20b}  \\
& \| c^i\|_\infty \le c_1, \quad \| c^0\|_\infty \le c_2,  \tag{4.20c}  \\
& | f^i|_\a^{(1)} \le F, \quad \| g\|_\infty \le G,  \tag{4.21}
\end{align}
for some positive constants $\la, \La, A, B, c_1, c_2, F$ and $G$. 
Then, we have
\begin{align}  \label{7.2}
|u|_{1+\a}^* \; \big( \equiv |u|_{1+\a}^{(0)} \big)
\le C(d,\a,\la,\La) (M \|u\|_\infty + F+G)
\end{align}
with
\begin{align}  \label{7.3}
M=[1+A+c_1]^{(1+\a)/\a} + \sum_i (|b^i|_0^{(1)})^{1+\a} + c_2 + \sum_i [b^i]_\a^{(1)}.
\end{align}
In our applications, we will take $b^i=c^i=0$.
Note that the above conditions (4.20c) and (4.21) are
simplifications of those in \cite{Li96} by noting that $\|\cdot\|_{p,\de}^{(b)} \le 
C \|\cdot\|_\infty$ for $b=1,2$, but these will suffice for our purpose.

First, we consider \eqref{1.6} to be a linear equation (4.22) for $u=\rho$
with $a^{ij}= D_{ij}(\rho(t,v))$,  $g=-Kf(\rho(t,v))$ where $\rho(t,v)$ is already given
and  $b^i=c^i=c^0=f^i=0$.   Then, since \eqref{7.1} implies
$[a^{ij}]_\si^{(0)} \le CK$ (note that $(t\wedge s)^{\a/2} \le T^{\a/2}$ is bounded), 
one can take $\a=\si$, $A=CK$  in 
the estimate (4.20a) and $G=CK$ in (4.21).  So, noting $M=[1+A]^{(1+\si)/\si}
\le CK^{(1+\si)/\si}$ and $\| u\|_\infty \le 1$, we obtain by \eqref{7.2} the estimate:
\begin{align}  \label{7.4}
|\rho|_{1+\si}^* \le CK^{(1+\si)/\si},
\end{align}
for the solution $\rho$ of \eqref{1.6}.  Here and in the following, we use the
comparison theorem for \eqref{1.6}, that is, under our assumption
\eqref{eq:1.12} for the initial value $\rho_0$, we have
$0<c\le \rho(t,v) \le 1-c <1$ for some $c\in (0,1/2)$;
see Lemma \ref{lem:max} below.

Considering the extended system in time similar to the proof of 
Lemma \ref{lem:7.2} (see also Theorem 4.2 of \cite{FS}), the regularity 
of the initial value $\rho_0(v)$ improves  the estimate \eqref{7.4} into that
without singularity near $t=0$.  In fact, we have the following lemma.

\begin{lem}  \label{lem:6.3}
For the norm without singularity, we have
\begin{align}  \label{7.4-BQ}
|\rho|_{1+\si} \le CK^{(1+\si)/\si}.
\end{align}
\end{lem}

\begin{proof}
Since we need to extend $a^{ij}$ to $\hat a^{ij}$ while preserving its $\si$-H\"older
property, we modify the choice of $\hat D_{ij}$ in the proof of Lemma \ref{lem:7.2}.
We take $\hat b^i=\hat c^i = \hat c^0 = \hat f^i=0$ and
define $\hat u(t,v)$, $\{\hat a^{ij}(t,v)\}$ and $\hat g(t,v)$, respectively, by
\begin{align}  \notag
& \hat u(t,v) = \left\{
\begin{aligned}
\rho(t-1,v),& \quad t\in [1,T+1], \\
\hat U(t,v),& \quad t \in [0,1).
\end{aligned} \right. \\   \label{eq:7.10-P}
& \hat a^{ij}(t,v) = \left\{
\begin{aligned}
D_{ij}(\rho(t-1,v)),& \quad t\in [1,T+1], \\
D_{ij}(\rho_0(v)),& \quad t \in [0,1),
\end{aligned} \right.  \\   \notag
& \hat g(t,v) = \left\{
\begin{aligned}
-K f(\rho(t-1,v)),& \quad t\in [1,T+1], \\
- 2\hat h(t,v),& \quad t \in [0,1),
\end{aligned} \right.
\end{align}
where $\hat U(t):= U(1-t)$ is defined from the solution $U(s), s\in [0,1]$
of the equation 
$$
\partial_s U = L_{\rho_0}U \equiv
\partial_{v_i}(D_{ij}(\rho_0(v))\partial_{v_j}U), \quad s\in (0,1],
$$
with
initial value $U(0)=\rho_0$, and $\hat h(t,v) := - L_{\rho_0} \hat U(t,v),
t\in [0,1]$.  Then, $\hat U(t)$ satisfies the equation
$$
\partial_t \hat U = -L_{\rho_0}\hat U
= L_{\rho_0}\hat U +2 \hat h.
$$
For $\hat g$, the measurability is sufficient. 

Note that $h(t) := \hat h(1-t) = - L_{\rho_0} U(t)$ satisfies the equation
$\partial_t h(t) = L_{\rho_0} h(t)$ and $h(0)=\hat h(1)= - L_{\rho_0} U(0)
=- L_{\rho_0} \rho_0$.
Thus, according to the maximum principle noting 
$$
L_{\rho_0}u= 
D_{ij}(\rho_0)\partial_{v_i}\partial_{v_j}u + D_{ij}'(\rho_0)\partial_{v_i}\rho_0
\cdot \partial_{v_j} u,
$$
we see 
$$
\|h(t)\|_\infty \le \|L_{\rho_0}\rho_0\|_\infty
\le C(\|\partial^2\rho_0\|_\infty + \|\partial\rho_0\|_\infty^2) < \infty,
$$
since $\rho_0\in C^2(\T^d)$ by our condition \eqref{eq:1.12}.
So we have $\|\hat g\|_\infty \le CK$.

Since $[\hat a^{ij}]_\si^{(0)} \le CK$,
\eqref{7.4} holds for the extended solution $\hat{u}(t,v)$ on
$[0,T+1]$.  Restricting this estimate to $[1,T+1]$,
\eqref{7.4-BQ} follows.
\end{proof}

Second, we take the derivative of \eqref{1.6} in $v_k$ and obtain by dropping
the sum in $i,j$
\begin{align}  \label{7.A}
\partial_t \rho_{v_k} 
& = \partial_{v_i} \Big( \partial_{v_k} \big\{ D_{ij}(\rho)\partial_{v_j}\rho \big\} \Big)
+ K f'(\rho) \rho_{v_k} \\
& = \partial_{v_i} \big\{ D_{ij}(\rho)\partial_{v_j}\rho_{v_k} \big\} 
- \partial_{v_i}  f^{i,(1)}(\rho)  
+ K f'(\rho) \rho_{v_k},          \notag    
\end{align}
where
$$
f^{i,(1)}(\rho) := - \sum_{j=1}^d D_{ij}'(\rho)\rho_{v_k}\rho_{v_j}.
$$
We consider this equation to be a linear equation (4.22) 
for $u=\rho_{v_k}$ with $a^{ij}= D_{ij}(\rho)$,  $b^i=c^i=0$,
$f^i= f^{i,(1)}(\rho)$ (here both $\rho_{v_k} = \rho_{v_k}(t,v)$ and $\rho_{v_j}
= \rho_{v_j}(t,v)$ are considered already given) and
$c^0=K f'(\rho)$.  We may add a bounded function $g$ (actually, $\hat{g}$)
to apply for the proof of the next Lemma \ref{lem:7.4}.
Then, one can take $A=CK$, $F= CK(K^{(1+\si)/\si})^2$ by using
\eqref{7.1} (for H\"older estimate for $D'(\rho)$) and
\eqref{7.4-BQ} (noting that we can drop $t^{1/2}$ and $(t\wedge s)^{(1+\a)/2}$ 
in the norm $|f^i|_\a^{(1)}$, since these are bounded) and $c_2= CK$.  Thus, 
by \eqref{7.2} with $G$ for $g$ (for later use, currently $g=0$),  
we obtain the estimate:
\begin{align}  \label{eq:6.u1+si}
|u|_{1+\si}^* \le C(M \|u\|_\infty+ K^{2(1+\si)/\si+1}+G)
\end{align}
with $M= [1+A]^{(1+\si)/\si} +K$.  Therefore, noting $\|u\|_\infty = 
\|\rho_{v_k}\|_\infty \le C K^{(1+\si)/\si}$ by \eqref{7.4-BQ}
and $G=0$, we have
\begin{align}  \label{7.5}
|\partial \rho|_{1+\si}^* \le CK^{2(1+\si)/\si+1}.
\end{align}

We can apply a similar argument to the proof of Lemma \ref{lem:6.3} 
to remove the singularity near $t=0$.

\begin{lem}  \label{lem:7.4}
For the norm without singularity, we have
\begin{align}  \label{7.5-BQ}
|\partial \rho|_{1+\si} \le CK^{2(1+\si)/\si+1}.
\end{align}
\end{lem}

\begin{proof}
We take $\{\hat a^{ij}(t,v)\}$ as in \eqref{eq:7.10-P} and $\hat b^i = \hat c^i=0$.
Then, for each $k$, we define $\hat u(t,v)$, $\{\hat f^i(t,v)\}$, $\hat c^0(t,v)$
and $\hat g(t,v)$, respectively, by
\begin{align}  \notag
& \hat u(t,v) = \left\{
\begin{aligned}
\rho_{v_k}(t-1,v),& \quad t\in [1,T+1], \\
\hat V(t,v),& \quad t \in [0,1).
\end{aligned} \right. \\  \label{eq:7.14-P}
& \hat f^i(t,v) = \left\{
\begin{aligned}
f^{i,(1)}(\rho(t-1,v)),
& \quad t\in [1,T+1], \\
f^{i,(1)}(\rho_0(v)),& \quad t \in [0,1),
\end{aligned} \right.  \\   \notag
& \hat c^0(t,v) = \left\{
\begin{aligned}
K f'(\rho(t-1,v)),& \quad t\in [1,T+1], \\
0,& \quad t \in [0,1),
\end{aligned} \right. \\   \notag
& \hat g(t,v) = \left\{
\begin{aligned}
0,& \quad t\in [1,T+1], \\
-2\hat h(t,v),& \quad t \in [0,1),
\end{aligned} \right.
\end{align}
where
$\hat V(t) \equiv \hat V_k(t):= V(1-t)$ is defined from the solution $V(s), s\in [0,1]$
of the equation 
$$
\partial_s V = L_{\rho_0}V - \partial_{v_i}f^{i,(1)}(\rho_0), \quad s\in (0,1],
$$
with
initial value $V(0)=\partial_{v_k}\rho_0$, and $\hat h(t,v) := - L_{\rho_0} \hat V(t,v),
t\in [0,1]$, where $L_{\rho_0}$ is the same as in the proof of Lemma \ref{lem:6.3}.
Then, $\hat V(t)$ satisfies the equation
$$
\partial_t \hat V = - L_{\rho_0} \hat V + \partial_{v_i}f^{i,(1)}(\rho_0)
= L_{\rho_0} \hat V + 2 \hat h + \partial_{v_i} f^{i,(1)}(\rho_0).
$$

Note that 
$h(t) := \hat h(1-t) = - L_{\rho_0} V(t)$ satisfies the equation
$$
\partial_t h(t) = L_{\rho_0} h(t)+L_{\rho_0}\partial_{v_i}f^{i,(1)}(\rho_0)
$$
and $h(0)=- L_{\rho_0} \partial_{v_k}\rho_0$.
Thus, by applying Duhamel's formula
$$
h(t) = e^{tL_{\rho_0}} h(0) + \int_0^t
e^{(t-s)L_{\rho_0}} L_{\rho_0}\partial_{v_i}f^{i,(1)}(\rho_0)ds
$$
with the semigroup $e^{tL_{\rho_0}}$ generated by $L_{\rho_0}$, which
is a contraction under the sup-norm by the maximum principle, 
we see 
$$
\|h(t)\|_\infty \le 
\|L_{\rho_0}\partial_{v_k}\rho_0\|_\infty + \|L_{\rho_0}\partial_{v_i}f^{i,(1)}(\rho_0)\|_\infty< \infty, \quad t\in [0,1],
$$ 
since $\rho_0\in C^4(\T^d)$ by the condition \eqref{eq:1.12}.
So $G=\|\hat g\|_\infty \le C$.

This shows \eqref{7.5-BQ} by applying \eqref{eq:6.u1+si} for the extended solution
defined on $[0,T+1]$.
\end{proof}

These two lemmas also imply an estimate for $\partial_t\rho$.
Indeed, expanding the right-hand side of \eqref{1.6} by noting $D_{ij}(\rho)
\in C^\infty([0,1])$, and then applying \eqref{7.4-BQ} and \eqref{7.5-BQ},
we obtain
\begin{align}  \label{7.6-A}
\|\partial_t \rho\|_\infty \le CK^{2(1+\si)/\si+1}.
\end{align}

Third, we take the derivative of \eqref{7.A} in $v_\ell$ and obtain
\begin{align}  \label{7.B}
\partial_t \rho_{v_k v_\ell} 
& = \partial_{v_i} \big\{ D_{ij}(\rho)\partial_{v_j}\rho_{v_k v_\ell} \big\} 
- \partial_{v_i}  f^{i,(2)}(\rho)
+ K f'(\rho) \rho_{v_k v_\ell} + K f''(\rho) \rho_{v_k} \rho_{v_\ell}, 
\end{align}
where
$$
f^{i,(2)}(\rho) = -\sum_{j=1}^d \big\{ D_{ij}'(\rho) \rho_{v_\ell} \rho_{v_j v_k}
+ D_{ij}''(\rho) \rho_{v_\ell} \rho_{v_k}\rho_{v_j}
+ D_{ij}'(\rho) \partial_{v_\ell} (\rho_{v_k} \rho_{v_j})\big\}.
$$
We consider this equation to be a linear equation (4.22) for $u=\rho_{v_k v_\ell}$ 
with $a^{ij}= D_{ij}(\rho)$,  $b^i=c^i=0$, $f^i= f^{i,(2)}$
(regarding all terms in $f^{i,(2)}$ already given), $c^0=K f'(\rho)$ and $g=
- K f''(\rho) \rho_{v_k} \rho_{v_\ell}$ (with all terms already given).  
Then, one can take $A=CK$, $F= CK^2 (K^{(1+\si)/\si})^3$ by using
\eqref{7.1}, \eqref{7.4-BQ} and \eqref{7.5-BQ},  $c_2= CK$ and 
$G= CK \cdot K^{2(1+\si)/\si}$.  
Thus, by \eqref{7.2},  we obtain the estimate:
\begin{align*}
|u|_{1+\si}^* \le C(M \|u\|_\infty+ K^{3(1+\si)/\si+2})
\end{align*}
with $M= [1+A]^{(1+\si)/\si} +K$.  Therefore, noting
$\|u\|_\infty = \|\rho_{v_k v_\ell}\|_\infty\le C K^{2(1+\si)/\si+1}$
by \eqref{7.5-BQ}, we have
\begin{align}  \label{7.5-C}
|\partial^2 \rho|_{1+\si}^* \le CK^{3(1+\si)/\si+2}.
\end{align}

We can again use the same argument as above to remove the singularity near
$t=0$.

\begin{lem}
For the norm without singularity, we have
\begin{align}  \label{7.5-CBQ}
|\partial^2 \rho|_{1+\si} \le CK^{3(1+\si)/\si+2}.
\end{align}
\end{lem}

\begin{proof}
Take $\{\hat a^{ij}(t,v)\}$ as in \eqref{eq:7.10-P}, $\hat c^0(t,v)$
as in \eqref{eq:7.14-P} and $\hat b^i=\hat c^i=0$.  Then, for each $k, \ell$,
we define $\hat u(t,v)$, $\{\hat f^i(t,v)\}$ and $\hat g(t,v)$, respectively, by
\begin{align}  \notag
& \hat u(t,v) = \left\{
\begin{aligned}
\rho_{v_k v_\ell}(t-1,v),& \quad t\in [1,T+1], \\
\hat V(t,v),& \quad t \in [0,1).
\end{aligned} \right. \\  \label{eq:7.18-P}
& \hat f^i(t,v) = \left\{
\begin{aligned}
f^{i,(2)}(\rho(t-1,v)),
& \quad t\in [1,T+1], \\
f^{i,(2)}(\rho_0(v)),
& \quad t \in [0,1),
\end{aligned} \right.  \\  \notag
& \hat g(t,v) = \left\{
\begin{aligned}
K f''(\rho(t-1,v))\rho_{v_k}(t-1,v)\rho_{v_\ell}(t-1,v),
& \quad t\in [1,T+1], \\
-2\hat h(t,v),& \quad t \in [0,1),
\end{aligned} \right.
\end{align}
where
$\hat V(t) \equiv \hat V_{k,\ell}(t):= V(1-t)$ is defined from the solution 
$V(s), s\in [0,1]$ of the equation 
$$
\partial_s V = L_{\rho_0}V -
\partial_{v_i}f^{i,(2)}(\rho_0), \quad s\in (0,1],
$$ 
with
initial value $V(0)=\partial_{v_k}\partial_{v_\ell} \rho_0$, and 
$\hat h(t,v) := - L_{\rho_0} \hat V(t,v), t\in [0,1]$.
Then, $\hat V(t)$ satisfies the equation
$$
\partial_t \hat V = - L_{\rho_0} \hat V + \partial_{v_i}  f^{i,(2)}(\rho_0)
= L_{\rho_0} \hat V + 2 \hat h + \partial_{v_i} f^{i,(2)}(\rho_0).
$$

Note that 
$h(t) := \hat h(1-t) = - L_{\rho_0} V(t,v)$ satisfies the equation
$$
\partial_t h(t) = L_{\rho_0} h(t) +L_{\rho_0}\partial_{v_i}f^{i,(2)}(\rho_0)
$$ 
and $h(0)=  - L_{\rho_0} \partial_{v_k}\partial_{v_\ell}\rho_0$.
Thus, by Duhamel's formula similarly as before, we see 
$$
\|h(t)\|_\infty \le 
\|L_{\rho_0}\partial_{v_k}\partial_{v_\ell}\rho_0\|_\infty 
+ \|L_{\rho_0}\partial_{v_i}f^{i,(2)}(\rho_0)\|_\infty < \infty,
$$
since $\rho_0\in C^5(\T^d)$ by the condition \eqref{eq:1.12}.
Thus, $\| \hat g\|_\infty$ has the same bound by $G$ stated above.
This shows \eqref{7.5-CBQ} by applying \eqref{7.5-C} for the 
extended solution.
\end{proof}

Finally for $\partial_{v_k}\partial_t \rho$, using \eqref{7.A}, expanding the
terms on the right-hand side and recalling
$D(\rho), f(\rho) \in C^\infty([0,1])$, we obtain from \eqref{7.4-BQ},
\eqref{7.5-BQ} and \eqref{7.5-CBQ}
\begin{align}  \label{7.6}
\|\partial_{v_k}\partial_t \rho\|_\infty \le CK^{3(1+\si)/\si+2}.
\end{align}
The estimate \eqref{eq:Schauder} follows from
\eqref{7.4-BQ}, \eqref{7.5-BQ}, \eqref{7.6-A}, \eqref{7.5-CBQ}, and \eqref{7.6}.

\section {Comparison theorem}  \label{sec:8}

The following lemma is standard and follows directly from the maximum principle.
Recall that $f(0)>0$ and $f(1)<0$ for $f$ in the PDE \eqref{1.6} and $K>0$.

\begin{lem}  \label{lem:max}
For the solution $\rho(t,v)$ of \eqref{1.6}, if $0<\rho_- \le \rho_0(v)\le \rho_+ <1$,
then we have $\rho_- \wedge \a_-\le \rho(t,v)\le \rho_+\vee \a_+ $ for
all $t\ge 0$, where $\a_- = \min\{ \a\in [0,1]; f(\a)=0\} >0$ and
$\a_+ = \max\{ \a\in [0,1]; f(\a)=0\}<1$.
\end{lem}

\begin{proof}
To show the upper bound, set $w(t,v) := \rho_+\vee \a_+ + \e - \rho(t,v)$
for $\e>0$ small enough such that $\rho_+\vee \a_+ + \e <1$.  Noting 
$w(0,v)>0$, assume that there exist $t_0>0$ and $v_0\in \T^d$ such that
$w(t,v)>0$ for $t\in [0,t_0)$ and $w(t_0,v_0)=0$.  Then, we have
\begin{align}  \label{eq:8.1-P}
\partial_t w(t_0,&v_0) - \sum_{i,j=1}^d
  \partial_{v_i} \big\{ D_{ij}(\rho)\partial_{v_j} w\big\}(t_0,v_0)\\
&  =- K f(\rho(t_0,v_0)) = -K f(\rho_+\vee \a_+ + \e)>0.  \notag
\end{align}
However, we see that $\partial_t w(t_0,v_0) \le 0$ and, since $w(t_0,\cdot)$
takes the minimum value $0$ at $v=v_0$, recalling \eqref{1.D} and 
\eqref{eq:1.9} for $D(\rho)$, similar to \cite{E}, p.\ \!345, p.\ \!390, we have
$$
\sum_{i,j=1}^d
  \partial_{v_i} \big\{ D_{ij}(\rho)\partial_{v_j} w\big\}(t_0,v_0)\ge 0.
$$
This contradicts with \eqref{eq:8.1-P} and therefore we have that
$\rho(t,v) < \rho_+\vee \a_+ + \e$ for all $\e>0$ and $t\ge 0$.
Thus the upper bound is shown.  The proof for the lower bound is similar.
\end{proof}

\section{Non-degeneracy and boundedness of diffusion matrix}
\label{sec:9}

This problem was studied in \cite{SY} in more general setting reversible
under the Gibbs measures.  Nevertheless, we give the proof, in particular
that of the lower bound, since it turns out to be
explicit in the setting reversible under the Bernoulli measures.
Recall the diffusion matrix $D(\rho)$ defined by \eqref{1.5}
from $\widehat c(\rho)$ and $\chi(\rho)$ 
determined by \eqref{1.3} and \eqref{1.4}, respectively.

\begin{lem}  \label{lem:1.1}
The bound \eqref{eq:1.9} holds for $D(\rho)$, that is,
\begin{align*}
c_* |\th|^2 \le (\th,D(\rho)\th) \le c^* |\th|^2, \quad \th \in \R^d, \; \rho\in [0,1],
\end{align*}
holds with $c_*, c^*>0$ defined by
$$
0< c_* := \min_{\eta\in\mathcal{X}, x\in \Z^d: |x|=1} c_{0,x}(\eta) 
\le c^* := \max_{\eta\in\mathcal{X}, x\in \Z^d: |x|=1} c_{0,x}(\eta) < \infty.
$$
\end{lem}

\begin{proof}
The upper bound $(\th,D(\rho)\th)\le c^* |\th|^2$ 
is easily obtained by taking $F=0$ in \eqref{1.3}.
Indeed, recalling $D(\rho)= \widehat c(\rho)/2\chi(\rho)$ in \eqref{1.5}
and the variational formula \eqref{1.3} for $\widehat c(\rho)$, we have
\begin{align*}
(\th,D(\rho)\th) 
& \le (\th,\widehat c(\rho;0)\th) /2\chi(\rho)  \\
& = \sum_{|x|=1} \lan c_{0,x}(\th,x(\eta_x-\eta_0))^2\ran_\rho /4\chi(\rho)  \\
& \le c^* \sum_{i=1}^d \th_i^2 \lan (\eta_{e_i}-\eta_0)^2\ran_\rho /2\chi(\rho) 
= c^* |\th|^2.
\end{align*}

To show the lower bound, noting $\pi_{0,x}(\eta_x-\eta_0) = -2  (\eta_x-\eta_0)$
for $x: |x|=1$, we have
\begin{align*}
\Big\lan (\eta_x-\eta_0) \pi_{0,x}\Big(\sum_y \t_y \, \th \cdot F\Big) \Big\ran_\rho
& = -2 \sum_y \lan (\eta_x-\eta_0) \t_y \, \th \cdot F \ran_\rho  \\
& = -2 \sum_y \lan (\eta_{x+y}-\eta_y) \, \th \cdot F \ran_\rho =0.
\end{align*}
Therefore,
\begin{align*}
|\th\cdot x| \lan (\eta_x-\eta_0)^2\ran_\rho
& = \bigg| \Big\lan (\eta_x-\eta_0) \Big( \th \, \cdot x (\eta_x-\eta_0)
  - \pi_{0,x}\big(\sum_y \t_y \,\th \cdot F\big) \Big) \Big\ran_\rho  \bigg| \\
& \le \lan (\eta_x-\eta_0)^2\ran_\rho^{1/2} \;
\Big\lan \frac{c_{0,x}}{c_*}\Big( \th\cdot x (\eta_x-\eta_0)
  - \pi_{0,x}\big(\sum_y \t_y \, \th \cdot F\big) \Big)^2 \Big\ran_\rho^{1/2}.
\end{align*}
Then, taking the square and summing in $x: |x|=1$ in both sides, we obtain
\begin{align*}
(2\chi(\rho))^2 \sum_{|x|=1} |\th\cdot x|^2
\le  \frac{4\chi(\rho)}{c_*} (\th,\widehat c(\rho;F)\th),
\end{align*}
for every $F\in \mathcal{F}_0^d$.  This implies the lower bound 
$(\th,D(\rho)\th)\ge c_* |\th|^2$.
\end{proof}

\section{Balance condition---Examples of rates $\{c_{x,y}\}$ and $\{c_x\}$}
\label{Section9}

Consider the jump rates $c_{x,y}(\eta)$ of the Kawasaki  part satisfying the condition
\begin{equation}  \label{eq:3-1-Q}
c_{x,y}(\eta)= c_{x,y}(\check\eta),
\end{equation}
where $\check\eta$ is the configuration defined from $\eta$ by
$(\check\eta)_x = 1-\eta_x$ for every $x$.
Then, the corresponding diffusion matrix is symmetric in $\rho$ under the
reflection at $1/2$, 
i.e.\ $D(\rho) = D(1-\rho)$
for $\rho\in [0,1]$.  
In fact, under the condition \eqref{eq:3-1-Q}, one can rewrite 
$\left(\th,\widehat c(\rho;F)\th\right)$ in the variational 
formula \eqref{1.3} for $\widehat c(\rho)$ as follows:
\begin{align*}
\left(\th,\widehat c(\rho;F)\th\right) 
& = \frac12   \sum_{|x|=1} \left\lan c_{0,x}(\eta) \bigg(\th,x(\eta_x-\eta_0) - 
    \pi_{0,x}\Big(\sum_{y \in\Z^d} \tau_y F(\eta)\Big)\bigg)^2 \right\ran_\rho  \\
& = \frac12   \sum_{|x|=1} \left\lan c_{0,x}(\check\eta) 
\bigg(\th,x(\check\eta_x-\check\eta_0) - 
    \pi_{0,x}\Big(\sum_{y \in\Z^d} \tau_y \check F(\check\eta) 
    \Big)\bigg)^2 \right\ran_\rho,    
\end{align*}
where $\check F(\eta) := - F(\check\eta)$.
We have used $\eta_x-\eta_0 = -(\check\eta_x-\check\eta_0)$ and 
$F(\eta) = - \check F(\check\eta)$.
However, under the map $\eta\mapsto \check\eta$, the Bernoulli measure
$\nu_\rho$ is transformed to $\nu_{1-\rho}$.  Therefore, we see
$\left(\th,\widehat c(\rho;F)\th\right) =
\left(\th,\widehat c(1-\rho; \check F)\th\right)$.
This shows $\widehat c(\rho)=\widehat c(1-\rho)$ and
$D(\rho) = D(1-\rho)$, since $\chi(\rho) =\rho(1-\rho)$ is symmetric.

Note that the rate satisfying \eqref{eq:3-1-Q} can be constructed 
as $\check c_{x,y}(\eta) = (c_{x,y}(\eta) + c_{x,y}(\check\eta))/2$
from any rate $c_{x,y}(\eta)$ which satisfies the non-degeneracy, locality, spatial
homogeneity and the detailed balance condition; recall (1)--(3) in Section
\ref{Section 1.1}.  The new rate $\check c_{x,y}(\eta)$ satisfies these three
conditions and \eqref{eq:3-1-Q}.

On the other hand, one can take the flip rates $c_x(\eta)$ of the Glauber part in 
such a manner that the corresponding reaction term has the form
\begin{align}  \label{eq:3-2-Q}
f(\rho) = - c(\rho-\rho_-)(\rho-1/2)(\rho-\rho_+), \quad \rho\in [0,1],
\end{align}
for some $c>0$ and $0<\rho_-< 1/2 < \rho_+<1$ such that 
$\rho_++\rho_-=1$; see Example 4.1-(2) of \cite{F18}.
This $f$ is anti-symmetric in $\rho$ under the 
reflection at $1/2$, i.e.\  $f(1-\rho)=-f(\rho)$, $\rho\in [0,1]$.

Therefore, by the change of variable $\rho=1-u$, we have
\begin{align*}
\int_{\rho_-}^{\rho_+} f(\rho) D(\rho) d\rho
& = \int_{\rho_-}^{\rho_+} f(1-u) D(1-u) du \\
& = -\int_{\rho_-}^{\rho_+} f(u) D(u) du.
\end{align*}
Thus the integral should vanish, and the balance condition in Section \ref{sec:1.4}
is satisfied in this 
example, i.e., $\{c_{x,y}\}$ satisfying \eqref{eq:3-1-Q} and $\{c_{x}\}$ 
corresponding to \eqref{eq:3-2-Q}.

One can slightly generalize $f$ and $\{c_x\}$ in this example.
Let us assume that the pair $(D(\rho),f(\rho))$, determined from 
certain rates $\{c_{x,y}\}$ and $\{c_x\}$,
satisfies that $f$ has three zeros $0<\rho_-<\rho_*<\rho_+<1$
with stable $\rho_\pm$ and unstable $\rho_*$, and the balance condition
$\int_{\rho_-}^{\rho_+} f(\rho) D(\rho) d\rho=0$ holds. The example constructed above
satisfies this condition. Then, let us keep $D(\rho)$ and $\{c_{x,y}\}$ as they are,
and try to find the generalization of $f(\rho)$ in the form:
$$
f(\rho;\a) = f(\rho) + \sum_{k=0}^N \a_k \rho^k, \quad \rho\in [0,1],
$$
where  $\a = (\a_k)_{k=0}^N\in \R^{N+1}$ is a parameter.

The goal is to find $\a\in \R^{N+1}$ 
such that the following two conditions are satisfied:
\begin{align}  \label{eq:3-1-A}
& f(\rho_+;\a) =f(\rho_-;\a) = 0,  \\
&\int_{\rho_-}^{\rho_+}  f(\rho;\a) D(\rho) d\rho=0.
\label{eq:3-2-A}
\end{align}

These two conditions can be rewritten as
\begin{align}\label{eq:3-5-A}
& \sum_{k=0}^N \a_k(\rho_+)^k = \sum_{k=0}^N \a_k(\rho_-)^k= 0,\\
& \sum_{k=0}^N \a_k\nu_{ij,k} =0
\label{eq:3-6-A}
\end{align}
for all $1\le i,j \le d$, where
\begin{align*}
 \nu_{ij,k} := \int_{\rho_-}^{\rho_+} \rho^k D_{ij}(\rho)d\rho.
\end{align*}
Since $\nu_{ij;k}$ is symmetric in $i, j$, the equation
\eqref{eq:3-6-A} provides $(d^2+1)/2$ conditions.

The equations \eqref{eq:3-5-A}
plus \eqref{eq:3-6-A} are $n:=(d^2+1)/2+2$ linear 
homogeneous equations for $(N+1)$ variables $\{\a_k\}_{k=0}^{N}$.  
We take $N$  such that $N\ge  n$.
Then, as is well-known, for an $n\times (N+1)$  matrix $A$ of rank $r$,
the solutions of the linear homogeneous equation $A\a = 0$ makes an 
$(N+1)-r$ dimensional linear space.  Since $r\le n$
and $N\ge n$ imply $(N+1)-r\ge 1$, we have at least one
non-trivial solution $\{\a_k\}_{k=0}^{N}$ and one can make
$\|\a\|_\infty = \sup_k |\a_k|$ to be sufficiently small.

If $\|\a\|_\infty$ is sufficiently small, $f(\rho;\a)$ has only three zeros including
$\rho_\pm$ and one can find the Glauber flip rates $\{c_x\}$ which correspond
to this $f(\rho;\a)$.  Indeed, to construct such rates, we may consider a linear 
combination of $\prod_{i=1}^k\eta_{x_i}$, $0\le k \le N-1$ taking different
points $\{x_1,\ldots,x_{N-1},0\}$ in $\Z^d$ as in Example 4.1-(2) of \cite{F18}.

\appendix

\section*{Appendix A \hskip 3mm $r$-Markov property of $\psi_t$}
\label{Section:A}

\renewcommand{\thethm}{A.\arabic{thm}}
\renewcommand{\thelem}{A.\arabic{lem}}

\setcounter{equation}{0}
\renewcommand{\theequation}{A.\arabic{equation}}
\makeatletter
\@addtoreset{equation}{section}
\makeatother

In this appendix, we show the $r$-Markov property of $\psi_t$ (or $P^{\psi_t}$)
defined by \eqref{eq:2psit}.  It was used in the proof of Lemma \ref{Lemma 3.1}
in Section \ref{sec:4.1-C}.  In the next Lemma \ref{lem:3.2-A}, we denote 
$\psi_t$ by $\psi$ for simplicity by omitting $t$ 
and $\frac1N$ in front of the sum involving $F$.
For $\La \subset \T_N^d$ and $r\in \N$, set $\partial_r\La = \{ y \in 
\T_N^d \setminus \La; \text{dist}(y,\La) \le r\}$, $\bar\La = \La \cup 
\partial_r\La$, where $\text{dist}(y,\La) = \min\{|x-y|; x\in \La\}$ is 
defined in the $\ell^\infty$ and modulo $N$-sense, i.e.\ $|x|= \max_{1\le i \le d}
\min_{n\in \Z}|x_i-nN|$ for $x=(x_i)_{i=1}^d \in \T_N^d$.
(We actually take $\La=\{x\}$ in 
\eqref{m-etax} in the proof of Lemma \ref{Lemma 3.1}.)
For $\eta\in \{0,1\}^\La$ and $\om\in \{0,1\}^{\partial_r\La}$,
$\eta\cdot\om \in \{0,1\}^{\bar\La}$ is defined by $(\eta\cdot\om)_x
=\eta_x$ for $x \in \La$ and $=\om_x$ for $x\in \partial_r\La$.  Define
\begin{align*}
& H_{\bar\La}(\eta\cdot\om) := \sum_{y\in \bar\La}\la(y/N) (\eta\cdot\om)_y
+ \sum_{y\in \T_N^d :\, \text{supp}\, \t_y F\subset \bar\La}
(\partial\la(y/N), \t_y F (\eta\cdot\om)), \\
& P_\La^\om(\eta) := Z_\om^{-1} e^{H_{\bar\La}(\eta\cdot\om)},
\quad Z_\om := \sum_{\eta\in \{0,1\}^\La} e^{H_{\bar\La}(\eta\cdot\om)}.
\end{align*}

\begin{lem}  \label{lem:3.2-A}
Let $F \in {\mathcal F}_0^d$ and $r=r(F)$.  Then, $\psi$ has an $r$-Markov
property, that is, for $\La \subset \T_N^d$ and $\eta\in \{0,1\}^\La$,
\begin{equation}  \label{r-Markov}
P^\psi([\eta]|{\mathcal F}_{\La^c})(\om) = P_\La^{\bar\om}(\eta),
\quad P^\psi\rm{\text{-}a.s.}\ \om
\end{equation}
holds, where $[\eta] = \{\zeta\in \{0,1\}^{\T_N^d}; \zeta|_\La = \eta\}$ and
$\bar\om = \om|_{\partial_r\La}$.  Recall $P^\psi = \psi \,d\nu^N$.
\end{lem}

\begin{proof}
The lemma is well-known and the proof is elementary, but we give it
for the sake of completeness.  Since the right-hand side of \eqref{r-Markov} is
${\mathcal F}_{\La^c}$-measurable in $\om$, it is enough to show
\begin{equation}  \label{r-Markov-fg}
E^\psi[f(\eta)g(\om)] = E^\psi\big[g(\om)E^{P_\La^{\bar\om}}[f] \big]
\end{equation}
for every ${\mathcal F}_{\La}$-measurable function $f$ and
${\mathcal F}_{\La^c}$-measurable function $g$.
The left-hand side of \eqref{r-Markov-fg} is given by
\begin{equation}  \label{eq:3.13}
Z^{-1} \sum_{\eta\in \{0,1\}^\La, \om \in \{0,1\}^{\La^c}}
f(\eta)g(\om) e^{H_N(\eta\cdot\om)},
\end{equation}
while the right-hand side is rewritten as
\begin{equation}  \label{eq:3.14}
Z^{-1} \sum_{\eta\in \{0,1\}^\La, \om \in \{0,1\}^{\La^c}}
g(\om) e^{H_N(\eta\cdot\om)} Z_{\bar\om}^{-1} \sum_{\zeta\in \{0,1\}^\La}
f(\zeta) e^{H_{\bar\La}(\zeta\cdot\bar\om)},
\end{equation}
where $Z$ is a normalization constant (including the factor $\nu^N(\eta\cdot\om)
\equiv 2^{-N}$), and
\begin{align*}
H_N(\eta\cdot\om) :=&
  \sum_{y\in \T_N^d}\la(y/N) (\eta\cdot\om)_y
+ \sum_{y\in \T_N^d} (\partial\la(y/N), \t_y F (\eta\cdot\om)) \\
= & H_{\bar\La}(\eta\cdot\bar\om) + \widetilde H_{\bar\La^c}(\eta\cdot\om),
\end{align*}
with $H_{\bar\La}(\eta\cdot\bar\om)$ defined above and
$$
\widetilde H_{\bar\La^c}(\eta\cdot\om)
= \sum_{y\in \bar\La^c}\la(y/N) (\eta\cdot\om)_y
+ \sum_{y \in \T_N^d:\, \text{supp}\, \t_y F\cap \bar\La^c \not= \emptyset}
(\partial\la(y/N), \t_y F (\eta\cdot\om)).
$$
However, $(\eta\cdot\om)_y =\om_y$ for $y \in \bar\La^c$ in the first sum and 
$\t_y F (\eta\cdot\om) = \t_y F (\om)$ for $y\in\T_N^d$ such that
$\text{supp}\t_y F\cap \bar\La^c 
\not= \emptyset$ in the second sum, since the radius of the support of $F$ is $r$ so that 
$\text{supp}\, \t_y F\cap\La = \emptyset$.
Thus, $\widetilde H_{\bar\La^c}(\eta\cdot\om) \equiv \widetilde H_{\bar\La^c}(\om)$
is a function of $\om$ only.  Therefore, in \eqref{eq:3.14},
$$
\sum_{\eta\in \{0,1\}^\La}e^{H_N(\eta\cdot\om)} 
= Z_{\bar\om} \cdot e^{\widetilde H_{\bar\La^c}(\om)}.
$$
Accordingly, the right-hand side is equal to
$$
Z^{-1} \sum_{\om \in \{0,1\}^{\La^c}} g(\om) \sum_{\zeta\in \{0,1\}^\La}
f(\zeta) e^{H_{\bar\La}(\zeta\cdot\bar\om)} e^{\widetilde H_{\bar\La^c}(\om)},
$$
which coincides with \eqref{eq:3.13}, that is the left-hand side
of \eqref{r-Markov-fg}.  This completes the proof of the lemma.
\end{proof}

\section*{Appendix B \hskip 3mm Proof of Lemma \ref{Theorem 3.3}
(large deviation type upper bound)} 
\label{sec:3.8}

\setcounter{thm}{0}
\renewcommand{\thelem}{B.\arabic{lem}}

\setcounter{equation}{0}
\renewcommand{\theequation}{B.\arabic{equation}}

Lemma \ref{Theorem 3.3} gives a large deviation type upper 
bound under $\psi_{\la(\cdot),F}^N\, d\nu^N$ defined by \eqref{B.8}
and provides an appropriate error estimate.

To prove this lemma,
we first recall some known facts on thermodynamic functions; see
\cite{3} and Appendix B of \cite{FUY} with $J\equiv 0$, $H_\La\equiv 0$
and $R_1=0$.  For a finite domain  $\La \Subset \Z^d$ ($\La \not=\emptyset$),
a chemical potential  $\la \in \R$ and particle
number  $k : 0 \le k \le |\La|$, define grand canonical and canonical partition
functions by
\begin{align}
&  Z_{\La,\la} = \sum_{\eta_\La \in \mathcal{X}_\La}
      \exp \bigg\{ \la \sum_{x\in\La} \eta_x  \bigg\}
      =e^{ p(\la)|\La|},      
         \label{B.1}  \\
&  Z_{\La,k} = \sum_{\eta_\La \in \mathcal{X}_\La :
             \sum_{x\in\La} \eta_x  = k } 1 = 
             \begin{pmatrix} |\La| \\k \end{pmatrix},   \label{B.2} 
\end{align}
respectively, where  $\mathcal{X}_\La = \{ 0,1\}^\La$ and recall \eqref{2.4} 
for $p(\la)$.  Then,
\begin{align}
\log Z_{\La,k} =|\La| q(k/|\La|) + \frac12 \log \frac{|\La|}{2\pi k(|\La|-k)} +o(1),
                        \label{B.4} 
\end{align}
as $|\La|, k \to \infty$; recall \eqref{2.4} for $q(u)$.  
(Note, for $k=0$ or $|\La|$, $Z_{\La,0}=Z_{\La,|\La|}=1$
and $q(0)=q(1)=0$.)  Indeed, by Stirling's formula $n! \sim n^n e^{-n} \sqrt{2\pi n}$
($a_n\sim b_n$ means $a_n/b_n\to 1$ as $n\to\infty$), we have
\begin{align*}
\log Z_{\La,k}  & = \log \big(|\La|!/ k! (|\La|-k)! \big)  \\
&= |\La| \log |\La|- |\La| + \tfrac12 \log 2\pi |\La| 
 - \big\{ k\log k - k + \tfrac12 \log 2\pi k\big\}  \\
& \quad - \big\{ (|\La|-k) \log (|\La|-k) - (|\La|-k) +\tfrac12 \log 2\pi(|\La|-k) \big\} + o(1) \\
&= |\La| \log |\La| - k \log k -(|\La|-k) \log(|\La|-k) +
\tfrac12 \log \frac{|\La|}{2\pi k(|\La|-k)} + o(1) \\
&= |\La| \Big\{ -\frac{k}{|\La|} \log \frac{k}{|\La|} - \big(1-\frac{k}{|\La|}\big) \log (1-\frac{k}{|\La|}) \Big\}
+ \tfrac12 \log \frac{|\La|}{2\pi k(|\La|-k)} + o(1),
\end{align*}
as $|\La|, k \to \infty$, and this shows \eqref{B.4}.

The functions  $p(\la)$  and  $- q(\rho)$  are continuous and convex, and
satisfy
\begin{align}
&  p(\la) = \sup_{\rho \in [0,1]}  \{ q(\rho) + \la\rho\}, \label{B.5} \\
&  q(\rho) = \inf_{\la \in \R}  \{ p(\la) - \la\rho\}. \label{B.6}
\end{align}
For each  $\rho \in [0,1]$, there exists a unique $\la = \bar{\la}(\rho)
\in \bar{\R} := \R \cup \{\pm \infty\}$  such that
\begin{equation}
  p(\la) =  q(\rho) + \la\rho.   \label{B.7}
\end{equation}

Before proceeding to the proof of Lemma \ref{Theorem 3.3}, 
we prepare a similar statement for finite 
volume Bernoulli measures with  constant chemical potentials.
For $\La \Subset \Z^d$ ($\La \not=\emptyset$),  $\la \in \R$,  the finite volume 
(grand canonical) Bernoulli measure
$\nu_{\La,\la}$  on  $\La$, which is a probability measure on $\mathcal{X}_\La$,  is 
defined by
\begin{equation}
\nu_{\La,\la} (\eta) =
  Z_{\La,\la}^{-1}  
      \exp \bigg\{  \la \sum_{x\in\La} \eta_x \bigg\}
      \equiv \bar\nu_\la^{\otimes\La}(\eta),
              \quad  \eta \in \mathcal{X}_\La.    \label{B.9}
\end{equation}

\begin{lem} \label{Proposition B.1} 
 For every  $\la\in \R$ and $G(\rho) =G_1+\frac1K G_2\in C^1([0,1])$,
$$
\ell_*^{-d} \log E^{\nu_{\La(\ell),\la}} [ \exp \{ \ell_*^d 
G(\bar{\eta}^\ell_0)\}]  
= \sup_{\rho \in [0,1]}
    \{ G(\rho) -I(\rho;\la) \} + Q_{\ell}(\la,G),
$$
recall $\La(\ell) = \La_{\ell,0}$ in \eqref{eq:3.1-P} and
\eqref{eq:3.saeta} for $\bar\eta_0^\ell$.
We have an estimate for the error term $Q_{\ell}(\la,G)$,
which is uniform in $\frac1K\in (0,1]$:
\begin{equation}  \label{eq:Q-ell-la}
|Q_{\ell}(\la,G)| \le C \ell^{-d}  (\log \ell + |\la|+ \|G'\|_\infty),
\end{equation}
where $G'\equiv \partial_\rho G$.
\end{lem}

\begin{proof}  By \eqref{B.1}, setting $\bar{\eta}_\La = \frac1{|\La|}  \sum_{x\in \La} \eta_x$ for $\La \subset \Z^d$ in general, we have
\begin{align*}
I:=& \frac1{|\La|} \log E^{\nu_{\La,\la}}
     [ \exp \{|\La|G(\bar{\eta}_\La)\}]   
       \\  
=& \frac1{|\La|} \log  \bigg[ \sum_{\eta\in \mathcal{X}_\La}
      \exp \bigg\{ \la \sum_{x\in\La} \eta_x 
                 + |\La|G(\bar{\eta}_\La)
                       \bigg\} \bigg] - \frac1{|\La|} \log  Z_{\La,\la}  
       \\  
=& \frac1{|\La|} \log  \bigg[ \sum_{m=0}^{|\La|}
       Z_{\La,m}
            \exp \bigg\{ \la m  + |\La|G(m/|\La|) \bigg\}\bigg]
                        - p(\la),
\end{align*}
where $Z_{\La,\la}$ and $Z_{\La,m}$ are defined by \eqref{B.1} 
and \eqref{B.2}, respectively.
Then we apply \eqref{B.4} for $Z_{\La,m}$.  First note that
\begin{align} \label{eq:3.Zlam}
\frac4{|\La|} \le \frac{|\La|}{m(|\La|-m)}  \le 2
\end{align}
for $1\le m \le |\La|-1$, $|\La| \ge 2$.  Therefore, by \eqref{B.4},
$Z_{\La,m} \le e^{|\La| q(m/|\La|) + \frac12 \log \frac2{2\pi}+C_1}$
for $0\le m \le |\La|, |\La|\ge 2$ (including $m=0, |\La|$).  Thus, denoting
$a^* :=  \sup_{\rho \in [0,1]} \{ G(\rho) -I(\rho;\la) \} +p(\la)$ and estimating 
all terms in the sum in $m$ by using $a^*$,
we have an upper bound of $I$:
\begin{align*}
I & \le \frac1{|\La|} \log \Big[ (|\La|+1) \exp \{ |\La| a^* + C_2\}
\Big] -p(\la) \\
& = \frac1{|\La|}  \Big[ \log(|\La|+1) + |\La| a^* + C_2\Big] -p(\la) \\
& \le a^*-p(\la) + \frac{C}{|\La|}  \log |\La|.
\end{align*}
(Note that the constant $C$ is uniform in $G$; indeed, $C=2$ is enough for large $|\La|$.)

Next, we show the lower bound for $I$.  The supremum for $a^*$ is attained 
at some $\rho_*\in [0,1]$.  (Note that $q$ is concave and $q(0)=q(1)=0$, and
$q'(\rho) = \log (1-\rho)/\rho$ so that $q'(0)= +\infty, q'(1)= -\infty$.)
Since $G\in C^1([0,1])$, choosing $m$ such that $\frac{m}{|\La|}$ is close to
$\rho_*$ in the sense that $|\frac{m}{|\La|}-\rho_*|\le \frac1{|\La|}$ and 
$\frac{m}{|\La|}\le \rho_*$ (if $\rho_*\le \frac12$) or 
$\frac{m}{|\La|}\ge \rho_*$ (if $\rho_*> \frac12$), we have
$$
\big| \{ G(\rho_*)-I(\rho_*;\la) \} - \{ G(m/|\La|)-I(m/|\La|;\la) \}\big| \le 
\frac{\|G'\|_\infty + |\la|}{|\La|}
+|q(\rho_*)-q(m/|\La|)|.
$$
Here, $|q(\rho_*)-q(m/|\La|)|$ is estimated as
$$
|q(\rho_*)-q(m/|\La|)| \le \frac{C_3}{|\La|}
$$
for some $C_3=C_3(\e)>0$ if $0<\e\le \rho_*\le 1-\e <1$.
If $0<\rho_*\le \e$, since $\frac{m}{|\La|} \le \rho_*$, we have 
$|q'(\frac{m}{|\La|})| \ge |q'(\rho)|$ for every $\rho \in [m/|\La|, \rho_*]$
and therefore
$$
|q(\rho_*)-q(m/|\La|)| \le \frac{1}{|\La|} \Big|q'\big(\frac{m}{|\La|}\big) \Big| 
= \frac{1}{|\La|} \log \frac{|\La|- m}{m}
\le \frac{1}{|\La|}  \log |\La|.
$$
The case $1>\rho_*\ge 1-\e$ is similar.

Then, taking only the term given by this $m$ in the sum of $I$ and using
the lower bound in \eqref{eq:3.Zlam} for $Z_{\La,m}$, we have
\begin{align*}
I & \ge \frac1{|\La|} \log \Big[ \exp \Big\{ |\La| \big(a^* -
\frac{\|G'\|_\infty+ |\la|+ C_3 + \log |\La|}{|\La|}\big)
+ \frac12 \log \frac4{2\pi|\La|} \Big\}\Big] -p(\la) \\
& = a^* - \frac{\|G'\|_\infty+ |\la|+C_3+\log |\La|}{|\La|} - \frac{\log |\La|-\log \frac2\pi}{2|\La|} -p(\la) \\
& \ge a^*-p(\la) - \frac{C}{|\La|}  (\log |\La|  + |\la|+ \|G'\|_\infty).
\end{align*}
We therefore have the conclusion with the estimate \eqref{eq:Q-ell-la}
by taking $\La = \La(\ell)$.
\end{proof}

We are now ready to give the proof of Lemma \ref{Theorem 3.3}.
As we noted, this theorem is an extension of Theorem 3.3 of \cite{FUY}
and gives a further error estimate.  Because the definition of $\widetilde{G}$
is different, the proof is modified by applying H\"older's inequality.

\begin{proof}[Proof of Lemma \ref{Theorem 3.3}]
To apply Lemma \ref{Proposition B.1}  by localizing in space, finally to
obtain the estimate in terms of a spatial integral and also to match
with the scale $\ell$ in $\widetilde{G}$, we divide $\T_N^d$ into disjoint 
boxes $\{\La_{\ell,a}\}_a$ with 
side length $\ell_* =2\ell+1$ and centered at $a\in \ell_*\T_{N/\ell_*}^d 
= \{ \ell_*, 2\ell_*,\ldots, N\}^d$, assuming $N/\ell_* \in \N$ for simplicity 
(if not, we may make the side lengths of some of the boxes smaller
or larger by $1$).

Since $x\in \T_N^d$ is uniquely decomposed as $x=y+a$ with
$y\in \La(\ell) = \La_{\ell,0}$ and $a\in \ell_*\T_{N/\ell_*}^d$
(we write $a(x)$ for $a$)
according to this division, writing the sum and the product in 
$a \in \ell_*\T_{N/\ell_*}^d$ by $\sum_a^*$ and $\prod_a^*$,
respectively, we have
\begin{align*}
I :=&  E^{\nu_{\la(\cdot),F}^N} [ \exp \widetilde{G}(\eta) ] 
       \\  
= & Z_{\la(\cdot),F,N}^{-1}  E^{\nu^N}\bigg[
    \exp\bigg\{ \sum_{y\in \La(\ell)}\sum_a\!{}^{^*}
  \big[ G(y+a)/N,\bar{\eta}_{y+a}^{\ell}) + \la(a/N) \bar{\eta}_{y+a}^{\ell} \big]
                  + R(\eta) \bigg\} \bigg]  \\
= & Z_{\la(\cdot),F,N}^{-1}  E^{\nu^N}\Bigg[
 \prod_{y\in \La(\ell)}
    \exp\bigg\{ \sum_a\!{}^{^*}
  \big[ G((y+a)/N,\bar{\eta}_{y+a}^{\ell}) + \la(a/N) \bar{\eta}_{y+a}^{\ell} \big]
           \bigg\}\cdot e^{R(\eta)} \Bigg]  \\
\le & e^{\|R\|_\infty}   Z_{\la(\cdot),F,N}^{-1}  \prod_{y\in \La(\ell)}
E^{\nu^N}\Bigg[
    \exp\bigg\{ \ell_*^d \sum_a\!{}^{^*}
  \big[ G((y+a)/N,\bar{\eta}_{y+a}^{\ell}) + \la(a/N) \bar{\eta}_{y+a}^{\ell} \big]
                  \bigg\} \Bigg]^{1/\ell_*^d},
\end{align*}                 
where $R(\eta) = R_1+R_2$ and
\begin{align*}
& R_1 = \frac1{N} \sum_{x\in\T_N^d} \left(\partial\la(x/N), \tau_x F(\eta)
         \right), \\
& R_2 = \sum_{x\in\T_N^d}\big( \la(x/N)\eta_x - \la(a(x)/N)\bar\eta_x^\ell\big).
\end{align*}
For the last line, we have first estimated $e^{R}\le e^{\| R\|_\infty}$ and then
applied H\"older's inequality under $\nu^N$.
Two error terms   $R_1$ and $R_2$ are estimated as follows:
\begin{align}\label{B.12-R} 
 |R_1(\eta)|  \le  N^{d-1} \|\partial\la\|_\infty \| F\|_\infty,
 \quad
 |R_2(\eta)|  \le  \ell_*N^{d-1} \|\partial\la\|_\infty.
\end{align}
In fact, for $R_2$, note that the sum of the second term of $R_2$ can be rewritten as
\begin{align*}
\sum_{x\in\T_N^d} \la(a(x)/N) \ell_*^{-d} \sum_{z\in \La(\ell)} \eta_{x+z}
= \sum_{x'\in\T_N^d} \eta_{x'}
\ell_*^{-d} \sum_{z\in \La(\ell)} \la(a(x'-z)/N).
\end{align*}

Since $\bar{\eta}_{y+a}^{\ell}$ are functions of $\eta\big|_{\La_{\ell,y+a}}$, they
are independent for different $a$ under $\nu^N$.  Therefore, the last 
expectation in the bound for $I$ is factorized and 
is rewritten as 
\begin{align*}
& E^{\nu^N}\Bigg[
    \exp\bigg\{ \ell_*^d \sum_a\!{}^{^*}
  \big[ G((y+a)/N,\bar{\eta}_{y+a}^{\ell}) + \la(a/N) \bar{\eta}_{y+a}^{\ell} \big]
                  \bigg\} \Bigg]          \\
& = \prod_a\!{}^{^*}           
    E^{\nu^N}\Bigg[
    \exp\bigg\{ \ell_*^d 
    \big[ G((y+a)/N,\bar{\eta}_{y+a}^{\ell}) + \la(a/N) \bar{\eta}_{y+a}^{\ell} \big]
                  \bigg\} \Bigg]                \\
& = \prod_a\!{}^{^*} \big(Z_{\La_{\ell, y+a},\la(a/N)}\cdot 2^{-\ell_*^d} \big)
E^{\nu_{\La_{\ell, y+a},\la(a/N)}} \bigg[\exp\big\{ \ell_*^d
  G((y+a)/N,\bar{\eta}_{y+a}^{\ell}) \big\}\bigg].
\end{align*}
In the second line, $\nu^N$ can be replaced by 
$\bar\nu_{1/2}^{\otimes\La_{\ell,y+a}}$
and, recalling \eqref{B.9}, we obtain the third line.
Thus, noting the shift-invariance of $Z_{\La,\la}$ and $\nu_{\La,\la}$ in $\La$,
and $\prod_y \prod_{a}^* 
= \prod_x$, we have
\begin{align*}
I \le e^{\|R\|_\infty}   Z_{\la(\cdot),F,N}^{-1} 
&  \prod_{a}\!{}^{^*} 
Z_{\La(\ell),\la(a/N)} \cdot 2^{-\ell_*^d}
\times \prod_{x\in \T_N^d}
E^{\nu_{\La(\ell),\la(a(x)/N)}} \bigg[\exp\big\{ \ell_*^d
  G(x/N,\bar{\eta}_0^{\ell}) \big\}\bigg]^{1/\ell_*^d},
\end{align*}
and therefore, 
\begin{align}  \label{eq:3.70}
N^{-d}\log I \le & N^{-d} \|R\|_\infty
+ N^{-d} \log   \Big( Z_{\la(\cdot),F,N}^{-1}  \prod_{a}\!{}^{^*} 
Z_{\La(\ell),\la(a/N)} \cdot 2^{-\ell_*^d} \Big) \\
& + N^{-d} \sum_{x\in \T_N^d} \ell_*^{-d} \log
E^{\nu_{\La(\ell),\la(a(x)/N)}} \bigg[\exp\big\{ \ell_*^d
  G(x/N,\bar{\eta}_0^{\ell}) \big\}\bigg].  \notag
\end{align}

For the partition functions, recalling that
\begin{align*}
& Z_{\La,\la} \cdot 2^{-|\La|}= E^{\bar\nu_{1/2}^{\otimes \La}}
\Big[\exp\Big\{ \la \sum_{x\in \La} \eta_x\Big\}\Big], \quad \la\in \R,\\
& Z_{\la(\cdot),F,N} = E^{\nu^N}\Big[\exp\Big\{ \sum_{x\in \T_N^d} 
\la(x/N)\eta_x+R_1(\eta) \Big\} \Big],
\end{align*}
and changing $\La(\ell)$ to $\La_{\ell,a}$ in $Z_{\La(\ell),\la(a/N)}$, we obtain
\begin{align*}
\prod_a\!{}^{^*}
Z_{\La_{\ell,a},\la(a/N)} \cdot 2^{-\ell_*^d}
= Z_{\widetilde\la(\cdot), 0, N},
\end{align*}
where $\widetilde\la(v)= \sum_a^* \la(a/N) 1_{\La_{\ell,a}}(Nv)$
is a step function on $\T^d$ and $F=0$ on the right-hand side.
To replace the step function $\widetilde\la(\cdot)$ with the original function 
$\la(\cdot)$, note that
\begin{align*}
\frac{ Z_{\widetilde\la(\cdot), 0, N}}{ Z_{\la(\cdot), 0, N}}
& = \frac{E^{\nu_N}
[\exp\{ \sum_{x\in \T_N^d} \widetilde\la(x/N)\eta_x\}]}
{E^{\nu_N}
[\exp\{ \sum_{x\in \T_N^d} \la(x/N) \eta_x\}]} \\
& = \frac{E^{\nu_N}
[\exp\{ \sum_{x\in \T_N^d} \la(x/N)\eta_x
+ \sum_{x\in \T_N^d} (\widetilde\la(x/N)-\la(x/N))
\eta_x  \}]}
{E^{\nu_N}
[\exp\{ \sum_{x\in \T_N^d} \la(x/N) \eta_x\}]} \\
&  \le e^{\ell N^{d-1}\|\partial\la\|_\infty}.
\end{align*}
To compare two partition functions with and without $F$, we have
\begin{align*}
\frac{ Z_{\la(\cdot), 0, N}}{ Z_{\la(\cdot), F, N} } \le e^{\| R_1\|_\infty}.
\end{align*}
Summarizing these and noting \eqref{B.12-R}, we have shown for
first two terms in \eqref{eq:3.70} that
\begin{align}  \label{eq:3.E1}
N^{-d}\|R\|_\infty & + 
N^{-d} \log   \Big( Z_{\la(\cdot),F,N}^{-1}  \prod_a\!{}^{^*}
Z_{\La(\ell),\la(a/N)} \cdot 2^{-\ell_*^d} \Big) \\
&\le CN^{-1}\|\partial\la\|_\infty (\ell + \|F\|_\infty).   \notag
\end{align}

For the last term of \eqref{eq:3.70},
one can apply Lemma \ref{Proposition B.1} to see that
\begin{align*}
N^{-d} & \sum_{x\in \T_N^d} \ell_*^{-d} \log
E^{\nu_{\La(\ell),\la(a(x)/N)}} \Big[\exp\big\{ \ell_*^d
  G(x/N,\bar{\eta}_0^{\ell}) \big\}\Big]  \\
& =  N^{-d} \sum_{x\in \T_N^d} \bigg\{ \sup_{\rho \in [0,1]}
    \{ G(x/N, \rho) -I(\rho;\la(a(x)/N)) \} + Q_{\ell}(\la(a(x)/N),G(x/N,\cdot)) \bigg\},
\end{align*}
and
\begin{equation}  \label{eq:3.E2}
|Q_{\ell}(\la(a(x)/N),G(x/N,\cdot))| \le \frac{C}{\ell^d}  (\log \ell + \|\la\|_\infty
+ \|\partial_\rho G_1\|_\infty + \|\partial_\rho G_2\|_\infty).
\end{equation}
Moreover, one can estimate
\begin{align}\label{B.17}  
 N^{-d}   \sum_{x\in\T_N^d}
  \sup_{\rho \in [0,1]} \{ G(x/N,\rho)  -I(\rho; \la(a(x)/N)) \}  
        \le \sup_{\rho: \text{step on }\T^d}
    \mathbb{G}(\rho) + Q_{\ell,N},
\end{align}
where the supremum on the right-hand side is taken over all
step functions $\rho(v)$ on $\T^d$,
\begin{align}  \label{eq:3.72}
& Q_{\ell,N}\equiv Q_{\ell,N}(\la;G) = \frac{2\ell}N \, \Big( \|\partial\la\|_\infty 
 +  \|\partial_v G_1\|_\infty
    + \|\partial_v G_2\|_\infty \Big)
\end{align}
and
$$
\mathbb{G}(\rho) = 
  \int_{\T^d} \{ G(v,\rho(v)) -I(\rho(v); \la(v)) \} \, dv,
$$
for  $\rho \in L^1(\T^d ; [0,1])$.
Indeed, we first replace $G(x/N,\rho)$ in \eqref{B.17}  by $G(a(x)/N,\rho)$
with an error within $\ell N^{-1} (\|\partial_v G_1\|_\infty + 
\|\partial_v G_2\|_\infty)$, and then the supremum is attained at
some $\rho(a(x)/N)$ for each $a=a(x)$.  This defines a step function
$\rho(v) = \sum_a^* \rho(a/N) 1_{\La_{\ell,a}}(Nv)$, and leads to the
functional $\mathbb{G}(\rho)$, but with an discretization also for
$G$ and $\la$ in the variable $v$.  The error of removing these discretization 
in $G$ and $\la$ is estimated by using
\begin{align*}
& \sup_{|v_1-v_2|\le \ell/N} |\la(v_1) - \la(v_2)| \le 
\ell N^{-1}  \|\partial\la\|_\infty, \\ 
& \sup_{|v_1-v_2|\le \ell/N} \sup_{\rho\in [0,1]}
    |G(v_1,\rho) - G(v_2,\rho)| \le \ell N^{-1}  \, \big( \|\partial_v G_1\|_\infty
    + \|\partial_v G_2\|_\infty \big),  \\  
& \bigg|  \int_{\T^d}  p(\la(v))  \, dv 
   - N^{-d} \sum_{x\in \T_N^d} p(\la(a(x)/N)) \bigg| \le \ell N^{-1} 
   \|p'\|_\infty \|\partial\la\|_\infty \le \ell N^{-1} \|\partial\la\|_\infty,
\end{align*}
and obtain \eqref{B.17} with $Q_{\ell,N}$ given by \eqref{eq:3.72}.

Since  $\mathbb{G}(\rho)$  is continuous in
$\rho \in L^1(\T^d ; [0,1])$  and  $C(\T^d ; [0,1])$  is dense
in  $L^1(\T^d ; [0,1])$,  we obtain the conclusion of  Lemma \ref{Theorem 3.3}
from \eqref{eq:3.E1}, \eqref{eq:3.E2}, \eqref{B.17} and \eqref{eq:3.72}.
\end{proof}

Finally, we make a small comment on other application of 
Lemma \ref{Theorem 3.3}.  Theorem \ref{Theorem 1.1} implies 
with the help of Chebyshev's inequality
$$
P\big( |\lan \rho^N(t),\phi\ran -\lan \rho_K(t),\phi\ran |>N^{-p}\big)
\le C \|\phi\|_{H^\a(\T^d)}^2 N^{-\k+2p}, \quad p\in (0,\k/2).
$$

This decay estimate of the probability follows also from Lemma \ref{Theorem 3.3},
Theorem \ref{Corollary 2.1} and the entropy inequality.  Indeed, set
$$
\mathcal{A}_{N,\phi,\la(\cdot),p}
= \big\{ \eta\in \mathcal{X}_N; \,  
|\lan \rho^N,\phi\ran -\lan \bar{\rho}(\la(\cdot)),\phi\ran |>N^{-p}\big\}.
$$
Then, Lemma \ref{Theorem 3.3} shows the moderate deviation type upper bound:
$$
P^{\psi_{\la(\cdot),F_N}^N}(\mathcal{A}_{N,\phi,\la(\cdot),p})
\le C_1 e^{-c N^{d-2p}},
$$
for some $C_1, c>0$, 
taking $F=F_N, \ell=\ell(N) =N^{a_2}$ as in Section \ref{sec:6.2} and
$G(v,\rho) = \phi(v)\{\rho-\bar\rho(\la(v))\}-N^{-p}/2$ in Lemma \ref{Theorem 3.3}.
Therefore, by the entropy inequality and Theorem \ref{Corollary 2.1}:
$H(f_t|\psi_t) \le CN^{d-\k}$, we obtain
\begin{align*}
P^{f_t}(\mathcal{A}_{N,\phi,\la_K(t,\cdot),p})
& \le \frac{\log 2+ H(f_t|\psi_t)}{\log \{1+ 
1/P^{\psi_t}(\mathcal{A}_{N,\phi,\la_K(t,\cdot),p})\}}  \\
& \le \frac{\log 2+ CN^{d-\k}}{\log(1+e^{cN^{d-2p}}/C_1)}
\le C_2 N^{-\k+2p}.
\end{align*}

\section*{Acknowledgements}

The author thanks Chenlin Gu for helpful discussions, especially for his deep 
insights from the quantitative homogenization and its extension to the 
hydrodynamic limit. He also thanks the referee, whose comments were 
quite useful for improving the presentation and providing easier access 
to the results.
This work was supported in part by the 
International Scientists Project of BJNSF, No.\ IS23007.

\end{document}